\documentclass[a4paper,11pt]{article}

\usepackage[T1]{fontenc}
\usepackage{lmodern}

\usepackage{dsfont}

\usepackage[latin1]{inputenc}
\usepackage{amsmath}
\usepackage{amsthm}
\usepackage{amssymb}
\usepackage{mathrsfs}
\usepackage{graphicx}
\usepackage[all]{xy}
\usepackage{hyperref}

\usepackage{xcolor}

\usepackage{makeidx}

\usepackage{stmaryrd}
\usepackage{caption}

\usepackage{abstract} 

\newtheorem{thm}{Theorem}[section]
\newtheorem{cor}[thm]{Corollary}
\newtheorem{claim}[thm]{Claim}
\newtheorem{fact}[thm]{Fact}

\newtheorem{lemma}[thm]{Lemma}
\newtheorem{prop}[thm]{Proposition}

\theoremstyle{definition}
\newtheorem{definition}[thm]{Definition}
\newtheorem{convention}[thm]{Convention}
\newtheorem{ex}[thm]{Example}
\newtheorem{remark}[thm]{Remark}
\newtheorem{question}[thm]{Question}

\usepackage{hyperref}
\hypersetup{
  colorlinks   = true, 
  urlcolor     = blue, 
  linkcolor    = blue, 
  citecolor   = red 
}

\def\rquotient#1#2{%
	\makeatletter
	\raise.3ex\hbox{$#1$}/\lower.3ex\hbox{$#2$}%
	\makeatother
}	

\makeatletter
\newcommand{\subjclass}[2][2010]{%
	\let\@oldtitle\@title%
	\gdef\@title{\@oldtitle\footnotetext{#1 \emph{Mathematics subject classification.} #2}}%
}
\newcommand{\keywords}[1]{%
	\let\@@oldtitle\@title%
	\gdef\@title{\@@oldtitle\footnotetext{\emph{Key words and phrases.} #1.}}%
}
\makeatother

\newcommand{\Address}{{
		\bigskip 
		\small
		
		\textsc{Institut Montpellierain Alexander Grothendieck, 499-554 Rue du Truel, 34090 Montpellier, France.}\par\nopagebreak
		\textit{E-mail address}: \texttt{anthony.genevois@umontpellier.fr}
\medskip

		\textsc{Institut de Math\'ematiques de Jussieu-Paris Rive Gauche, Place Aur\'elie Nemours, 75013 Paris, France.}\par\nopagebreak
		\textit{E-mail address}: \texttt{romain.tessera@imj-prg.fr}
\medskip
		
}}

\title{Lamplighter-like geometry of groups}
\date{\today}
\author{Anthony Genevois and Romain Tessera}


\begin{document}

\maketitle

\begin{abstract}
In this article, we introduce \emph{halo products} as a natural generalisation of wreath products. They also encompass lampshuffler groups $\mathrm{FSym}(H) \rtimes H$ and lampcloner groups $\mathrm{FGL}(H) \rtimes H$, as well as many possible variations based for instance on braid groups, Thompson's groups, mapping class groups, automorphisms of free groups. We build a geometric framework that allows us to study the large-scale geometry of halo groups, providing refined invariants distinguishing various halo groups up to quasi-isometry. 
\end{abstract}

\footnotesize
\tableofcontents
\normalsize

\section{Introduction}

\noindent
The \emph{wreath product} of two groups $F$ and $H$ is
$$F \wr H := \left( \bigoplus\limits_H F \right) \rtimes H,$$
where $H$ acts on the direct sum by permuting the coordinates according to its action on itself by left-multiplication. Wreath products, also known as \emph{lamplighter groups}, are well-known in group theory, and have been studied from various perspectives over the years, including for instance random walks \cite{MR732356, MR1905862, MR1708557}, isoperimetric profiles \cite{MR2011120}, functional analysis \cite{MR2557962}, subgroup distortion \cite{MR2811580}, morphisms and automorphisms \cite{MR918632, MR3438162, MR188280, MR511604, MR4308638, AutLamp, HopfLamp}, Haagerup property \cite{MR2318545, MR2393636, MR2888241, LampMedian}, Hilbert space compression \cite{MR2271228, MR2644886, MR2783928, QM}, fixed-point properties \cite{MR2764930, MR3786300, genevois_2022, MR4518652, FWLamp}. On the one hand, lamplighter groups have an easy and explicit definition, allowing an easy access to various properties and calculations. On the other hand, these groups are sufficiently exotic, i.e.\ sufficiently far away from most of the well-understood classes of groups available in the literature, in order to exhibit interesting behaviours. The combination of these two observations probably explains the success of lamplighter groups, and why they are often used to produce counterexamples.

\medskip \noindent
From the perspective of large-scale geometry, lamplighter groups first appeared in \cite{MR1800990}, which shows thanks to wreath products that various algebraic properties, such as being virtually solvable, are not preserved by quasi-isometries. But the geometry of lamplighter groups remained poorly understood for a long time. A major progress appeared in \cite{EFWI, EFWII}, dedicated to the large-scale geometry of some \emph{horospherical products}, which include wreath products of the form $(\text{finite}) \wr (\text{two-ended})$ \cite{MR2421161}. In particular, it is proved that the lamplighter groups $\mathbb{Z}/n\mathbb{Z} \wr \mathbb{Z}$ and $\mathbb{Z}/m \mathbb{Z} \wr \mathbb{Z}$ are quasi-isometric if and only if $n,m$ are powers of a common number. Inspired by this work, \cite{MR2730576} highlights lamplighter groups as the first source of quasi-isometric groups that are not bijectively quasi-isometric. In a more recent work \cite{LampGT}, by using a different approach, we introduced new techniques and completely classified wreath products $(\text{finite}) \wr (\text{one-ended finitely presented})$ up to quasi-isometry. 

\medskip \noindent
A natural question is whether this new understanding can be extended beyond lamplighter groups, and in particular to the few variations of wreath products that can already be found in the literature. Such examples include the \emph{lampshuffler groups}\footnote{Terminology coined in \cite{Lampshuffler}.}
$$\circledS H:= \mathrm{FSym}(H) \rtimes H,$$
where $\mathrm{FSym}(H)$ denotes the group of the finitely supported permutations of $H$ and where $H$ acts on $\mathrm{FSym}(H)$ through its action on itself by left-multiplication; as well as the \emph{lampcloner groups}\footnote{Terminology coined in here, see Section~\ref{section:Halo} for a justification.}
$$\oslash_\mathfrak{k} H:= \mathrm{FGL}(H,\mathfrak{k}) \rtimes H,$$
where $\mathfrak{k}$ is a finite field, where $\mathrm{FGL}(H,\mathfrak{k})$ is the group of the linear automorphisms of the $\mathfrak{k}$-vector space spanned by $H$ fixing all but finitely many basis vectors, and where $H$ acts on $\mathrm{FGL}(H)$ through its action on itself by left-multiplication. The literature dedicated to lampshufflers and lampcloners, even though it is smaller than the literature dedicated to lamplighters, highlights these groups as being also a good source of counter-examples and exotic behaviours. See for instance \cite{MR1458419, MR2220572, MR2535083, MR3742543, MR3867301, MR3941140, Mimura, MR4017156, MR4243070, MR4245575, Bradford}. 

\medskip \noindent
Another worth mentioning example is given by \emph{verbal wreath products}, introduced in \cite{MR0193131}. The idea is to replace, in the definition of wreath products, the direct sum with a \emph{verbal product} \cite{MR97439}. A verbal product $\overset{\mathfrak{w}}{\ast}$ depends on a set of words $\mathfrak{w}$. When $\mathfrak{w}$ is a commutator of two letters, the verbal product coincides with a direct sum. Combining commutators, one can define \emph{nilpotent} and \emph{metabelian products} \cite{MR0075948, MR0075947}. Given two groups $F,H$ and a set of words $\mathfrak{w}$, the verbal wreath product is then defined as
$$F \overset{\mathfrak{w}}{\wr} H : = \left( \overset{\mathfrak{w}}{\underset{H}{\ast}}F \right) \rtimes H$$
where $H$ acts on the verbal product by permuting the coordinates. We refer the reader to Section~\ref{section:Halo} for precise definitions. Originated from varieties of groups, a vast literature is now dedicated to verbal and verbal wreath products, with various applications to group theory. 

\medskip \noindent
Inspired by the examples above, one can imagine an endless list of groups sharing with wreath products a similar algebraic structure. But, despite this seemingly similar structure, it is not clear whether all these groups should be considered as belonging to a common coherent family. In this article, taking the point of view of large-scale geometry, we motivate the idea that this should be the case. Formally, we introduce and study the following family of groups, which encompasses all the examples previously mentioned.

\begin{definition}
Let $H$ be a group. A \emph{halo of groups $\mathscr{L}$ over $H$} is the data, for every subset $S \subset H$, of a group $L(S)$ such that:
\begin{itemize}
	\item for all $R,S \subset H$, if $R \subset S$ then $L(R) \leq L(S)$;
	\item $L(\emptyset)=\{1\}$ and $L(H)= \langle L(S), S \subset H \text{ finite} \rangle$;
	\item for all $R,S \subset H$, $L(R) \cap L(S) = L(R \cap S)$. 
\end{itemize}
Given a morphism $\alpha : H \to \mathrm{Aut}(L(H))$ satisfying $\alpha(h) (L(S))= L(hS)$ for all $S \subset G$, $h \in H$, the \emph{halo product} $\mathscr{L}_{\alpha} H$ (or just $\mathscr{L}H$ if the action of $H$ on $L(H)$ is clear) is the semidirect product $L(H) \rtimes_\alpha H$. 
\end{definition}

\noindent
As announced, halo products include wreath products, and more generally verbal wreath products, as well as lampshuffler groups and lampcloner groups. But many more examples can be imagined. See Section~\ref{section:Halo} for other possible constructions.

\medskip \noindent
In this article, we focus on large-scale geometry in order to justify the unity of this family of groups. More precisely, we show that the geometry highlighted for lamplighter groups in \cite{LampGT} extends to halo products, which allows us to construct refined invariants up to quasi-isometry.

\paragraph{The Embedding Theorem.} In \cite{LampGT}, the main tool exploited in the study of the large-scale geometry of wreath products was an embedding theorem \cite[Theorem~1.19]{LampGT}, stating that, given a finite group $F$ and a finitely generated group $H$, every coarse embedding $Z \to F \wr H$ from a uniformly one-ended coarsely simply connected space $Z$ (e.g.\ a one-ended finitely presented group) must have its image contained in a neighbourhood of an $H$-coset. Thus, we can imagine our wreath product $F \wr H$ as covered by \emph{leaves}, namely the $H$-cosets, and our embedding theorem implies that, in case $H$ is one-ended and finitely presented, the leaves can be characterised purely geometrically. In particular, they are (quasi-)preserved by quasi-isometries. 

\medskip \noindent
As the main result of this article, we extend this embedding theorem to halo products:

\begin{thm}\label{Intro:EmbeddingThm}
Let $\mathscr{L}H$ be a finitely generated halo product with $L(H)$ locally finite. Let $Z$ be a geodesic metric space satisfying the thick bigon property. For every coarse embedding $Z \to \mathscr{L}H$, the image of $Z$ lies in a neighbourhood of some pseudo-leaf. 
\end{thm}

\noindent
Our statement requires some explanation. First, it would be tempting to claim that the image of our coarse embedding $Z \to \mathscr{L}H$ must lie in a neighbourhood of a leaf, i.e.\ an $H$-coset. Unfortunately, such a statement does not hold, even in simple cases. See Example~\ref{ex:NotInALeaf}. This is why we need to replace, in our statement, leaves with pseudo-leaves. Formally, a leaf coarsely coincides with a subset of the form
$$\left\{ (c,p) \in L(H) \rtimes H \mid p \in H, c \in c_0 L(B(p,R)) \right\}$$
for some $c_0 \in L(H)$ and $R \geq 0$. For wreath products and for $c_0$ trivial, this amounts to requiring that the lamps may be on only if they are close to the position of the lamplighter. In order to define a pseudo-leaf, we add one point around which the lamps are allowed to be on. Formally, a pseudo-leaf is a subset of the form
$$\left\{(c,p) \in L(H) \rtimes H \mid p \in H, c \in c_0 L \left( \{p,p_0\}^{+R} \right) \right\}$$
for some $c_0 \in L(H)$, $p_0 \in H$, and $R \geq 0$. Replacing leaves with pseudo-leaves is sufficient to generalise the embedding theorem from \cite{LampGT}. Nevertheless, we introduce in Section~\ref{section:Full} specific halos of groups, which we call \emph{full halos}, for which we recover an embedding theorem with respect to leaves. See Theorem~\ref{thm:InALeaf}. Full halos include lamplighters, lampshufflers, lampcloners, and $2$-nilpotent wreath products. 

\medskip \noindent
The proof of the Embedding Theorem for lamplighters presented in \cite{LampGT} relied on explicit group presentations that can be truncated in order to obtain semidirect products admitting some nice quasi-median geometry. For an arbitrary halo product $\mathscr{L}H$, no explicit group presentation is available without additional information on the structure of the locally finite group $L(H)$. This is why the proof of Theorem~\ref{Intro:EmbeddingThm} requires new tools and ideas. A description of the strategy will be given later in the introduction. Our arguments, which are conceptually more efficient, allow us to identify more precisely the spaces to which our embedding theorem applies. They are the space satisfying the \emph{thick bigon property}. We refer to Section~\ref{section:TBP} for a detailed description of this property. The key point is that such spaces include many groups of interest, such as:

\begin{prop}
The following groups satisfy the thick bigon property:
\begin{itemize}
	\item one-ended finitely presented groups;
	\item wreath products $E \wr H$ of two finitely generated groups with $E$ infinite and $H$ non-trivial;
	\item permutational wreath products $F\wr_{H/K} H$ with $F$ finitely generated, $H$ one-ended finitely presented, and $K \leq H$ infinite;
	\item finitely generated groups that are not virtually cyclic but contain normal free abelian subgroups of positive rank (e.g.\ torsion-free solvable groups that are not cyclic);
	\item finitely generated groups containing infinite normal subgroups of infinite index that are finitely generated;
	\item finitely generated groups containing $s$-normal one-ended finitely presented subgroups.
\end{itemize}
\end{prop}

\noindent
Therefore, Theorem~\ref{Intro:EmbeddingThm} generalises the embedding theorem from \cite{LampGT} to halo products but also improves it for lamplighters.

\paragraph{Algebraic consequences.} Before describing the geometric consequences of Theorem~\ref{Intro:EmbeddingThm}, it is worth mentioning that non-trivial algebraic information can be also extracted. For instance, due to the fact that subgroups are always coarsely embedded, we can prove that:

\begin{prop}\label{Intro:Algebraic}
Let $\mathscr{L}H$ be a finitely generated halo group with $\mathscr{L}$ full and $L(H)$ locally finite. Every one-ended finitely presented subgroup in $\mathscr{L}H$ is conjugate to a subgroup of~$H$.
\end{prop}

\noindent
Even for explicit examples such as lamplighters, lampshufflers, and lampcloners, the statement is non-trivial, and it is not clear how to prove it from a purely algebraic point of view. 

\medskip \noindent
As another application of Theorem~\ref{Intro:EmbeddingThm}, it follows that many one-ended halo products are not finitely presented. Indeed, given a one-ended halo product $\mathscr{L}H$ with $L(H)$ locally finite, if $\mathscr{L}H$ is finitely presented then Theorem~\ref{Intro:EmbeddingThm} applies to the identity map $\mathscr{L}H \to \mathscr{L}H$, proving that $H$ must have finite index in $\mathscr{L}H$, which amounts to saying that $L(H)$ is finite. For instance, this shows that lampshuffler groups over infinite groups are never finitely presented. 

\medskip \noindent
Interestingly, part of the strategy followed in order to prove Theorem~\ref{Intro:EmbeddingThm} can be adapted to halo products $\mathscr{L}H$ with $L(H)$ not necessarily locally finite. This allows us to prove the following strengthening of the previous observation:

\begin{thm}\label{Intro:NotFP}
Let $\mathscr{L}H$ be a halo product. Assume that there exists a finite subset $F\subset L(H)$ for which there exist infinitely many $h \in H$ such that $\langle F,hFh^{-1} \rangle \neq \langle F \rangle \ast \langle hFh^{-1} \rangle$. Then $\mathscr{L}H$ is not finitely presented. 
\end{thm}

\noindent
We emphasize that there exist finitely presented halo products. Of course, there are the trivial cases: if $L(H)$ is finite and $H$ finitely presented, or if $L(H)$ is finitely presented and $H$ finite, then $\mathscr{L}H$ is finitely presented. A more interesting, but yet elementary, example is given by the halo product $\mathscr{L}H$ where each $L(S)$ is the free product $\ast_{s \in S} F$ for some fixed group $F$. Then $\mathscr{L}H$ coincides with the free product $F \ast H$, and will be finitely presented if so are $F$ and $H$. Theorem~\ref{Intro:NotFP} states these are essentially the only examples of finitely presented halo products. 

\medskip \noindent
Notice that Theorem~\ref{Intro:NotFP} generalises \cite{MR120269}, which characterises finitely presented wreath products.

\paragraph{First geometric consequences.} The Embedding Theorem has deep consequences on the large-scale geometry of $\mathscr{L}H$. First of all, it implies that quasi-isometries between many halo products must be leaf-preserving. In Section~\ref{section:QIinvariants}, we introduce \emph{graphs of leaves} in order to record the geometric pattern of leaves in halo products and we show that every leaf-preserving quasi-isometry between two halo products induces a quasi-isometry between their graphs of leaves. As a consequence, we able to distinguish several halo products up to quasi-isometry.

\begin{prop}
Let $F$ be a finite group and $H$ a one-ended finitely presented group. If $F$ is not perfect, then the $2$-nilpotent wreath product $F \wr^{\mathfrak{n}_2} H$ cannot be quasi-isometric to a lamplighter $E \wr K$, to a lampshuffler $\circledS K$, or to a lampcloner $\oslash_\mathfrak{k} K$ where $K$ is a one-ended finitely presented group, $E$ a finite group, and $\mathfrak{k}$ a finite field.
\end{prop}

\begin{prop}
Let $F$ be a finite group and $H,K$ two one-ended finitely presented groups. The lampligther $F \wr K$ and the lampshuffler $\circledS H$ are not quasi-isometric.
\end{prop}

\noindent
It is worth mentioning that the Embedding Theorem does not only allow us to distinguish halo products, it also brings sever restrictions on arbitrary finitely generated groups quasi-isometric to some halo products. A remarkable statement in this direction is the Subgroup Theorem.

\begin{thm}
Let $H$ be a one-ended finitely presented group and let $\mathscr{L}H$ be a finitely generated halo group with $\mathscr{L}$ full and $L(H)$ locally finite. Every finitely generated group $G$ quasi-isometric to $\mathscr{L}H$ contains a finite collection of subgroups $\mathcal{H}$ such that:
\begin{itemize}
	\item all the groups in $\mathcal{H}$ are quasi-isometric to $H$;
	\item the collection $\mathcal{H}$ is almost malnormal;
	\item every one-ended finitely presented group in $G$ is contained in a conjugate of a subgroup from $\mathcal{H}$.
\end{itemize}
\end{thm}

\noindent
Currently, a complete classification of finitely generated groups quasi-isometric to halo products seems to be out of reach, even in simple cases. For instance, we do not know exactly which finitely generated groups are quasi-isometric to $\mathbb{Z}/2\mathbb{Z} \wr \mathbb{Z}^2$. Nevertheless, in the realm of amenable groups, we can deduce from the Subgroup Theorem:

\begin{cor}
Let $\mathscr{L}H$ be a finitely generated halo product with $\mathscr{L}$ full, $L(H)$ locally finite, and $H$ nilpotent. Every elementary amenable group $G$ quasi-isometric to $\mathscr{L}H$ splits as a semidirect product $L \rtimes \bar{H}$ where $\bar{H}$ is quasi-isometric to $H$, where $L$ is locally finite, and where $\bar{H}$ acts by conjugation on $L\backslash \{1\}$ with finite stabilisers.
\end{cor}

\noindent
We refer to Theorem~\ref{thm:QIrigidityEA} for a more general statement, and to Section~\ref{section:SubThm} for more consequences of the Subgroup Theorem.

\paragraph{Aptolic quasi-isometries.} Following \cite{LampGT}, we introduce a notion of quasi-isometry between halo products that is compatible with the halo structure. Namely:

\begin{definition}
Let $\mathscr{M}A,\mathscr{N}B$ be two finitely generated halo groups. A quasi-isometry $\Psi : \mathscr{M}A \to \mathscr{N}B$ is \emph{aptolic} if there exists a bijection $\alpha : M(A) \to N(B)$ and a quasi-isometry $\beta : A \to B$ such that $\Psi : (c,p) \mapsto (\alpha(c),\beta(p))$. 
\end{definition}

\noindent
In \cite{LampGT}, we proved that a leaf-preserving quasi-isometry between lamplighter groups is always aptolic (up to finite distance). In this article, we show that the same phenomenon occurs for many halo products. We emphasize that our arguments generalise those from \cite{LampGT} to more general halo products, but also simplify them, both technically and conceptually. This is due to specific configurations of leaves, called \emph{ladders}, which we introduce in Section~\ref{section:Ladders}. We refer to Theorem~\ref{thm:Stiff} and Corollary~\ref{cor:AptoAltitude} for sufficient conditions that assure that leaf-preserving quasi-isometries are aptolic. The key point is that the main examples we are interested in are covered, including lamplighters, lampshufflers, and lampcloners.

\medskip \noindent
It is worth mentioning that our techniques also apply to aptolic coarse embeddings. See Section~\ref{section:CoarseEmb} for more details.

\paragraph{Growth of lamps.} The property that every quasi-isometry between two halo products lies at bounded distance from an {\it aptolic} quasi-isometry  turns out to be a powerful tool, which was already crucial in our quasi-isometric classification of lamplighter groups in \cite{LampGT}. As we shall see, this property allows us to obtain refined quasi-isometric invariants for certain families of halo products. 

 In the rest of the introduction, we focus on a specific class of halos, which we call \emph{charming}. This is not necessary, but this will simplify the exposition. See Section~\ref{section:ApplicationsTwo} for more general statements. For now, just keep in mind that charming halos include lamplighters, lampshufflers, lampcloners, and nilpotent wreath products. 

\medskip \noindent
Given a finitely generated halo product $\mathscr{L}H$ with $L(H)$ locally finite and $\mathscr{L}$ charming, the \emph{lamp growth sequence} is
$$\Lambda_\mathscr{L} : n \mapsto |L(S)| \text{ where } S \subset H \text{ satisfies } |S|=n.$$
It follows from the definition of a charming halo that this quantity is well-defined. Roughly speaking, it quantifies the size of a finite subgroup of $L(H)$ in terms of its support. For instance, the lamp growth sequence of a lamplighter $E \wr K$ (resp.\ a lampshuffler $\circledS K$) is $n \mapsto |E|^n$ (resp.\ $n \mapsto n!$). 

\begin{thm}\label{Intro:LampGrowthBis}
Let $\mathscr{M}A,\mathscr{N}B$ be two finitely generated halo products with $\mathscr{M},\mathscr{N}$ charming and with $M(A),N(B)$ locally finite. If there exists an aptolic coarse embedding $\mathscr{M}A \to \mathscr{N}B$, then $\Lambda_\mathscr{M} \prec \Lambda_\mathscr{N}$. 
\end{thm}

\noindent
Recall that a sequence $\varphi : \mathbb{N} \to \mathbb{N}$ \emph{dominates} a sequence $\psi : \mathbb{N} \to \mathbb{N}$, which we write $\psi \prec \varphi$, if there exists some $C \geq 0$ such that $\varphi(n) \leq \psi(C \cdot n)$ for every $n \geq 0$. The sequences are \emph{equivalent}, which we write $\varphi \sim \psi$, if $\varphi \prec \psi$ and $\psi \prec \varphi$. 

\medskip \noindent
Thus, Theorem~\ref{Intro:LampGrowthBis} allows us to distinguish halo products up to quasi-isometry whenever their lamp growth sequences do not have the same asymptotic behaviour. For instance, this implies that, over a one-ended finitely presented group, a lamplighter, a lampshuffler, and a lampcloner are never quasi-isometric. More precisely:

\begin{cor}
Let $F$ be a finite group, $I,J,K$ three one-ended finitely presented groups, and $\mathfrak{k}$ a finite field. Then the lamplighter $F \wr I$, the lampshuffler $\circledS J$, and the lampcloner $\oslash_\mathfrak{k} K$ are pairwise non-quasi-isometric.
\end{cor}

\noindent
See Corollary~\ref{cor:ManyNotQI} for a more general statement, including other halo groups introduced in Section~\ref{section:Halo}. It is worth mentioning that it follows from \cite{MR2011120,MR4245575} that a lampshuffler and a lamplighter over the same amenable groups have distinct F\o lner profiles, and hence are never quasi-isometric. In the non-amenable case, or when the lampshuffler and the lamplighter are taken over different groups, then no previously known invariant seems to be able to distinguish these two groups up to quasi-isometry. See also Corollary~\ref{cor:NoCoarseEmb} for applications to coarse embeddings. For instance, it shows that lampcloners (resp.\ lampshufflers) usually do not coarsely embed in lampshufflers (resp.\ lamplighters). 

\medskip \noindent
However, the asymptotic behaviour of lamp growth sequences is not sufficient to distinguish, say, two lampshufflers or two lampcloners. In Section~\ref{section:GrowthArithm}, we show that the lamp growth sequences of quasi-isometric halo products do not only share the same asymptotic behaviour, but also the same arithmetic properties. 

\begin{thm}\label{Intro:Interlaced}
Let $\mathscr{M}A,\mathscr{N}B$ be two finitely generated halo products with $\mathscr{M},\mathscr{N}$ charming and with $M(A),N(B)$ locally finite. If there exists an aptolic quasi-isometry $\mathscr{M}A \to \mathscr{N}B$, then the lamp growth $\Lambda_\mathscr{M}$ is $\mathbb{B}$-interlaced with $\Lambda_\mathscr{N}$, where $\mathbb{B}$ denotes the boundary growth of $A$. 
\end{thm}

\noindent
The theorem requires some explanation. First, given a locally finite graph $X$, its \emph{boundary growth} is the map
$$\mathbb{B} : n \mapsto \min \{ |\partial S| \mid S \subset X \text{ connected of size }  n\}.$$
Applied to a finitely generated group, the boundary growth depends on the choice of the specific finite generating set, but its asymptotic growth only depends on the group. Next, given a function $\Delta : \mathbb{N} \to \mathbb{N}$ and two sequences $\varphi, \psi : \mathbb{N} \to \mathbb{N}$, we say that $\varphi$ is \emph{(arithmetically) $\Delta$-interlaced} with $\psi$ if there exist a constant $C \geq 0$ and a sequence $(x_n)$ satisfying $x_n= \Theta(n)$ such that
$$\varphi(n) \text{ divides } \psi(x_n), \text{ which divides } \varphi(n+ C \cdot \Delta(n))$$
for every $n \geq 1$. In other words, we know that each term of $\varphi$ divides a term of $\psi$, and conversely, with a control on where to find the divisible term (depending on $\Delta$). 

\medskip \noindent
As an application, it is not difficult to reprove that, given two finite groups $E,F$ and two one-ended finitely presented groups $H,K$, if $E \wr H$ and $F \wr K$ are quasi-isometric, then $|E|,|F|$ must have the same prime divisors \cite{LampGT}. See Theorem~\ref{thm:LighterPrime}. In fact, Theorems~\ref{thm:LighterPrime} and~\ref{thm:LampAmenable} generalises the classification of \cite{LampGT} by replacing one-ended finitely presented groups with finitely generated groups satisfying the thick bigon property.

\medskip \noindent
More interestingly, we also get some information about lampshuffler and lampcloner groups:

\begin{thm}
Let $H,K$ be two one-ended finitely presented groups. The lampshuffler groups $\circledS H$ and $\circledS K$ are quasi-isometric if and only if $H$ and $K$ are bijectively quasi-isometric. 
\end{thm}

\begin{thm}\label{thm:IntroCloner}
Let $\mathfrak{h},\mathfrak{k}$ be two finite fields and $H,K$ two finitely presented one-ended groups. If the lampcloner groups $\oslash_\mathfrak{h} H$ and $\oslash_\mathfrak{k} K$ are quasi-isometric, then $\mathrm{char}(\mathfrak{h})= \mathrm{char}(\mathfrak{k})$ and there exists a quasi-$\kappa$-to-one quasi-isometry $H \to K$ where $\kappa$ denotes the ratio $\sqrt{ \frac{\log(|\mathfrak{k}|)}{\log(|\mathfrak{h}|)}}$. 
\end{thm}

\noindent
See Corollary~\ref{cor:QILampshuffler}, Theorem~\ref{thm:Lampcloner}, and Lemma~\ref{lem:ClonerKappa}. We refer to Section~\ref{section:amenable} for the definition of \emph{quasi-$\kappa$-to-one} quasi-isometries. Roughly speaking, the $\kappa$ quantifies the obstruction for a quasi-isometry to be at finite distance from a bijection. 

\medskip \noindent
The restriction imposed by Theorem~\ref{Intro:Interlaced} on lamp growth sequences is even stronger when the groups under consideration are amenable, because in this case the boundary growth is sublinear. This allows us to distinguish up to quasi-isometry other halo groups introduced in Section~\ref{section:Halo}, which we call \emph{lampjuggler} and \emph{lampdesigner groups}. See Theorems~\ref{thm:LampjugglerClassification} and~\ref{thm:ClassificationDesigner} for precise statements.

\paragraph{About the proof of the Embedding Theorem.} Let $\mathscr{L}H$ be a finitely generated halo group with $L(H)$ locally finite. The first step towards the proof of Embedding Theorem is to construct a contractible cube complex $\mathscr{C}$, endowed with a height function $\mathfrak{h} : \mathscr{C} \to \mathbb{N}$, such that each subcomplex $\mathscr{C}_k := \mathfrak{h}^{-1}([1, k])$ is naturally quasi-isometric to $\mathscr{L}H$. The idea is the following. Geometrically, an element of $\mathscr{L}H$ can be thought of as a pair $(c,h)$ with a \emph{colouring} $c \in L(H)$ and an \emph{arrow} $h \in H$, and moving in $\mathscr{L}H$ amounts to moving the arrow in $H$ and, along the way, to modifying the colouring by an element of $L(H)$ supported in a small ball round the arrow. In order to construct $\mathscr{C}$, we replace the arrow with a finite connected subgraph of $H$, a \emph{crowd}. Then, moving in $\mathscr{L}H$ amounts to moving the crowd by adding or removing vertices and, along the way, to modifying the colouring by an element of $L(H)$ supported in the crowd. The height of a vertex of $\mathscr{C}$ records the size of the crowd. We refer the reader to Section~\ref{section:EmbeddingTheorem} for details.

\medskip \noindent
The key property is that, in each $\mathscr{C}_k$, two points $x$ and $y$ that do not belong to a small neighbourhood of some pseudo-leaf must be ``topologically separated'' by specific subcomplexes, called \emph{blocks}. More precisely, for every path $\gamma$ connecting $x$ to $y$, there exists a block such that every path homotopically equivalent to $\gamma$ has to intersect this block. In wreath products, this property records the following phenomenon. Fix a path $\gamma$ in a lamplighter group $\mathbb{Z}/2\mathbb{Z} \wr H$ from the identity to some point $(c,h)$. Also, fix a point $p$ at which $c$ is non-trivial. Along $\gamma$, the arrow has to move from $1$ to $h$ and to turn on the lamps at each point in the support of $c$. In particular, we know that, at some point, the arrow has to point to $p$. The same observation applies to any path $\gamma'$ connecting $1$ to $(c,h)$. Now, compare the colourings outside $p$ at the points of $\gamma$ and $\gamma'$ where the arrows point to $p$. A possibility for these two colourings to be different in a significative way is that orders according to which the lamps of $\mathrm{supp}(c)$ are turned on along $\gamma$ and $\gamma'$ are very different. For instance, there exists $q \in \mathrm{supp}(c)$ far from $p$ such that $\gamma$ turns on $q$ before $p$ while $\gamma'$ turns on $q$ after $p$. But the cycle in $\mathbb{Z}/2\mathbb{Z} \wr H$ that turns on $p$, turns on $q$, turns off $p$, and finally turns off $q$ is (coarsely) homotopically trivial only at a very large scale (because $p$ and $q$ are very far apart). Therefore, at a small scale, the paths $\gamma$ and $\gamma'$ cannot be (coarsely) homopoty equivalent. Thus, if we are only allowed to modify $\gamma$ up to (coarse) homotopy (at some fixed scale), then, along our new path, we can always find a point where the arrow points to $p$ and where the colouring outside $p$ is fixed. 

\medskip \noindent
Now, assume that we have a coarse embedding $Z \to \mathscr{L}H$. By construction of our cubical model, it can be thought of as a coarse embedding $Z \to \mathscr{C}_2 \subset \mathscr{C}$. For simplicity, assume that $Z$ is coarsely simply connected. Then the image of $Z$ in $\mathscr{C}_2$ must be simply connected in $\mathscr{C}_k$ for some large enough $k \geq 2$. As a consequence of what we explained in the previous paragraph, if the image of $Z$ is not contained in a small neighbourhood of a pseudo-leaf, then we can find a path $\gamma$ in the image of $Z$ that connects two far apart points such that every path homotopy equivalent to $\gamma$ has to intersect a given block. But our blocks are all bounded. So, if, say, $Z$ is one-ended, a path can always be modified up to (coarse) homotopy in order to avoid a given bounded subspace. Hence a contradiction. In full generality, we only assume that $Z$ satisfy the thick bigon property, and not that it is coarsely simply connected and one-ended, but the idea is essentially the same.

\paragraph{Organisation of the article.} In Section~\ref{section:Halo}, we describe in details halo products of groups together with explicit examples and non-examples. Throughout the article, our applications only deal with wreath products, lampshuffler and lampjuggler groups, $2$-nilpotent wreath products, lampdesigner groups, and lampcloner groups over finite fields. 

\medskip \noindent
In Section~\ref{section:TBP}, we define the thick bigon property and exhibit many examples of finitely generated groups satisfying it. The section is essentially self-contained, so the proofs can be skipped by hurried readers. 

\medskip \noindent
Section~\ref{section:EmbeddingTheorem} is dedicated to the Embedding Theorem and its consequences. Sections~\ref{section:Cubical} and~\ref{section:Separation} contain preliminaries, the proof of the theorem being given in Section~\ref{section:BigProof}. In Section~\ref{section:Full}, we introduce full halos of groups and show how to strengthen the conclusion of the Embedding Theorem for the corresponding halo products. In Section~\ref{section:SubThm}, we deduce the Subgroup Theorem, which imposes severe restrictions on finitely generated groups quasi-isometric to halo products, and we include some applications. Finally, in Section~\ref{section:QIinvariants}, we introduce graphs of leaves, which record how the $H$-cosets in halo products $\mathscr{L}H$ are geometrically organised, and we deduce from the Embedding Theorem that they are invariant under quasi-isometry. Some concrete applications are mentioned. Sections~\ref{section:SubThm} and~\ref{section:QIinvariants} are essentially independent of the sequel of the article.

\medskip \noindent
In Section~\ref{section:FP}, we exploit the techniques used during the proof of the Embedding Theorem in order to show that most halo products are not finitely presented. The section, mainly algebraic, is independent of the rest of the article.

\medskip \noindent
Aptolic quasi-isometries are introduced in Section~\ref{section:Aptolic}. Our main goal there is to show that leaf-preserving quasi-isometries between halo products are automatically aptolic (up to finite distance). We focus essentially on large-scale commutative halos of groups, defined in Section~\ref{section:QIinvariants}. Our main tool is given by specific configurations of leaves, called ladders. They are defined and described in Section~\ref{section:Ladders}. Aptolicity of quasi-isometries is deduced in Sections~\ref{section:Morse} and~\ref{section:Sufficient} under various assumptions. Section~\ref{section:CoarseEmb} addresses similar questions for coarse embeddings.

\medskip \noindent
In Section~\ref{section:ApplicationsTwo}, we introduce invariants under aptolic quasi-isometries, of analytic nature in Section~\ref{section:VolumeGrowth} and of arithmetic nature in Section~\ref{section:GrowthArithm}. Section~\ref{section:GrowthApplications} contains concrete applications. In Section~\ref{section:amenable}, we strengthen the conclusions obtained in Section~\ref{section:GrowthApplications} under the additional assumption that the groups under consideration are amenable. 

\medskip \noindent
Finally, Section~\ref{section:Conclusion} records various comments and open questions regarding the article.


\section{Halos of groups}\label{section:Halo}


\noindent
In this section, we define \emph{halo products} and provide various examples illustrating the notion. 

\begin{definition}\label{def:Halo}
Let $X$ be a set. A \emph{halo of groups $\mathscr{L}$ over $X$} is the data, for every subset $S \subset X$, of a group $L(S)$ such that:
\begin{itemize}
	\item for all $R,S \subset X$, if $R \subset S$ then $L(R) \leq L(S)$;
	\item $L(\emptyset)=\{1\}$ and $L(X)= \langle L(S), S \subset X \text{ finite} \rangle$;
	\item for all $R,S \subset X$, $L(R) \cap L(S) = L(R \cap S)$. 
\end{itemize}
Given an action $H \curvearrowright X$ and a morphism $\alpha : H \to \mathrm{Aut}(L(X))$ satisfying $\alpha(h) (L(S))= L(hS)$ for all $S \subset X$, $h \in H$, the \emph{(permutational) halo product} $\mathscr{L}_{X,\alpha} H$ is the semidirect product $L(X) \rtimes_\alpha H$. 
\end{definition}

\noindent
Permutational wreath products are basic examples of permutational halo products and motivate our definition. In this case, given two groups $F,H$ and a set $X$ on which $H$ acts, the corresponding halo of groups $\mathscr{L}$ over $X$ is obtained by setting $L(S):= \bigoplus_S F$, $S \subset X$. Then, the halo product $\mathscr{L}_{X,\alpha} H$ coincides with the permutational wreath product $F \wr_X H$ where $\alpha$ corresponds to the action of $H$ on $L(X)= \bigoplus_X F$ by permuting coordinates according to the action $H \curvearrowright X$. 

\medskip \noindent
The third item in Definition~\ref{def:Halo} is fundamental. In particular, it allows us to define a notion of \emph{support}:

\begin{definition}
Let $\mathscr{L}$ be a halo of groups over some set $X$. For every $g \in L(X)$, the \emph{support} of $g$, denoted by $\mathrm{supp}(g)$, is the smallest subset $S \subset X$, with respect to the inclusion, satisfying $g \in L(S)$. 
\end{definition}

\noindent
Alternatively, $\mathrm{supp}(g)$ can be defined as the intersection of all the subsets $S \subset X$ for which $g \in L(S)$. The second and third items in Definition~\ref{def:Halo} imply that the support of a non-trivial element is non-empty and finite. 
The following basic fact will be used throughout.

\begin{fact}\label{fact:supp}
Let $\mathscr{L}$ be a halo of groups over some set $X$. For all $x,y\in L(X)$, $\mathrm{supp}(xy) \subset \mathrm{supp}(x) \cup \mathrm{supp}(y)$.
\end{fact}

\begin{proof}
By the first axiom, we have 
$$x \in L(\mathrm{supp}(x)) \leq L(\mathrm{supp}(x) \cup \mathrm{supp}(y)) \text{ and } y \in L(\mathrm{supp}(y)) \leq L(\mathrm{supp}(x) \cup\mathrm{supp}(y)).$$
Hence, we have $xy \in L(\mathrm{supp}(x) \cup \mathrm{supp}(y))$, which amounts to saying that the desired inclusion $\mathrm{supp}(xy) \subset \mathrm{supp}(x) \cup \mathrm{supp}(y)$ holds.
\end{proof}

\noindent
In the same vein, we have:

\begin{fact}\label{fact:DisjointSupp}
Let $\mathscr{L}$ be a halo of groups over some set $X$. For all $x,y \in L(X)$, if $\mathrm{supp}(x) \cap \mathrm{supp}(y)=\emptyset$ then $\mathrm{supp}(x) \cup \mathrm{supp}(y) \subset \mathrm{supp}(xy)$.
\end{fact}

\begin{proof}
We have
$$\mathrm{supp}(x)= \mathrm{supp}(xyy^{-1}) \subset \mathrm{supp}(xy) \cup \mathrm{supp}(y^{-1})= \mathrm{supp}(xy) \cup \mathrm{supp}(y).$$
Since $\mathrm{supp}(x)$ is disjoint from $\mathrm{supp}(y)$, necessarily $\mathrm{supp}(x) \subset \mathrm{supp}(xy)$. The inclusion $\mathrm{supp}(y) \subset \mathrm{supp}(xy)$ is proved similarly.
\end{proof}

\noindent
The rest of the section is dedicated to (non-)examples of halo products. Since this is the case we are the most interested in, we only describe halo products given by groups acting on themselves by left-multiplication. But, of course, all the definitions can be extended to more general permutation halo products.

\paragraph{Lampshufflers and lampjugglers.} Let $H$ be a group and $n \geq 1$ an integer. The \emph{lampjuggler group} is 
$$\circledS_n H:= \mathrm{FSym}(H \times \{1, \ldots, n\}) \rtimes H$$
where the action $\alpha$ of $H$ on $\mathrm{FSym}(H \times \{1, \ldots, n\})$ is induced by the action of $H$ on $H \times \{1, \ldots, n\}$ given by $h \cdot (k,i)= (hk,i)$. For $n=1$, we refer to the group as the \emph{lampshuffler group} $\circledS H$. The lampjuggler group $\circledS_nH$ coincides with the halo group $\mathscr{L}_\alpha H$ where $\mathscr{L}$ is the halo that associates $L(S):= \mathrm{FSym}(S \times \{1, \ldots, n\})$ to every $S \subset H$. 

\medskip \noindent
Given a generating set $T_H$ of $H$, set 
$$T:= \left\{ (ij) \mid i \in \{1\} \times \{1, \ldots, n\}, j \in \{t\} \times \{1, \ldots, n\}, t \in T_H\right\}.$$ 
Then $T \cup T_H$ generates $\circledS_nH$. As a consequence, $\circledS_nH$ is finitely generated as soon as so is $H$. An element $(\sigma,p)$ of $\circledS_nH$ can be represented by the labelling of the vertices of $\mathrm{Cayl}(H,T_H) \times \{1, \ldots, n\}$ given by the bijection $\sigma : H \times \{1, \ldots, n\} \to H \times \{1, \ldots, n\}$ and an arrow pointed to the vertex $p \in \mathrm{Cayl}(H,T_H) \times \{1\}$. Right-multiplying $(\sigma,p)$ by a generator in $T \cup T_H$ amounts either to moving the arrow from $p$ to an adjacent vertex or to switching the labels of two vertices in $p \times \{1, \ldots, n\}$ and $q \times \{1, \ldots, n\}$ for some neighbour $q$ of $p$.

\paragraph{Lampdesigner.} Let $F,H$ be two groups. The \emph{lampdesigner group} is 
$$F \boxplus H:= (F \wr_H \mathrm{FSym}(H)) \rtimes H$$
where $H$ acts on $\bigoplus_H F$ by permuting the coordinates according to the action of $H$ on itself by left-multiplication and acts on $\mathrm{FSym}(H)$ according to the action on itself also by left-multiplication. The lampdesigner group $F \boxplus H$ coincides with the halo group $\mathscr{L} H$ where $\mathscr{L}$ associates $L(S):= F\wr_S \mathrm{FSym}(S)$ to every $S \subset H$.

\medskip \noindent
Lampdesigner groups are closely related to lampjuggler groups. The picture to keep in mind is that an element of $F \wr_H \mathrm{FSym}(H)$ provides a permutation of $F \times H$ that permutes the blocks $F \times \{h\}$ and that acts on each such block as a multiplication by an element of $F$. Observe that, if $F$ is finite, then $F \boxplus B$ is naturally a subgroup of the lampjuggler group $\circledS_{|F|} H$.

\medskip \noindent
Fix a generating set $T_H$ of $H$, a generating set $T_F$ of $F$, and set
$$T:= \{ \text{permutation of the blocks } F \times \{1\} \text{ and } F \times \{t\} \mid t \in T_H \}.$$
Then $T^+:= T \cup T_F \cup T_H$ generates $F \boxplus H$. In particular, $F \boxplus H$ is finitely generated as soon as so are $F$ and $H$. An element $(\sigma,p)$ of $F \boxplus H$ can be represented by the labelling of the vertices $\mathrm{Cayl}(F,T_F) \times \mathrm{Cayl}(H,T_H)$ given by the bijection $\sigma : F \times H \to F \times H$ and an arrow pointed to the vertex $p \in \{1\} \times \mathrm{Cayl}(H,T_H)$. Right-multiplying $(\sigma,p)$ by a generator in $T^+$ amounts either to moving the arrow from $p$ to a neighbour, or to switching the block $F \times \{p\}$ with the block $F \times \{q\}$ for some neighbour $q$ of $p$, or to permuting the labels in the block $F \times \{p\}$ according to the action of a generator in $T_F$.

\paragraph{Lampcloner.} Given a set $S$ and a ring $R$, the \emph{elementary group} $\mathrm{E}(S,R)$ is the group of linear automorphisms of the free $R$-module over $S$ generated by the \emph{transvections} 
$$\tau_{pq}(\lambda) :  \sum\limits_{s \in S} \mu_s e_s \mapsto \sum\limits_{s \neq p} \mu_se_s + (\mu_p + \lambda \mu_q)e_p,$$
with $p,q \in S$ distinct and $\lambda \in R$ invertible; and the diagonal matrices
$$\delta_p(\lambda) : \sum\limits_{s \in S} \mu_s e_s \mapsto \sum\limits_{s \neq p} \mu_s e_s + \lambda \mu_p e_p$$
with $p \in S$ and $\lambda \in R$ invertible. Given two groups $F,H$ and a ring $R$, the \emph{lampcloner product}  $F \oslash_R H$ is $ \mathrm{E}(H,R[F]) \rtimes H$, where $H$ acts on $\mathrm{E}(H,R[F])$ by permuting the coordinates of the free module over $H$. If $F$ is trivial, we note $\oslash_R H$ instead of $\{1\} \oslash_R H$, and we call such a group a \emph{lampcloner group}. In this article, we will be essentially interested in lampcloner groups over finite fields. 

\medskip \noindent
Given a finite field $\mathfrak{k}$ and a generating set $T_H$ of $H$, the lampcloner group $\oslash_\mathfrak{k} H$ is generated by the subset
$$T:= T_H \cup \{ \delta_1(\lambda), \tau_{1t}(\lambda) \mid t \in T_H, \lambda \in \mathfrak{k}^\times \}.$$
In particular, $\oslash_\mathfrak{k} H$ is finitely generated as soon as so is $H$. 

\medskip \noindent
Let $V$ denote the $\mathfrak{k}$-vector space freely generated by $H$. An element $(\varphi,p)$ in $\oslash_\mathfrak{k} H=\mathrm{FGL}(V) \rtimes H$ can be described by the labelling $h \mapsto \varphi(h) \in V$ of the vertices of $\mathrm{Cayl}(H,T_H)$ and by an arrow pointing to the vertex $p$. Right-multiplying by a generator from $T$ amounts either to moving the arrow from $p$ to a neighbour; or to multiplying the label $\varphi(p)$ by a non-trivial element of $\mathfrak{k}$; or to \emph{clone} (i.e.\ duplicate) the vector $\varphi(p)$ and add it to the label of a neighbour of $p$ after multiplication by an element of $\mathfrak{k}\backslash \{0\}$. 

\begin{remark}\label{rem:lamdesigner/lampcloner}
Note that a lampdesigner group, and a fortiori a lampshuffler or lampjuggler group, appears naturally as a subgroup of a lampcloner product. Indeed, a lampcloner product $F \oslash_R H$ naturally contains the wreath product $F \wr H$ (taking only diagonal matrices with entries in $F$) and the lampjuggler group $F \boxplus H$ (taking diagonal and permutation matrices, or equivalently the matrices with exactly one non-zero entry in each column and each line, each one given by element of $F$). 
\end{remark}

\paragraph{Verbal wreath products.} Let $\mathcal{G}= \{G_i \mid i \in I\}$ be a collection of groups and $\mathfrak{w}$ a set of words. The \emph{verbal product} \cite{MR97439} is
$$\overset{\mathfrak{w}}{\underset{i \in I}{\ast}} G_i := \left( \underset{i \in I}{\ast} G_i \right) / \left( \mathfrak{w}(\mathcal{G}) \cap [\mathcal{G},\mathcal{G}] \right)$$
where 
$$[\mathcal{G}, \mathcal{G}]:= \left\langle [g,h] \mid g \in G_i, h \in G_j \text{ with }  i,j \in I \text{ distinct} \right\rangle$$
and
$$\mathfrak{w}(\mathcal{G}):= \left\langle m(g_1,\ldots, g_k) \mid \begin{array}{c} m(x_1, \ldots, x_k) \in \mathfrak{w}, g_1 \in G_{i_1}, \ldots, g_k \in G_{i_k} \\ \text{with } i_1, \ldots, i_k \in I  \end{array} \right\rangle.$$
The definition extends \emph{k-nilpotent} and \emph{metabelian products} introduced in \cite{MR0075948, MR0075947}. Given two groups $F,H$ and a set of words $\mathfrak{w}$, the \emph{verbal wreath product} \cite{MR0193131} is
$$F \overset{\mathfrak{w}}{\wr} H : = \left( \overset{\mathfrak{w}}{\underset{H}{\ast}}F \right) \rtimes H$$
where $H$ acts on the verbal product by permuting the coordinates. This is also a halo product (see \cite[Lemma~4.6]{MR4502610} for the property on intersections). 

\begin{remark}\label{rem:subgraphs}
Let $\mathscr{L}H$ be a finitely generated halo product that is either a wreath product, a verbal wreath product, a lampjuggler group, a lampdesigner group, or a lampcloner product. Once $H$ is endowed with a finite generating set, it can be identified with its corresponding Cayley graph. Then, if $S \subset H$ is connected, it is worth noticing that $L(S)$ is generated by elements supported on vertices and edges. Moreover, when $L(H)$ is locally finite, the size of $L(S)$ only depends on the size of $S$. These properties are not automatic, and will be exploited in Section~\ref{section:GrowthApplications} in order to simplify statements proved in Sections~\ref{section:VolumeGrowth} and~\ref{section:GrowthArithm}.
\end{remark}

\paragraph{Others.} So far, we have defined examples already available in the literature, at least partially, and those which will interest us in the rest of the article. But many other constructions are possible, a few of which we describe now. 

\medskip \noindent
Let $H$ be an orderable group. Fixing a total order, we embed $H$ into $\mathbb{R}$ so that the order of $H$ coincides with the order induced by $\mathbb{R}$. Let $B(H)$ denote the group of the braids $\beta \subset \mathbb{R} \times [0,1]$ connecting $H \times \{0\}$ to $H \times \{1\}$ for which there exists a finite subset $S \subset H$ such that each strand starting from a point $(p,0) \in H \backslash S \times \{0\}$ is a straight line connected to $(p,1)$ always passing below the strands connected to $S \times \{0,1\}$. Because the order of $H$ is $H$-invariant, we can define the \emph{lampbraider} $\mathscr{B}H:= B(H) \rtimes H$. This is another example of a halo group.

\medskip \noindent
In the previous construction, braid groups can be replaced by one of the many variations of braid groups available in the literature, such as virtual braid groups, loop braid groups, welded braid groups, twin groups, cactus groups, and so on. 

\medskip \noindent
A variation of the previous construction, which does not require our group to be orderable, is the following. Let $n \geq 1$ be an integer and let $H$ be a group generated by a finite subset $T$. Because $H$ acts by automorphisms on its Cayley graph $\mathrm{Cayl}:= \mathrm{Cayl}(H,T)$, one can define a halo group $B_n(\mathrm{Cayl}) \rtimes H$, where $B_n(\mathrm{Cayl})$ denotes the braid group over the graph $\mathrm{Cayl}$ with $n$ strands. 

\medskip \noindent
To the graph $\mathrm{Cayl}$, one can also associate a surface $\Sigma$ obtained by taking a connected sum of tori indexed by the vertices of $\mathrm{Cayl}$, the tubes between the tori being pairwise disjoint and given by the edges of $\mathrm{Cayl}$. The action of $H$ on $\mathrm{Cayl}$ induces an action of $H$ on $\Sigma$ by homeomorphisms. If $\mathrm{FMod}(\Sigma)$ is the compactly supported mapping class group of $\Sigma$, then one can define a halo group $\mathrm{FMod}(\Sigma) \rtimes H$. It provides a natural finitely generated group containing the mapping class groups of all the closed surfaces. 

\medskip \noindent
A similar group can be defined for automorphisms of free groups. Given a group $H$, let $F(H)$ denote the free group over $H$. Let $\mathrm{FAut}(F(H))$ denote the group of the automorphisms $F(H) \to F(H)$ that fix all but finitely many free generators. Then the semi-direct product $\mathrm{FAut}(F(H)) \rtimes H$, where $H$ acts on $\mathrm{FAut}(F(H))$ through its action on $F(H)$ permuting the free generators, is a halo group. Again, it is finitely generated when so is $H$. An interesting question is whether such groups with better finiteness properties can be constructed from permutational halo products; see Question~\ref{question:FP}. 

\medskip \noindent
Given a group $A$, let $R(A)$ denote the graph obtained by taking a disjoint union of infinite rays indexed by the elements of $A$. Let $H(A)$ denote the group of the bijections $R(A)^{(0)} \to R(A)^{(0)}$ that preserve adjacency (and non-adjacency) for all but finitely many pairs of vertices. Because the action of $A$ on itself induces an action of $A$ on $R(A)$ that permutes the rays, one can define the halo group $\mathscr{H}A:= H(A) \rtimes A$. If $A$ is infinite, $\mathscr{H}A$ contains all the so-called Houghton groups.

\medskip \noindent
Finally, as a last example, let us mention how to construct halo groups from Thompson's group $V$. Here we think of $V$ as a group of homeomorphisms of the Cantor space $\mathfrak{C}$ that send a dyadic decomposition to a dyadic decomposition through affine maps. Given an integer $r \geq 1$, let $V_r$ denote the group defined similarly with a disjoint union of $r$ copies of $\mathfrak{C}$. More generally, given a group $H$, let $\mathfrak{C}(H)$ denote the disjoint union of copies of $\mathfrak{C}$ indexed by the elements of $H$; and let $V(H)$ denote the group of the compactly supported homeomorphisms $\mathfrak{C}(H) \to \mathfrak{C}(H)$ that send a dyadic decomposition to a dyadic decomposition through affine maps. Because the action of $H$ on itself induces an action by homeomorphisms on $\mathfrak{C}(H)$ by permuting the copies of $\mathfrak{C}$, one can define the halo product $\mathscr{V}H:= V(H) \rtimes H$.

\paragraph{Non-examples.} In the examples given above, we did not verify the three conditions given by the definition of halos of groups (see Definition~\ref{def:Halo}). The first two conditions are usually satisfied by construction, but the intersection property given by the third condition, which is central in our definition, may be tricky to verify. Let us conclude this section with two examples where this condition is not satisfied.

\medskip \noindent
Recall from \cite{Lampshuffler} the definition of the $\bigvee$-product. Given a collection of groups $(H_i)_{i \in I}$, let $\cup^1_{i \in I} H_i$ denote the pointed sum of the $H_i$, i.e.\ the disjoint union of the $H_i$ glued at their neutral elements. Each group $H_i$ naturally acts on $\cup^1_{i \in I}H_i$ by permuting the points in the copy of $H_i$ and fixing the points in the copies of the $H_j\backslash \{1\}$, $j \neq i$. These subgroups of $\mathrm{Sym}(\cup^1_{i \in I} H_i)$ generate the product $\bigvee_{i \in I}H_i$. Now, given two non-trivial groups $F$ and $H$, consider the semi-direct product $\bigvee_H F \rtimes H$ where $H$ acts on $\bigvee_H F$ through its action on $\cup^1_{h \in H} F$ that permutes the copies of $F$. For all $R,S \subset H$, observe that
$$\bigvee_R F \cap \bigvee_S F = \left\{ \begin{array}{c} \text{permutations fixing pointwise the copies} \\ \text{of $F\backslash \{1\}$ indexed by $H \backslash (R \cup S)$} \end{array} \right\}.$$
According to \cite[Proposition~2.2]{Lampshuffler}, this intersection contains every finitely supported even permutation of a copy of $F$ indexed by $R \cap S$. However, if $R\cap S$ is reduced to a single point, say $p$, such a non-trivial permutation does not belong to $\bigvee_p F$. 

\medskip \noindent
A more interesting example is given by lamplighter groups of higher rank. For instance, consider the group 
$$L_2:= \mathbb{Z}/2\mathbb{Z}[ x, 1/x, 1/(1+x)] \rtimes \mathbb{Z}^2,$$
where one generator of $\mathbb{Z}^2$ acts by multiplication by $x$ and the other by multiplication by $1+x$ \cite{MR299662, Lamp}. An interpretation of $L_2$ as a lamplighter-like group over the $\tilde{A}_2$-tessellation of the plane can be found in \cite{MR3312858}. However, this description does not verify our definition of halo group. Even worse, $L_2$ cannot be described as a halo group. The reason is that $L_2$ is finitely presented (compare with Theorem~\ref{thm:NotFP}) and one-ended, so the embedding theorem provided by Theorem~\ref{Intro:EmbeddingThm} would imply that our halo product would have a quasi-dense pseudo-leaf. Notice also that it easily follows from the geometric description of $L_2$ as a horospherical product of three regular trees \cite{MR2421161} that $L_2$ contains quasi-isometrically embedded copies of $\mathbb{Z}^2$ that are not contained in neighbourhoods of a subgroup $\mathbb{Z}^2$.

\section{The Thick Bigon Property}\label{section:TBP}

\noindent
The embedding theorem proved in Section~\ref{section:EmbeddingTheorem} applies to a family of metric spaces satisfying a specific geometric property, which we call the \emph{thick bigon property}. In this section, we define this property and we give various examples of spaces (and groups) satisfying~it.

\medskip \noindent
First of all, we recall the notion of \emph{coarsely homotopic paths}.

\begin{definition}
Let $X$ be a metric space. Given a constant $C \geq 0$, two paths $\alpha$ and $\beta$ with the same endpoints are \emph{$C$-coarsely homotopic} if there exists a sequence of paths
$$\gamma_0=\alpha, \ \gamma_1, \ldots, \ \gamma_{n-1}, \ \gamma_n=\beta$$
such that, for every $0 \leq i \leq n-1$, $\gamma_{i+1}$ is obtained from $\gamma_i$ by replacing a subpath $\zeta \subset \gamma_{i}$ with a path $\xi$ between the same endpoints satisfying $\mathrm{diam}(\zeta \cup \xi) \leq C$. A path is \emph{$C$-coarsely homotopically trivial} if it is $C$-coarsely homotopic to a constant path. A space is \emph{$C$-coarsely simply connected} if every path is $C$-coarsely homotopically trivial. A space is \emph{coarsely simply connected} if it $C$-coarsely simply connected for some $C \geq 0$. 
\end{definition}

\noindent
It is well-known that a finitely generated group turns to be coarsely simply connected if and only if it is finitely presented.

\medskip \noindent
We are now ready to introduce the central property of this section. Roughly speaking, it requires that any two points of our space can be connected by a path $\gamma$ satisfying the following property: for every ball $B$ centred at a point of $\gamma$ but far from its endpoints, it is possible to deform $\gamma$ in order to get a path coarsely homotopic to $\gamma$ with the same endpoints that avoids the ball $B$. 

\begin{definition}
A metric space $X$ has the \emph{thick bigon property} if there exists some $C \geq 0$ such that, for every $R \geq 0$, the following holds for some $L \geq 0$. Any two points $x,y \in X$ are connected by some path $\gamma_1$ such that, for every $p \in \gamma_1$ satisfying $d(x,p),d(y,p) \geq L$, there exists a new path $\gamma_2$ between $x$ and $y$ that is $C$-coarsely homotopic to $\gamma_1$ and that avoids $B(p,R)$. 
\end{definition}

\noindent
Our first source of examples of spaces satisfying the thick bigon property is given by coarsely simply connected uniformly one-ended spaces. Before proving this claim, recall that:

\begin{definition}
Let $X$ be a geodesic metric space. Given a map $\rho$, $X$ is \emph{$\rho$-uniformly one-ended} if, for every $R \geq 0$, the complement of a ball of radius $R$ has one unbounded connected component and all the other connected components have diameter $\leq \rho(R)$. A space is \emph{uniformly one-ended} if it is $\rho$-uniformly one-ended for some $\rho$.
\end{definition}

\noindent
It is worth noticing that, for graphs of bounded degrees (e.g.\ finitely generated groups), uniform one-endedness and one-endedness are equivalent. 

\begin{lemma}\label{lem:TBP}
Let $X$ be a geodesic metric space. If $X$ is uniformly one-ended and coarsely simply connected, then it satisfies the thick bigon property.
\end{lemma}

\begin{proof}
Fix a constant $C$ and a map $\rho$ such that $X$ is $\rho$-uniformly one-ended and $C$-coarsely simply connected. Given an $R \geq 0$, set $L:=R+ \rho(R)+1$. 

\medskip \noindent
Now, let $x,y \in X$ be two points, which we connect by a geodesic $\gamma_1$. Given a point $p \in \gamma_1$ satisfying $d(p,x),d(p,y) \geq L$, notice that the subpath of $\gamma_1$ connecting $x$ to $B(p,R)$ has length $\geq d(x,p)-R> \rho(R)$. Consequently, $x$ must belong to the unique unbounded connected component of $X \backslash B(p,R)$. The same argument shows that $y$ also belongs to this component. This implies that there exists a path $\gamma_2$ connecting $x$ and $y$ that avoids $B(p,R)$. Because $X$ is $C$-coarsely simply connected, necessarily $\gamma_1$ and $\gamma_2$ are $C$-coarsely homotopic.
\end{proof}

\noindent
As a consequence of Lemma~\ref{lem:TBP}, one-ended finitely presented groups satisfy the thick bigon property. Clearly, being one-ended is necessary, i.e.\ a multi-ended group never satisfies the thick bigon property. However, being finitely presented is not necessary at all. Our main source of such examples comes from the following family of spaces:

\begin{definition}
Let $X$ be a metric space and $\mathcal{P}$ a family of uniformly coarsely embedded subspaces. Then $X$ is \emph{fleshy relative to $\mathcal{P}$} if 
\begin{itemize}
	\item there exist a constant $C$ and a map $\rho$ such that every $P \in \mathcal{P}$ is $C$-coarsely simply connected and $\rho$-uniformly one-ended;
	\item for all $x,y \in X$, there exist $P_1, \ldots, P_n \in \mathcal{P}$ such that $x \in P_1$, $y \in P_n$, and $P_i \cap P_{i+1}$ is unbounded for every $1 \leq i \leq n-1$.
\end{itemize}
A metric space is \emph{fleshy} if it is fleshy relative to some collection of subspaces.
\end{definition}

\noindent
Notice that, given a quasi-isometry $\varphi : X \to Y$ between two spaces, if $X$ is fleshy relative to some collection $\mathcal{P}$, then clearly $Y$ is fleshy relative to $\{ \varphi(P)^{+K}, P \in \mathcal{P}\}$ where $K$ is a constant sufficiently large compared to the parameters of $\varphi$. Therefore, since being fleshy is preserved under quasi-isometry, it makes sense to say that some finitely generated groups are fleshy.

\begin{prop}\label{prop:FleshyThickBig}
Let $X$ be a geodesic metric space. If $X$ is fleshy, then it satisfies the thick bigon property.
\end{prop}

\begin{proof}
Let $C$ and $\rho$ be such that every $P \in \mathcal{P}$ is $C$-coarsely simply connected and $\rho$-uniformly one-ended. Because the subspaces in $\mathcal{P}$ are uniformly coarsely embedded, there exist nondecreasing maps $\sigma_-,\sigma_+$ such that
$$\sigma_-(d(x,y)) \leq d_P(x,y) \leq \sigma_+(d(x,y))$$
for every $P \in \mathcal{P}$ and all $x,y \in P$. Fix an $R \geq 0$ and let $L$ be a constant sufficiently large so that $\sigma_-(L-R) > \rho(\sigma_+(2R)) + \sigma_+(2R)$.

\medskip \noindent
Now, let $x,y \in X$ be two points. We know that there exist $P_1, \ldots, P_n \in \mathcal{P}$ such that $x \in P_1$, $y \in P_n$, and $P_i \cap P_{i+1}$ is unbounded for every $1 \leq i \leq n-1$. For every $1 \leq i \leq n-1$, fix a point $a_i \in P_i \cap P_{i+1}$. We construct a path $\gamma_1$ connecting $x$ to $y$ as a concatenation of segments $[x,a_1] \cup [a_1,a_2] \cup \cdots \cup [a_{n-1},y]$ where $[x,a_1] \subset P_1$, $[a_{n-1},y] \subset P_n$, and $[a_i,a_{i+1}] \subset P_{i+1}$ for every $1 \leq i \leq n-1$. Fix a point $p \in \gamma_1$ satisfying $d(p,x),d(p,y) \geq L$.

\begin{center}
\includegraphics[width=0.6\linewidth]{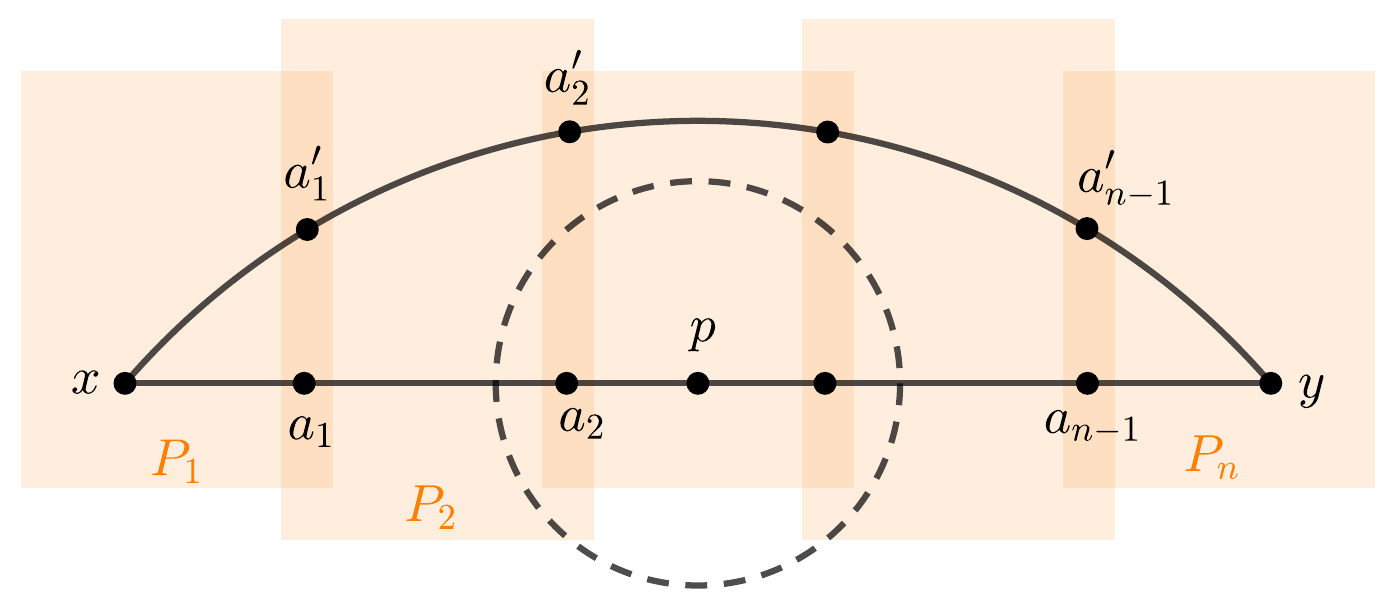}
\end{center}

\begin{claim}\label{claim:PointUnboundedComponent}
The point $x$ (resp.\ $y$) belongs to the unbounded component of $P_1 \backslash B(p,R)$ (resp.\ $P_2 \backslash B(p,R)$).
\end{claim}

\noindent
If $P_1 \cap B(p,R)$ is empty, there is nothing to prove, so we assume from now on that the intersection is non-empty. Fix a point $q \in P_1 \cap B(p,R)$. Because
$$d_{P_1}(q,r) \leq \sigma_+(d(q,r)) \leq \sigma_+(d(q,p) + d(p,r)) \leq \sigma_+(2R)$$
for every $r \in P_1 \cap B(p,R)$, we know that $B(p,R) \cap P_1 \subset B_{P_1}(q,\sigma_+(2R))$. But
$$d_{P_1}(x,q) \geq \sigma_-(d(x,q)) \geq \sigma_-(d(x,p)-d(p,q)) \geq \sigma_-(L-R),$$
so, in $P_1$, $x$ lies from $B_{P_1}(q,\sigma_+(2R))$ at distance at least
$$d_{P_1}(x,q)- \sigma_+(2R) \geq \sigma_-(L-R)- \sigma_+(2R)> \rho(\sigma_+(2R)).$$
This concludes the proof of Claim~\ref{claim:PointUnboundedComponent} for $x$. The same conclusion for $y$ is obtained similarly.

\medskip \noindent
We are now ready to construct our path $\gamma_2$ between $x$ and $y$. If $n=1$, i.e.\ if $x$ and $y$ belong to the same piece of $\mathcal{P}$, then we know from Claim~\ref{claim:PointUnboundedComponent} that $x$ and $y$ both belong to the connected component of $P_1 \backslash B(p,R)$, so there exists a path $\gamma_2 \subset P_1$ from $x$ to $y$ that avoids $B(p,R)$. Moreover, because $P_1$ is $C$-coarsely simply connected, $\gamma_2$ is necessarily $C$-coarsely homotopic to $\gamma_1$. Next, assume that $n \geq 2$. For every $1 \leq i \leq n-1$, because $P_i \cap P_{i+1}$ is unbounded, we can fix a point $a_i' \in P_i \cap P_{i+1}$ that belongs to the unbounded components of $P_i \backslash B(p,R)$ and $P_{i+1} \backslash B(p,R)$. We can fix a path $[a_i',a_{i+1}']$ in $P_{i+1}$ that avoids $B(p,R)$. Moreover, we know from Claim~\ref{claim:PointUnboundedComponent} that there exist paths $[x,a_1']$ and $[a_{n-1}',y]$ that avoid $B(p,R)$. Thus, the concatenation
$$\gamma_2:= [x,a_1'] \cup [a_1',a_2'] \cup \cdots \cup [a_{n-2}',a_{n-1}'] \cup [a_{n-1}',y]$$
avoids $B(p,R)$. Moreover, due to the fact that $P_1, \ldots, P_n$ are all $C$-coarsely simply connected, we know that $\gamma_2$ is $C$-coarsely homotopic to $\gamma_1$. 
\end{proof}

\noindent
Fleshy groups provide examples of one-ended groups that satisfy the thick bigon property but that are not finitely presented. For instance, it is clear that the fundamental group of a graph of groups whose vertex-groups are one-ended finitely presented and whose edge-groups are infinite is fleshy relative to the (thickenings of the) cosets of the vertex-groups. As a simple example, consider an amalgamated product $G:= A \ast_C B$ with $A,B$ one-ended finitely presented and $C$ infinite. Then $G$ is fleshy, so it satisfies the thick bigon property, but it is not finitely presented whenever $C$ is not finitely generated. Other interesting examples include wreath products:

\begin{prop}\label{prop:InfWreathFleshy}
Let $E$ and $H$ be two finitely generated groups. If $E$ is infinite and $H$ non-trivial, then $E \wr H$ is fleshy.
\end{prop}

\begin{proof}
Once for all, we endow $E$ and $H$ with finite generating sets and we endow $E \wr H$ with the generating set that is the union of these two sets (where $E$ is identified with the copy of $E$ in $E \wr H$ indexed by the neutral element of $H$). This allows us to identify $E$, $H$, and $E \wr H$ with their corresponding Cayley graphs.

\medskip \noindent
Notice that $E$ contains bi-infinite geodesics. Indeed, it suffices to fix a geodesic $[a_n,b_n]$ between two points $a_n,b_n \in E$ at distance $n$, to fix a vertex $m_n \in [a_n,b_n]$ at distance $\leq 1/2$ from the midpoint, and to deduce by local finiteness that $m_n^{-1}[a_n,b_n]$ subconverges to a bi-infinite geodesic (passing through the neutral element). Moreover, we have the following elementary observation:

\begin{claim}\label{claim:ChainLines}
For any two bi-infinite geodesic lines $P,Q \subset E$, there exists a sequence
$$L_0:=P, \ L_1, \ldots, L_{k-1}, \ L_k:=Q$$
of bi-infinite geodesic lines such that $L_i^{+1} \cap L_{i+1}^{+1} \neq \emptyset$ for every $0 \leq i \leq k-1$.
\end{claim}

\noindent
Fix a geodesic line $L$ in $E$ and a path $x_1, \ldots, x_k$ in $E$ connecting $P$ to $Q$. Up to translating $L$, we assume that it passes through the neutral element for convenience. Setting $L_i:=x_iL$ for every $1 \leq i \leq k-1$ concludes the proof of Claim~\ref{claim:ChainLines}.

\medskip \noindent
Now, define a \emph{piece} in $E \wr H$ as an $(E \wr H)$-translate of
$$P:= \{ (c,p) \mid \{p\} \cup \mathrm{supp}(c) \subset \mathrm{star}_H(1) \};$$
and a \emph{flat} as an $(E \wr H)$-translate of
$$\{ (c,p) \mid \{p\} \cup \mathrm{supp}(c) \subset \{ h_1,h_2\}, c(h_1) \in L_1^{+1}, c(h_2) \in L_2^{+1} \}$$
where $h_1,h_2 \in H$ are two distinct points and where $L_1,L_2 \subset E$ are two bi-infinite geodesics. Notice that, because $H$ is non-trivial, flats exist. 

\begin{claim}\label{claim:CapGenerators}
For every generator $s$ of $E \wr H$, $P \cap sP$ contains a flat.
\end{claim}

\noindent
If $s$ belongs to $E$, then $P \cap sP=P$, and there is nothing to prove. Otherwise, if $s$ belongs to $H$, then
$$P \cap sP = \{ (c,p) \mid \{p\} \cup \mathrm{supp}(c) \subset \mathrm{star}(1) \cap \mathrm{star}(s) \}.$$
Because $\mathrm{star}(1) \cap \mathrm{star}(s)$ contains at least two distinct points, namely $1$ and $s$, necessarily $P \cap sP$ contains a flat.

\begin{claim}\label{claim:SequenceFlats}
In $P$, any two flats $B,C$ are connected by a sequence of flats
$$F_0:=B, \ F_1, \ldots, F_{k-1}, \ F_k:=C$$
such that $F_i \cap F_{i+1}$ is unbounded for every $0 \leq i \leq k-1$.
\end{claim}

\noindent
Let $a,b$ (resp.\ $p,q$) denote the two points of $H$ given by the flat $B$ (resp.\ $C$), and let $L,M$ (resp.\ $R,S$) denote the two bi-infinite geodesic lines of $E$ given by the flat $B$ (resp.\ $C$). Fix an arbitrary bi-infinite geodesic $I$ in $E$. 

\medskip \noindent
According to Claim~\ref{claim:ChainLines}, there exists a sequence
$$L_0:=M, \ L_1, \ldots, L_{k-1}, \ L_k:=I$$
of bi-infinite geodesic lines such that $L_i^{+1} \cap L_{i+1}^{+1} \neq \emptyset$ for every $0 \leq i \leq k-1$. For every $0 \leq i \leq k$, consider the flat
$$F_i:= \left\{ (c,x) \mid \{x\} \cup \mathrm{supp}(c) \subset \{a,b\}, c(a) \in L^{+1}, c(b) \in L_i^{+1}\right\}.$$
Notice that $F_0=B$ and that any two consecutive flats among $F_0, \ldots, F_k$ have an unbounded intersection (more precisely, the intersection contains a copy of $L$). Consequently, we can assume without loss of generality that $M=I$. By iterating the same argument, we can also assume without loss of generality that $R$, $S$, and $I$ all equals $L$. 

\medskip \noindent
Fix a path $x_{0}:=a, x_1:=b, x_1, \ldots, x_{n-1}:=p,x_n=q$ in $E$ satisfying $x_i \neq x_{i+1}$ for every $0 \leq i \leq n-1$. For every $0 \leq i \leq n-1$, consider the flat
$$F_i':= \left\{ (c,x) \mid \{x\} \cup \mathrm{supp}(c) \subset \{x_i,x_{i+1}\}, c(x_i),c(x_{i+1}) \in I^{+1} \right\}.$$
Notice that $F_{0}'=B$ and $F_{n-1}'=C$, and that any two consecutive flats among $F_0', \ldots, F_{n-1}'$ have an unbounded intersection (more precisely, the intersection contains a copy of $I$). This concludes the proof of Claim~\ref{claim:SequenceFlats}. 

\medskip \noindent
We are now ready to prove that $E \wr H$ is fleshy relative to its flats. Fix two points $x,y \in E \wr H$. Write $x^{-1}y$ as a product of generators and inverses of generators $s_1 \cdots s_n$. According to Claim~\ref{claim:CapGenerators}, each intersection $xs_1 \cdots s_i P \cap xs_1 \cdots s_{i+1}P$ contains a flat $F_i$. Fix a flat $F_{-1}$ in $xP$ containing $x$ and a flat $F_n$ in $yP$ containing $y$. So any two consecutive flats among $F_{-1}, F_0, \ldots, F_n$ belong to the same piece. We know from Claim~\ref{claim:SequenceFlats} that each $F_i$ and $F_{i+1}$ are connected by a sequence of flats such that any two consecutive flats in our sequence have an unbounded intersection. Thus, the same property holds for $F_{-1}$ and~$F_n$. 
\end{proof}

\noindent
Let us conclude this section with a last source of examples of finitely generated groups satisfying the thick bigon property.

\begin{thm}\label{thm:AbSubThick}
Let $G$ be a finitely generated group. If $G$ is not virtually cyclic and if it contains a normal free abelian subgroup of positive rank, then $G$ satisfies the thick bigon property.
\end{thm}

\noindent
Notice that, as an immediate consequence of Theorem~\ref{thm:AbSubThick}, we deduce that:

\begin{cor}
If a finitely generated torsion-free group is solvable but not cyclic, then it satisfies the thick bigon property.
\end{cor}

\noindent
We emphasize that being torsion-free cannot be removed from the assumptions. Indeed, the wreath product $\mathbb{Z}/2\mathbb{Z} \wr \mathbb{Z}^2$ is solvable but it does not satisfy the thick bigon property (since the embedding theorem proved in the next section applies to it). 

\medskip \noindent
In order to prove Theorem~\ref{thm:AbSubThick}, we distinguish two cases, depending on whether the normal free abelian subgroup is finitely generated or not. These two cases are treated by Lemma~\ref{lem:SnormalFleshy} and Proposition~\ref{prop:FgSubThick} below, which are both of independent interest as sufficient conditions for satisfying the thick bigon property.  

\begin{lemma}\label{lem:SnormalFleshy}
Let $G$ be a group endowed with a finite generating set $S$. If $H \leq G$ is a one-ended finitely presented subgroup such that $|H \cap sHs^{-1}|= \infty$ for every $s \in S$, then $G$ is fleshy. 
\end{lemma}

\begin{proof}
Fix two points $x,y \in G$. Write $x^{-1}y$ as a product of generators and inverses of generators $s_1 \cdots s_n$ and consider the sequence of cosets
$$xH, \ xs_1H, \ xs_1s_2H, \ldots, \ x s_1 \cdots s_nH=yH.$$
For every $0 \leq i \leq n-1$, notice that $xs_1 \cdots s_i (H \cap s_{i+1}Hs_{i+1}^{-1})$, which is unbounded as a consequence of our assumptions, lies in the $1$-neighbourhoods of $xs_1 \cdots s_i H$ and $xs_1 \cdots s_{i+1}H$. We conclude that $G$ is fleshy relative to $\{gH^{+1}, g \in G\}$. 
\end{proof}

\noindent
For instance, Lemma~\ref{lem:SnormalFleshy} applies to finitely generated groups containing normal (or more generally, $s$-normal) one-ended finitely presented subgroups. The lemma applies to many semidirect products with infinitely generated kernels. For instance: 

\begin{cor}
Let $F,H$ be two finitely generated groups and $K \leq H$ a subgroup. If $H$ is one-ended finitely presented and $K$ infinite, then the permutational wreath product $F \wr_{H/K} H$ is fleshy.
\end{cor}

\begin{proof}
Fix finite generating sets $R \subset F$ and $S \subset H$. For every $s \in S$, $H \cap sHs^{-1}=H$; and, for every $r \in F$, $H \cap rHr^{-1}$ contains $K$. Therefore, Lemma~\ref{lem:SnormalFleshy} applies and we conclude that $F \wr_{H/K} H$ is fleshy. 
\end{proof}

\noindent
The second step towards the proof of Theorem~\ref{thm:AbSubThick} is the following proposition. We emphasize that, in its statement, the normal subgroup is not assumed to be abelian. 

\begin{prop}\label{prop:FgSubThick}
Let $G$ be a finitely generated group. If $G$ contains an infinite normal subgroup $N \lhd G$ of infinite index that is finitely generated, then $G$ satisfies the thick bigon property.
\end{prop}

\begin{proof}
We endow $G$ with a finite generating set $R \cup S$ such that $\langle R \rangle =N$ and such that $S$ projects to non-trivial generators of $G/N$. Geometrically, we think of $G$ as a disjoint union of copies of $N$ (namely, the cosets of $N$, which we identify with $N$ and which we think geometrically as $\mathrm{Cayl}(N,R)$) connected by edges labelled by generators from $S$. Notice that the quotient map $\pi : G \twoheadrightarrow G/N$ is $1$-Lipschitz (when $G/N$ is endowed with the image of $S$ as its generating set). We refer to a path labelled by generators from $R$ (resp.\ $S$) as \emph{horizontal} (resp.\ \emph{vertical}). 

\begin{claim}\label{claim:VertHoriz}
Let $x \in G$ be a point. For all horizontal path $h$ and vertical path $v$ both starting from $x$, there exist a horizontal path $h'$ and a vertical path $v'$ such that:
\begin{itemize}
	\item $h'$ starts from the terminal point of $v$ and ends at the terminal point of $v'$, and $v'$ starts at the terminal point of $h$;
	\item $v'$ has the same projection as $v$ to $G/N$;
	\item the cycle obtained from the union of $h,h',v,v'$ is $C$-homotopically trivial where $C:=3+ \max \{ \| s^{-1} rs \|_R \mid r \in R, s \in S \}$. 
\end{itemize}
\end{claim}

\noindent
It suffices to consider the case where $v$ is a single case, the general case following by induction. Without loss of generality, we also assume for convenience that $x=1$. Let $n \in N$ denote the terminal point of $h$ and $n=n_1 \cdots n_k$ denote the decomposition of $n$ labelling the path $h$. Let $s \in S$ denote the generator labelling the edge $v$.

\begin{center}
\includegraphics[width=0.4\linewidth]{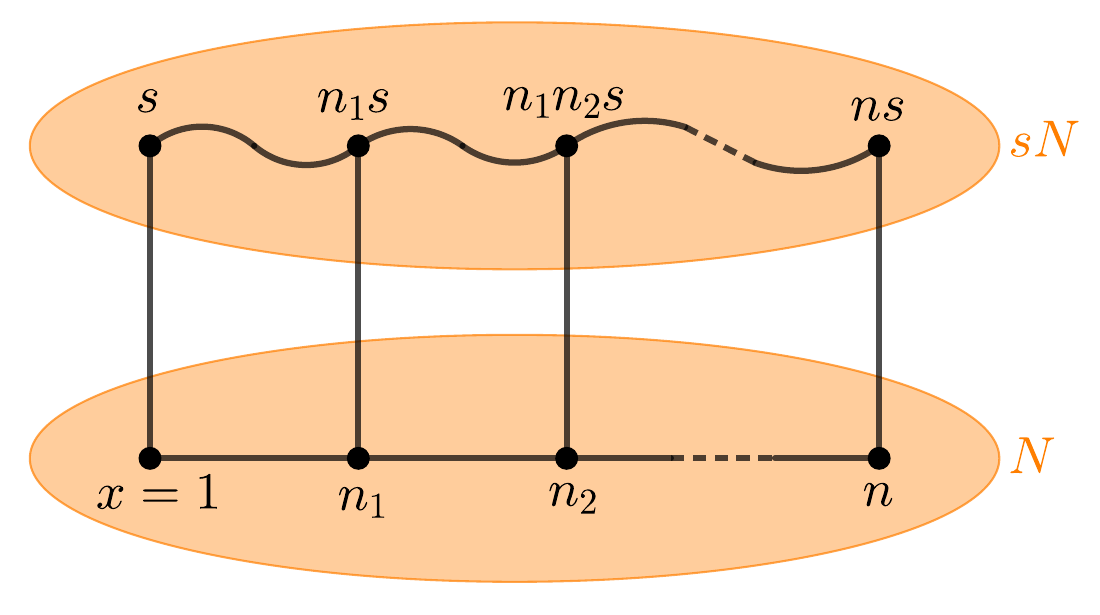}
\end{center}

\noindent
For every $0 \leq i \leq k$, the point $n_1 \cdots n_i$ of $h$ has a neighbour $n_1 \cdots n_i s$ in $sN$. For every $0 \leq i \leq k-1$, let $c_i$ be a horizontal geodesic connecting $n_1 \cdots n_is$ to $n_1 \cdots n_{i+1}s$. Notice that $c_i$ has length $\| s^{-1} n_{i+1} s \|_R \leq C$. Taking for $v'$ the edge connecting $n$ and $ns$, and for $h'$ the concatenation $\alpha_0 \cup \cdots \cup \alpha_{k-1}$ concludes the proof of Claim~\ref{claim:VertHoriz}.

\medskip \noindent
Now, let $x,y \in G$ be two points and $R \geq 0$ a constant. Set $L:=5R+1$. Let $v_0$ be a vertical geodesic connecting $x$ to the $N$-coset containing $y$ and let $h_0$ be a horizontal path connecting the terminal point $z$ of $v_0$ to $y$. Our path $\gamma_1$ connecting $x$ to $y$ is obtained by concatenating $v_0$ and $h_0$. Fix
\begin{itemize}
	\item a horizontal geodesic $h_1$ of length $>R+2 \cdot \mathrm{length}(v_0)$ starting from $x$;
	\item and vertical geodesic $v_2$ of length $>R$ starting from $z$.
\end{itemize}
We apply Claim~\ref{claim:VertHoriz} to $v_0$ and $h_0$ (resp.\ $h_1$ and $v_0$, $h_0$ and $v_2$) in order to get vertical and horizontal paths $v_0'$ and $h_0'$ (resp.\ $v_0''$ and $h_1'$, $v_2'$ and $h_0''$), as illustrated by the following figure:

\begin{center}
\includegraphics[width=0.4\linewidth]{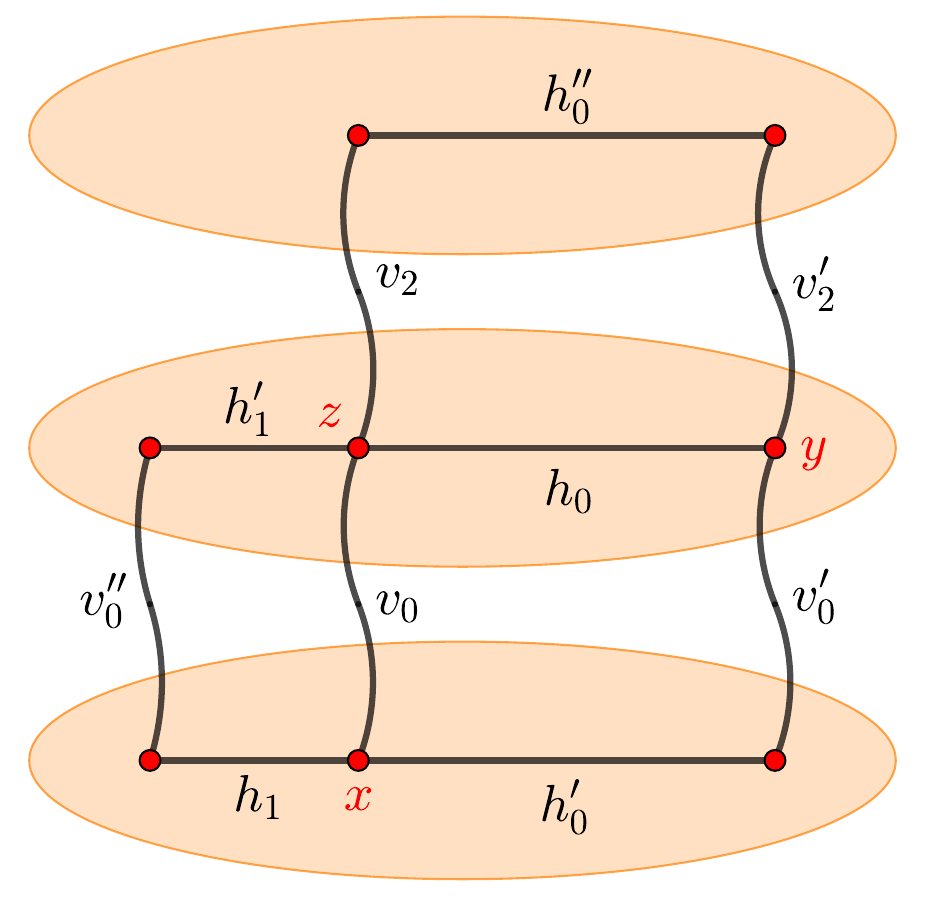}
\end{center}

\noindent
Notice that $v_0',v_0'',v_2'$ are vertical geodesics since they have the same projections as $v_0,v_2$, which are vertical geodesics. Now let $p \in \gamma_1$ be a point satisfying $d(p,x),d(p,y) \geq L$. We distinguish three cases depending on where $p$ is along $\gamma_1$.

\medskip \noindent
First, assume that $p$ belongs to $h_0$ and that it satisfies $d(p,z) > 2R$. If $B(p,R)$ intersects $v_2$, then a point of intersection, say $q$, must lie at distance $\leq R$ from $z$. Indeed,
$$d(z,q)=d(\pi(z),\pi(q)) = d(\pi(p),\pi(q)) \leq d(p,q) \leq R,$$
where the first equality is justified by the fact that $v_2$ is a vertical geodesic. Hence $d(p,z) \leq d(p,q)+d(q,z)\leq 2R$, which is a contradiction. The same argument shows that $B(p,R)$ cannot intersect $v_0$ nor $v_2'$. Notice that $B(p,R)$ cannot intersect $h_0''$ either since $d(h_0,h_0'') \geq d(\pi(h_0),\pi(h_0'')) = \mathrm{length}(v_2) >R$. Thus, the path $\gamma_2:= v_0 \cup v_2 \cup h_0'' \cup v_2'$ avoids $B(p,R)$. Moreover, since we know that $h_0 \cup v_2 \cup h_0'' \cup v_2''$ is $C$-coarsely homotopically trivial, necessarily $\gamma_2$ is $C$-coarsely homotopic to $\gamma_1$. 

\medskip \noindent
Next, assume that $p$ belongs to $v_0$ and that it satisfies $d(p,z) >R$. Since the distance between the projections of a point in $B(p,R)$ and a point in the $N$-coset containing $z$ is at least $d(p,z)-R>0$, necessarily $B(p,R)$ is disjoint from $h_0$ and $h_1'$. The same argument shows that $B(p,R)$ is disjoint from $h_1$. Because
$$d(v_0,v_0'') \geq \mathrm{length}(h_1)- \mathrm{length}(v_0'')- \mathrm{length}(v_0) = \mathrm{length}(h_1)- 2 \cdot \mathrm{length}(v_0)>R,$$
the ball $B(p,R)$ is also disjoint from $v_0''$. Therefore, the path $\gamma_2=h_1 \cup v_0'' \cup h_1' \cup h_0$ avoids $B(p,R)$. Moreover, it is $C$-coarsely homotopic to $\gamma_1$ since $h_1 \cup v_0''\cup h_1' \cup v_0$ is $C$-coarsely homotopically trivial.

\medskip \noindent
Finally, assume that $d(p,z) \leq 2R$. If $B(p,R)$ contains a point in $h_0'$, say $q$, then
$$\begin{array}{lcl} d(p,x) & \leq & d(x,z) +2R = d(\pi(x),\pi(z)) +2R = d(\pi(q),\pi(z)) +2R \\ \\ & \leq & d(\pi(q),\pi(p)) +4R \leq d(p,q) +4R \leq 5R.\end{array}$$
But we know that $d(p,x) \geq L> 5R$, so $B(p,R)$ cannot intersect $h_0'$. If $B(p,R)$ contains a point in $v_0'$, say $s$, then
$$\begin{array}{lcl} d(p,y) & \leq & d(s,y) +R = d(\pi(s),\pi(y))+R \leq d(\pi(p),\pi(y))+2R \\ \\ & \leq & d(\pi(z),\pi(y))+4R = 4R.\end{array}$$
But we know that $d(p,y) \geq L > 4R$, so $B(p,R)$ cannot intersect $v_0'$. We conclude that $\gamma_2:= h_0' \cup v_0'$ avoids $B(p,R)$. Moreover, $\gamma_2$ is $C$-coarsely homotopic to $\gamma_1$ since $\gamma_1 \cup \gamma_2$ is $C$-coarsely homotopically trivial. 
\end{proof}

\begin{proof}[Proof of Theorem~\ref{thm:AbSubThick}.]
Let $N \lhd G$ be a normal free abelian subgroup of positive rank. If $N$ is finitely generated, the desired conclusion follows from Proposition~\ref{prop:FgSubThick}. So we assume from now on that the rank of $N$ is infinite. Fix a finite subset $S \subset G$ that projects under $G \twoheadrightarrow G/N$ to a generating set. Then $G=N \langle S \rangle$. Fix a finite subset $R \subset N$ such that $G= \langle R \rangle \langle S \rangle$ and $N= \langle R^G \rangle$. Up to enlarging $R$ if necessary, we can assume that $\langle R \rangle$ is not cyclic. Set
$$R_+:= R \cup \bigcup\limits_{s \in S} sRs^{-1} \text{ and } H:= \langle R_+ \rangle.$$
Notice that $H \cap rHr^{-1}$ contains $\langle R \rangle$ for every $r \in R$ and that $H \cap sHs^{-1}$ contains $\langle sRs^{-1} \rangle$ for every $s \in S$. Thus, Lemma~\ref{lem:SnormalFleshy} applies and shows that $G$ is fleshy, and a fortiori satisfies the thick bigon property according to Proposition~\ref{prop:FleshyThickBig}. 
\end{proof}

\section{The Embedding Theorem}\label{section:EmbeddingTheorem}

\noindent
We endow each finitely generated halo product with a finite generating set of specific form. In order to avoid repetition, we record our choice once for all. 

\begin{convention}\label{Conv}
Let $\mathscr{L}H$ be a finitely generated halo product. Necessarily, $L(H)$ is finitely generated as an $H$-module, so we can fix a finite generating set $X_{L(H)} \subset L(H)$; and $H$ is finitely generated, so we can fix one of its finite generating sets $X_H$. Since $X_{L(H)}$ is finite, it belongs to $L(B(1,r_0))$ for some $r_0 \geq 0$. For every subset $S \subset H$, we set $L^+(S):=L(S^{+r_0})$. For convenience, we assume that $X_{L(H)}=L^+(1)$ when $L(H)$ is locally finite. We endow $\mathscr{L}H$ with its generating set $X_{L(H)} \cup X_H$. 
\end{convention}

\noindent
The goal of this section is to prove the following embedding theorem:

\begin{thm}\label{thm:EmbeddingThmGeneral}
Let $Z$ be a geodesic metric space satisfying the thick bigon property. For every coarse embedding $\rho : Z \to \mathscr{L}H$, there exist $R \geq 0$, $\varphi_0 \in L(H)$, and $h_0 \in H$ such that $\rho(Z) \subset \left\{ (\varphi,h) \mid \mathrm{supp}(\varphi_0^{-1}\varphi) \subset \{h_0,h\}^{+R} \right\}$ and $(\varphi_0,h_0) \in \rho(Z)$. Moreover, $R$ only depends on $\mathscr{L}H$, $Z$, and the parameters of $\rho$. 
\end{thm}

\noindent
Sections~\ref{section:Cubical} and~\ref{section:Separation} contain preliminaries, and Section~\ref{section:BigProof} contains the proof of the theorem. We refer to the introduction of sketch of proof that illustrates our main arguments. In Section~\ref{section:Full}, we show that, under an additional assumption on our halo of groups, the conclusion of Theorem~\ref{thm:EmbeddingThmGeneral} can be strengthened: the image of the coarse embedding lies in a neighbourhood of an $H$-coset. Examples include lamplighters, lampdesigners, and lampcloners. Sections~\ref{section:SubThm} and~\ref{section:QIinvariants} record applications of the embedding theorem.

\subsection{A cubical model}\label{section:Cubical}

\noindent
In this section, we construct a contractible cube complex $\mathscr{C}$ endowed with a height function such that each subcomplex spanned by the vertices below a fixed height provides a quasi-isometric model of $\mathscr{L}H$. In Definition~\ref{def:CubeComplex} and Proposition~\ref{prop:SimplyConnected}, we do not assume that $L(H)$ is locally finite, i.e.\ we allow $X_{L(H)}=L^+(1)$ from Convention~\ref{Conv} to be infinite. This will be useful in Section~\ref{section:FP}. 

\begin{definition}\label{def:CubeComplex}
Let $\mathscr{C}$ denote the cube complex
\begin{itemize}
	\item whose vertices are the classes $[\varphi,S]$ of pairs $(\varphi,S)$ with $\varphi \in L(H)$ and $S \subset H$ finite, non-empty, and connected, up to the equivalence relation: $(\varphi_1,S_1) \sim (\varphi_2,S_2)$ if $S_1=S_2$ and $\varphi_1L^+(S_1)=\varphi_2L^+(S_2)$ in $L(H)$;
	\item whose edges connect each class $[\varphi,S]$ to $[\varphi, S\cup \{p\}]$ for every $p \in S^{+1} \backslash S$;
	\item and whose $n$-cubes are spanned by vertices $$[\varphi, S \cup \{ p_i, i \in I \}], \ I \subset \{1, \ldots, n\}$$ with $\varphi \in L(H)$ and $p_1, \ldots, p_n \in S^{+1} \backslash S$ pairwise distinct.
\end{itemize}
\end{definition}

\noindent
Here is an intuitive description of our complex (to fix the ideas, let us consider the case of the lamplighter group). The vertices of our complex can be thought of as ``fuzzy'' elements of the group, where the  second coordinate (i.e.\  the lamplighter) is replaced by a finite {\it connected} subset (a crowd of lamplighters). In addition, we consider these up to all lamp modifications that can be made by our crowd, which explains our use of the word ``fuzzy'': the larger the crowd, the fuzzier it gets. Roughly, the process of enlarging the crowd such that it ultimately covers all finite subsets of $H$ can be used to fill any combinatorial sphere. We shall make this idea precise in Proposition~\ref{prop:SimplyConnected} to prove simple connectedness. On the other hand, it is easy to see that, if we impose an upper bound on the size of the crowds, the resulting subcomplex is quasi-isometric to the group. In order to formalise these ideas, we introduce the following notion of height.

\medskip
\noindent
The \emph{height} of a vertex $[\varphi,S]$ is $\# S$. For every $k \geq 1$, we denote by $\mathscr{C}_k$ the subcomplex spanned by the vertices of height $\leq k$. 

\medskip \noindent
A \emph{leaf} in $\mathscr{C}$ (or in $\mathscr{C}_k$ for some $k \geq 1$) is a subcomplex spanned by the vertices
$$[\varphi, S], \ S \subset H$$
for some fixed element $\varphi \in L(H)$. In other words, following our previous intuition, we fix the colouring and only move the ``fuzzy'' lamplighter.

\medskip \noindent
Notice that $\mathscr{L}H$ naturally acts on $\mathscr{C}$ via
$$(\alpha,h) \cdot [\varphi,S] := [\alpha \varphi^h, hS], \ \alpha, \varphi \in L(H), h \in H, S \subset H.$$
Moreover, the action preserves the height function and consequently stabilises the $\mathscr{C}_k$. 

\medskip \noindent
Even though $\mathscr{C}$ is contractible, this fact will not be necessary in full strength. The following statement of simple connectedness will be sufficient for us. 

\begin{prop}\label{prop:SimplyConnected}
For all $k, \ell \geq 0$, a loop of length $\ell$ in $\mathscr{C}_k$ is homotopically trivial in $\mathscr{C}_{k+ \ell}$. Consequently, $\mathscr{C}$ is simply connected.
\end{prop}

\noindent
We begin by proving two preliminary lemmas. 

\begin{lemma}\label{lem:Connected}
For every $k \geq 2$, $\mathscr{C}_k$ is connected. Consequently, the cube complex $\mathscr{C}$ is connected.
\end{lemma}

\begin{proof}
We begin by proving that $\mathscr{C}_2$ is connected. So let $[\varphi,S]$ be a vertex in $\mathscr{C}_2$, which amounts to saying that $S$ is either a single vertex or a single edge. Up to replacing our vertex with one of its neighbours in $\mathscr{C}_2$ obtained by removing a vertex from $S$, we can assume without loss of generality that $S$ is a single vertex. Because $L^+(1)$ generates $L(H)$ as an $H$-module, $L(H)$ is generated by the $L^+(h)$ with $h \in H$, so we can write $\varphi$ as $\varphi_1 \cdots \varphi_r$ where each $\varphi_i$ belongs to $L^+(h_i)$ for some $h_i \in H$. Fix a path in $H$ from $S$ to $h_r$, say $x_0, \ldots, x_\ell$. Then
$$[\varphi, x_0], [\varphi, \{x_0,x_1\}], [\varphi, \{x_1\}], \ldots, [\varphi,x_{\ell-1}], [\varphi, \{x_{\ell-1},x_\ell\}], [\varphi, x_\ell]$$
defines a path in $\mathscr{C}_2$ from $[\varphi, S]$ to $[\varphi,h_r]$. Observe that
$$[\varphi,h_r]= [\varphi_1 \cdots \varphi_r, h_r]= [\varphi_1 \cdots \varphi_{r-1}, h_r]$$
because $\varphi_r \in L^+(h_r)$. By iterating the argument, we are able to connect our vertex $[\varphi,S]$ to $[1,\{1\}]$ by a path in $\mathscr{C}_2$. Thus, every vertex in $\mathscr{C}_2$ can be connected to $[1,\{1\}]$ by a path staying in $\mathscr{C}_2$, proving that $\mathscr{C}_2$ is connected, as desired.

\medskip \noindent
Given an arbitrary $k \geq 2$, every vertex $[\psi,R] \in \mathscr{C}_k$ can be connected to a vertex of $\mathscr{C}_2$ by a path in $\mathscr{C}_k$. Indeed, it suffices to remove step by step the vertices of $R$ (keeping the subgraph connected during the process). Therefore, the connectedness of $\mathscr{C}_k$ follows from the connectedness of $\mathscr{C}_2$. 
\end{proof}

\begin{lemma}\label{lem:HigherNeighbour}
Let $x,y \in \mathscr{C}$ be two adjacent vertices with $y$ higher than $x$. Let $\varphi \in L(H)$ and $S \subset H$ be such that $x=[\varphi,S]$. Then there exists some $p \in H$ such that $y=[\varphi, S \cup \{p\}]$.
\end{lemma}

\begin{proof}
By definition of $\mathscr{C}$, there exist $\psi \in L(H)$, $R \subset H$, and $p \in H$ such that $x=[\psi,R]$ and $y=[\psi, R \cup \{p\}]$. From $[\psi,R]=x=[\varphi,S]$, we know that $R=S$ and that $\varphi L^+(S)= \psi L^+(S)$. It follows from the latter equality that $\varphi L^+(S \cup \{p\})= \psi L^+(S \cup \{p\})$ since $\varphi^{-1} \psi \in L^+(S) \subset L^+(S\cup \{p\})$. Therefore, $y=[\psi, R \cup \{p\}]= [\varphi, S \cup \{p\}]$. 
\end{proof}

\begin{proof}[Proof of Proposition~\ref{prop:SimplyConnected}.]
Given a loop $\gamma$ in $\mathscr{C}$, we define its complexity $\chi(\gamma)$ as
$$(\text{$-$ minimal height $h$ of a vertex in $\gamma$, number of vertices of height $h$ in $\gamma$})$$
ordered lexicographically, and its support as
$$\mathrm{supp}(\gamma):= \bigcup\limits_{[\xi,T] \in \gamma} T \cup \mathrm{supp}(\xi).$$
Let $\gamma_0$ be a loop of length $\ell$ in $\mathscr{C}_k$. Up to translating by an element of $\mathscr{L}H$, we assume that $[1,S_0] \in \gamma_0$ for some $S_0 \subset H$. Fix a vertex $v \in \gamma_0$ of minimal height, which we write as $[\varphi,S]$. If $\gamma_0$ is not reduced to a single vertex, the neighbours of $v$ along $\gamma_0$ ---call them $u$ and $w$--- must have higher height, so it follows from Lemma~\ref{lem:HigherNeighbour} that there exist $p,q \in H$ such that $u= [\varphi,S \cup \{p\}]$ and $w=[\varphi, S \cup \{q\}]$. Replacing $v$ with $v':= [\varphi, S\cup \{p,q\}]$ yields a new loop $\gamma_1$, which is homotopy equivalent to $\gamma_0$ since $u,v,v',w$ span a square. By iterating the process, we get a sequence of loops $\gamma_0, \gamma_1, \ldots$ such that:
\begin{itemize}
	\item $\gamma_{i+1}$ is homotopy equivalent to $\gamma_i$;
	\item $\chi(\gamma_{i+1})< \chi(\gamma_i)$ and $\mathrm{supp}(\gamma_{i+1})= \mathrm{supp}(\gamma_i)$
\end{itemize}
for every $i \geq 0$. Because
$$\chi(\gamma_i) \geq (- \# \mathrm{supp}(\gamma_i), 1)=(- \# \mathrm{supp}(\gamma_0), 1)$$
for every $i \geq 0$, the process has to stop eventually, with a single vertex. Because all the $\gamma_i$ have the same support, it follows that $\gamma_0$ is homotopically trivial in $\mathscr{C}_r$ for every $r \geq \# \mathrm{supp}(\gamma_0)$. But $\gamma_0$ passes through $[1,S_0]$ and has length $\ell$, so $\mathrm{supp}(\gamma_0)$ can be obtained from $S_0$ by adding at most $\ell$ vertices, hence
$$\# \mathrm{supp}(\gamma_0) \leq \# S_0 + \ell \leq k+\ell.$$
Therefore, we can take $r=k+ \ell$, as desired. 

\medskip \noindent
We know from Lemma~\ref{lem:Connected} that $\mathscr{C}$ is connected and we have just proved that every loop in $\mathscr{C}$ is homotopically trivial. Therefore, $\mathscr{C}$ is simply connected. 
\end{proof}

\noindent
From now on, we assume that $L(H)$ is locally finite, i.e.\ that $X_{L(H)}=L^+(1)$ is finite. Finally, let us show that our complexes $\mathscr{C}_k$ are quasi-isometric models of $\mathscr{L}H$. 

\medskip \noindent
Let $\Psi : \mathrm{Cayl}(\mathscr{L}H, X_{L(H)} \cup X_H) \to \mathscr{C}$ be the map that sends
\begin{itemize}
	\item the vertex $(\varphi, h)$ to the vertex $[\varphi, \{h\}]$;
	\item the edge connecting $(\varphi, h)$ and $(\varphi ,h)(1, b)$ to the path $( [\varphi, \{h\}], [\varphi, \{h, hb\}], [\varphi, \{hb\}])$ for every $b \in X_H$;
	\item the edge connecting $(\varphi ,h)$ and $(\varphi, h) ( a, 1)$ to the vertex $[\varphi, \{h\}] = [\varphi hah^{-1}, \{h\}]$ for every $a \in X_{L(H)}$. 
\end{itemize}
Notice that the image of $\Psi$ coincides with $\mathscr{C}_2$.

\begin{lemma}\label{lem:ApproxQI}
For every $k \geq 2$, the map $\Psi : \mathscr{L}H \to \mathscr{C}_k$ is a quasi-isometry. 
\end{lemma}

\begin{proof}
Because $\Psi$ sends an edge of $\mathscr{L}H$ to a single vertex or to a path of length two, we know that
$$d(\Psi(a),\Psi(b)) \leq 2 d(a,b), \text{ for all } a,b \in \mathscr{L}H.$$
Next, define
$$\Phi : \left\{ \begin{array}{ccc} \mathscr{C}_k & \to & \mathfrak{P}(\mathscr{L}H) \\ \left[ \varphi,S \right] & \mapsto & (\varphi L^+(S),S) \end{array} \right.,$$
where $(\varphi L^+(S),S)$ is a shorthand for $\{ (\psi, h) \mid \psi \in \varphi L^+(S), h \in S\}$. Setting 
$$M_k:= \max \{ \mathrm{diam}~L^+(S) \mid S \subset H \text{ non-empty, connected, of cardinality $\leq k$}\},$$
where the diameter is taken with respect to the word length of $L(H)$ given by the generating set $\{ L^+(h) \mid h \in H\}$, observe that
$$\mathrm{diam}~\Phi([\varphi,S]) \leq 2\#S + M_k \leq 2k+M_k$$
for every $[\varphi,S] \in \mathscr{C}_k$. Moreover, given two adjacent vertices $[\varphi,S]$ and $[\varphi, S\cup \{p\}]$ in $\mathscr{C}_k$, we have $\Phi([\varphi,S]) \subset \Phi([\varphi,S \cup \{p\}])$, so the Hausdorff distance between $\Phi([\varphi,S])$ and $\Phi([\varphi,S\cup \{p\}])$ is $\leq 2k+M_k$. 

\medskip \noindent
As a consequence, if we choose some $\Theta(x) \in \Phi(x)$ for every $x \in \mathscr{C}_k$, then $d(\Theta(x),\Theta(y)) \leq 2k+M_k$ for all adjacent vertices $x,y \in \mathscr{C}_k$. This implies that $\Theta : \mathscr{C}_k \to \mathscr{L}H$ is $(2k+M_k)$-Lipschitz. For every point $x:=(\varphi,h) \in \mathscr{L}H$, $\Phi \circ \Psi (\varphi,h)= \Phi([\varphi,\{h\}]) = (\varphi L^+(h), \{h\})$ contains $(\varphi,h)$ and has diameter $\leq 2k+M_k$, so $d(x, \Theta \circ \Psi(x)) \leq 2k+M_k$. We conclude that
$$ d(a,b)  \leq d(\Theta \circ \Psi(a), \Theta \circ \Psi(b)) +2(2k+M_k) \leq (2k+M_k) (d(\Psi(a),\Psi(b))+2)$$
for all $a,b \in \mathscr{L}H$, proving that $\Psi$ is a quasi-isometric embedding. Because every vertex in $\mathscr{C}_k$ lies at distance $\leq k-2$ from $\mathscr{C}_2$, which coincides with the image of $\Psi$, we obtain that $\Psi$ is a quasi-isometry as desired. 
\end{proof}

\subsection{Essential separation}\label{section:Separation}

\noindent
A key step in the proof of Theorem~\ref{thm:EmbeddingThmGeneral} is the existence of separating subspaces in the universal covers of the $\mathscr{C}_k$. Since it will be more convenient to work in the $\mathscr{C}_k$ themselves instead of their universal covers, we introduce the notion of \emph{essential separation}. 

\begin{definition}
Let $X$ be a topological space and $Y \subset X$ a subspace. A path $\alpha$ intersects $Y$ \emph{essentially} if every path homotopy equivalent to $\alpha$ relative to its endpoints intersects $Y$ as well. 
\end{definition}

\noindent
The connection with actual separation is made by the following statement. Given a universal cover $\pi:\widetilde{X}\to X$ and a connected subcomplex $Z$, we shall call a ``lift'' of $Z$ any connected component of $\pi^{-1}(Z)$.

\begin{lemma}\label{lem:UnivCover}
Let $X$ be a cellular complex, $\widetilde{X}$ its universal cover, and $Y \subset X$ an induced subcomplex. If a combinatorial path $\alpha$ connecting two vertices not in $Y$ intersects $Y$ essentially, then, for every lift $\widetilde{\alpha}$ of $\alpha$ in $\widetilde{X}$, there is a lift of some connected component of $Y$ that separates the endpoints of $\widetilde{\alpha}$. 
\end{lemma}

\begin{proof}
Up to modifying $\alpha$ up to isotopy, we assume that the number of times $\alpha$ crosses $Y$ is minimal in the homotopy class of $\alpha$. (Here, the number of crossings refers to the number of edges along our path with one endpoint in $Y$ and the other outside $Y$.) We denote by $x$ and $y$ the endpoints of $\alpha$ and we orient $\alpha$ from $x$ to $y$. Let $z$ denote the first vertex of $\alpha$ not in $Y$ that follows a vertex in $Y$ and $\beta$ the initial segment of $\alpha$ between $x$ and $z$. By construction, $\beta$ crosses $Y$ exactly once; we denote by $Y_0$ the connected component of $Y$ crossed by $\beta$. Let $\widetilde{z}$ denote the lift of $z$ in $\widetilde{\alpha}$ and $\widetilde{Y_0}$ the lift of $Y_0$ crossing the lift $\widetilde{\beta}$ of $\beta$ in $\widetilde{\alpha}$. 

\begin{center}
\includegraphics[width=0.5\linewidth]{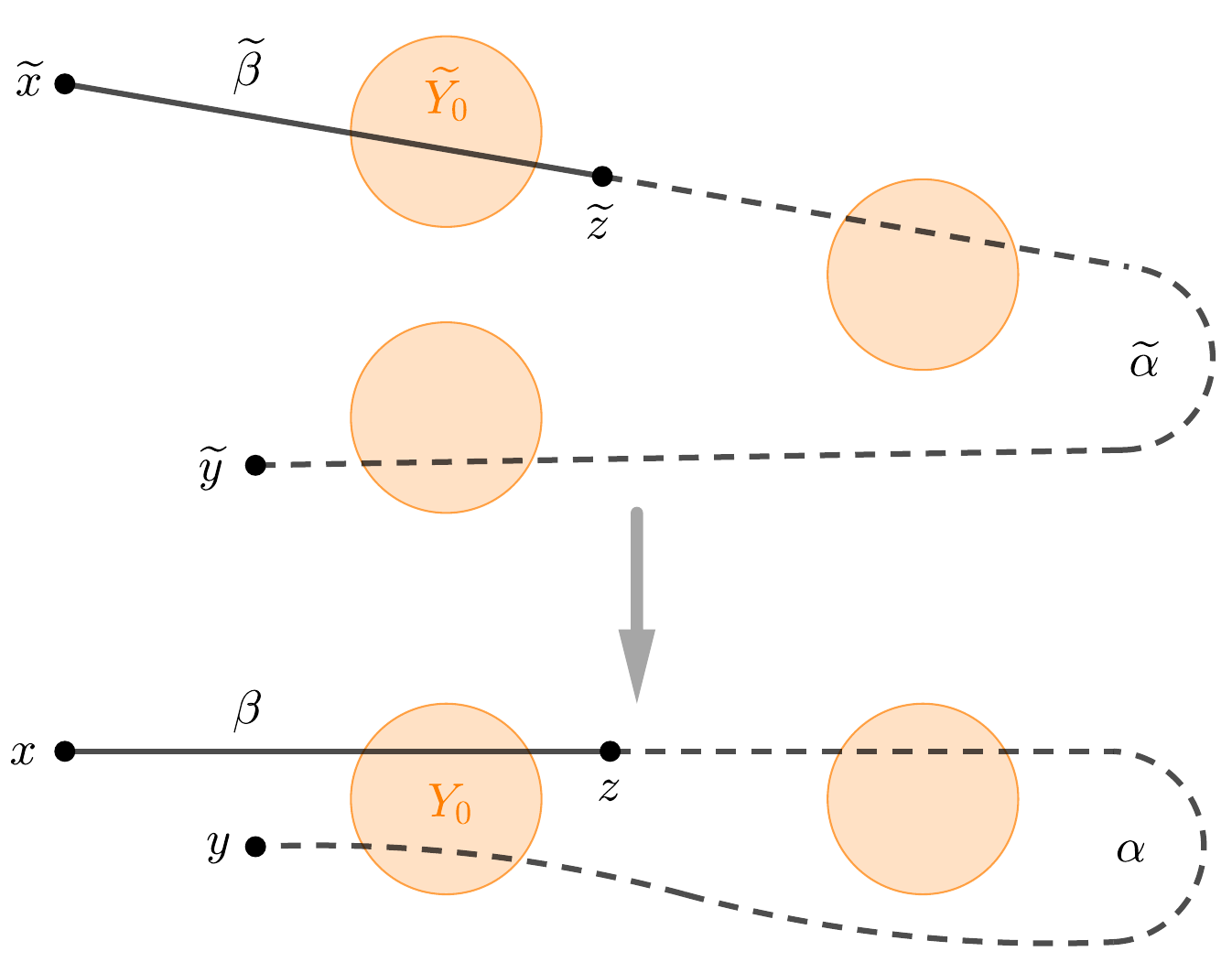}
\end{center}

\noindent
We claim that $\widetilde{Y_0}$ separates $\widetilde{x}$ and $\widetilde{z}$ in $\widetilde{X}$. Otherwise, there exists a path $\widetilde{\gamma}$ in $\widetilde{X}$ connecting $\widetilde{x}$ and $\widetilde{z}$ but disjoint from $\widetilde{Y_0}$. The image $\gamma$ of $\widetilde{\gamma}$ in $X$ provides a path between $x$ and $z$ homotopy equivalent to $\beta$. Let $\Delta \to X$ be a combinatorial disc bounded by $\beta \cup \gamma$ and let $\Theta$ denote the pre-image of $Y$ in $\Delta$. Because $\beta$ intersects $Y$ essentially (otherwise we would be able to decrease the number of times $\alpha$ crosses $Y$), the subcomplex $\Theta$ must separate the common endpoints of $\beta$ and $\gamma$ in $\Delta$, which implies that there exists a combinatorial arc contained in $\Theta$ connecting $\beta$ and $\gamma$ in $\Delta$. Its image in $X$ under $\Delta \to X$ provides a combinatorial path contained in $Y$ connecting $\beta$ and $\gamma$ such that the concatenation of this path with the two corresponding initial segments of $\beta$ and $\gamma$ is a homotopically trivial loop. This implies that the lift $\widetilde{\gamma}$ must intersect $\widetilde{Y_0}$, a contradiction. This concludes the proof of our claim.

\medskip \noindent
In order to conclude, it suffices to observe that the lifts of $Y_0$ crossed by $\widetilde{\alpha}$ are pairwise distinct. Indeed, this implies that the subsegment of $\widetilde{\alpha}$ between $\widetilde{z}$ and $\widetilde{y}$ is disjoint from $\widetilde{Y_0}$, so the fact that $\widetilde{Y_0}$ separates $\widetilde{x}$ and $\widetilde{z}$ immediately implies that $\widetilde{Y_0}$ separates $\widetilde{x}$ and $\widetilde{y}$ as well. So assume for contradiction that $\widetilde{\alpha}$ crosses twice the same lift of $Y_0$. Let $\widetilde{\delta}$ denote a subsegment of $\widetilde{\alpha}$ not contained in a lift of $Y_0$ but with its two endpoints in the same lift $\overline{Y_0}$ of $Y_0$. Because $Y_0$, and a fortiori its lifts, is connected, $\widetilde{\delta}$ is homotopically equivalent to a path in $\overline{Y_0}$. Therefore, the image $\delta \subset \alpha$ of $\widetilde{\delta}$ is homotopically equivalent to a path contained in $Y_0$, but this contradicts the minimality of the number of times $\alpha$ crosses~$Y$. 
\end{proof}

\noindent
Let us define the specific subcomplexes of our $\mathscr{C}_k$ which will essentially separate.

\begin{definition}\label{def:block}
Given an integer $k \geq 1$, a subset $X \subset H$, and a coset $c \in L(H)/L^+(X)$, the subcomplex (spanned by the vertices in)
$$\left\{ [\varphi,S] \in \mathscr{C}_k \mid S \subset X, \varphi L^+(X) = c\right\}$$
is an \emph{$X$-block} of $\mathscr{C}_k$. Given an integer $d \geq 0$, a \emph{$d$-block} is an $X$-block for some $X \subset H$ of diameter $\leq d$. 
\end{definition}

\noindent
Notice that a $d$-block is always bounded, with an upper bound on its diameter depending only on $\mathscr{L}$ and $d$. Let us conclude the section by showing how to obtain some essential separation from blocks.

\medskip \noindent
Once again, let us start with an intuitive description of the situation. To simplify, we consider a lamplighter group $\mathbb{Z}/2\mathbb{Z} \wr H$. Fix a path $\gamma$ from the identity to some point $(c,h)$. Also, fix a point $p$ at which $c$ is non-trivial. Along $\gamma$, the arrow has to move from $1$ to $h$ and to turn on the lamps at each point in the support of $c$. In particular, we know that, at some point, the arrow has to point to $p$. The same observation applies to any path $\gamma'$ connecting $1$ to $(c,h)$. Now, compare the colourings outside $p$ at the points of $\gamma$ and $\gamma'$ where the arrows point to $p$. A possibility for these two colourings to be different in a significative way is that orders according to which the lamps of $\mathrm{supp}(c)$ are turned on along $\gamma$ and $\gamma'$ are very different. For instance, there exists $q \in \mathrm{supp}(c)$ far from $p$ such that $\gamma$ turns on $q$ before $p$ while $\gamma'$ turns on $q$ after $p$. But the cycle in $\mathbb{Z}/2\mathbb{Z} \wr H$ that turns on $p$, turns on $q$, turns off $p$, and finally turns off $q$ is (coarsely) homotopically trivial only at a very large scale (because $p$ and $q$ are very far apart). Therefore, at a small scale, the paths $\gamma$ and $\gamma'$ cannot be (coarsely) homopoty equivalent. Thus, if we are only allowed to modify $\gamma$ up to (coarse) homotopy (at some fixed scale), then, along our new path, we can always find a point where the arrow points to $p$ and where the colouring outside $p$ is fixed. 

\begin{lemma}\label{lem:Separating}
Let $k \geq 1$ be an integer and $[\varphi,R],[\psi,S] \in \mathscr{C}_k$ two vertices. For every subset $V \subset H$ satisfying
\begin{itemize}
	\item $V \cap \mathrm{supp}(\varphi^{-1}\psi) \neq \emptyset$;
	\item $d(V,R),d(V,S) > r_0$,
\end{itemize}
every path connecting $[\varphi,R]$ and $[\psi,S]$ in $\mathscr{C}_k$ intersects essentially a $W$-block, where $W:= V^{+r_0+k}$. 
\end{lemma}

\begin{proof}
Set $U:=V^{+r_0}$ and $W:= U^{+k}$. For every coset $c \in L(H)/ L^+(W)$, set
$$\mathcal{W}(c):= \{ [ \xi,T] \mid T \subset W, \xi L^+(W)=c \};$$
and for every coset $c \in L(H)/ L^+(U^c)$, set
$$\mathcal{U}(c):= \{ [\xi,T] \mid T \subset U^c, \xi L^+(U^c)=c \}.$$
Let us verify that
$$\mathcal{O}:= \{ \mathcal{W}(c), c \in L(H) / L^+(W) \} \sqcup \{ \mathcal{U}(c), c \in L(H)/L^+(U^c) \}$$
is a covering of $\mathscr{C}_k$ whose nerve complex is a graph. The fact that the nerve complex of $\mathcal{O}$ is a graph is clear because no two distinct $\mathcal{W}(\cdot)$ (resp. $\mathcal{U}(\cdot)$) intersect. Next, let $C$ be a cube in $\mathscr{C}_k$. So there exist an element $\zeta \in L(H)$, a finite subset $S \subset H$, and points $p_1, \ldots, p_r \in H$ such that
$$[\zeta, S \cup \{ p_i, i \in I \} ], \ I \subset \{1, \ldots, r\}$$
are the vertices of $C$. Because $S \cup \{p_1, \ldots, p_r\}$ has diameter $\leq k$, it must lie in $W$ or $U^c$, which implies that $C$ is contained in $\mathcal{W}(\zeta L^+(W))$ or $\mathcal{U}(\zeta L^+(U^c))$, proving our assertion. 

\medskip \noindent
Observe that, because $V$ is sufficiently far from $R$ and $S$, $R$ and $S$ must lie in $U^c$, so there exist cosets $a,b \in L(H)/L^+(U^c)$ such that $[\varphi,R] \in \mathcal{U}(a)$ and $[\psi,S] \in \mathcal{U}(b)$. Observe that $\mathcal{U}(a) \neq \mathcal{U}(b)$. Indeed, otherwise we would have $\varphi L^+(U^c)= \psi L^+(U^c)$, hence $\mathrm{supp}(\varphi^{-1} \psi) \subset (U^c)^{+r_0} \subset V^c$, contradicting our assumptions.

\medskip \noindent
As a consequence of the previous observation, given a path $\alpha$ in $\mathscr{C}_k$ connecting $[\varphi,R]$ and $[\psi,S]$, the path in the nerve graph of $\mathcal{O}$ induced by $\alpha$ must intersect essentially some vertex $\mathcal{W}(c)$, which implies that $\alpha$ intersects essentially in $\mathscr{C}_k$ the $W$-block $\mathcal{W}(c)$.
\end{proof}

\subsection{Proof of the embedding theorem}\label{section:BigProof}

\noindent
Thanks to the geometric model constructed in Section~\ref{section:Cubical} and to the essential intersections exhibited in Section~\ref{section:Separation}, we are now able to prove our embedding theorem. Recall from Convention~\ref{Conv} that our finitely generated halo product is endowed with a specific finite generating set.

\begin{proof}[Proof of Theorem~\ref{thm:EmbeddingThmGeneral}.]
Set $\phi := \Psi \circ \rho : Z \to \mathscr{C}$, where $\mathscr{C}$ is the cube complex constructed in Section~\ref{section:Cubical}. Up to perturbing $\rho$, we can assume without loss of generality that $\phi$ is continuous, which implies that $\phi(Z)$ is path connected in $\mathscr{C}$. Let $C$ be the constant given by the thick bigon property satisfied by $\phi(Z)$. It follows from Proposition~\ref{prop:SimplyConnected} that there exists some $k \geq 1$ such that two $C$-coarsely homotopic paths in $\phi(Z)$ are homotopic in $\mathscr{C}_k$. From now on, we focus on the embedding $Z \to \mathscr{C}_k$, which is also coarse according to Lemma~\ref{lem:ApproxQI} and which we also denote by $\phi$ for convenience. Notice that the parameters of $\phi$ only depends on $\mathscr{L}H$, $Z$, and the parameters of $\rho$. 

\medskip \noindent
By construction, $\phi(Z) \subset \mathscr{C}_k$ satisfies the following topological version of the thick bigon property. For every $R \geq 0$, there exists some $L \geq 0$ such any two points $x,y \in \phi(Z)$ are connected by some path $\alpha$ satisfying the following condition. For every $p \in \alpha$ satisfying $d(p,x),d(p,y) \geq L$, there exists a path homotopic to $\alpha$ connecting $x$ and $y$ that avoids $B(p,R)$. Let $R$ be an arbitrary constant larger than the diameters of blocks over balls of radius $r_0+k$ and let $L$ be the constant given by the property just mentioned.

\medskip \noindent
Now, assume for contradiction that the conclusion of our theorem does not hold. Thus, fixing an arbitrary point $(\varphi_0,h_0) \in \rho(Z)$, there exists some $(\varphi,h) \in \rho(Z)$ such that $\mathrm{supp}(\varphi_0^{-1} \varphi) \nsubseteq \{h_0,h\}^{+r_0+k+L}$. Fix a point $p$ that belongs to $\mathrm{supp}(\varphi_0^{-1} \varphi)$ but not to $\{h_0,h\}^{+r_0+k+L}$. Consider the two points $[\varphi_0, \{h_0\}]$ and $[\varphi, \{h\}]$ of $\phi(Z)$, and let $\alpha \subset \phi(Z)$ be the path given by the thick bigon property stated in the previous paragraph. According to Lemma~\ref{lem:Separating}, there exists a $W$-block $\mathcal{O}$ intersected essentially by $\alpha$ where $W:= \{p\}^{+r_0+k}$. But we have
$$d([\varphi,\{h\}], \mathcal{O}) \geq d(h,p)-r_0-k >C,$$
and similarly
$$d([\varphi_0,\{h_0\}], \mathcal{O}) \geq d(h_0,p) -r_0-k >C.$$
So it should be possible to homotope $\alpha$ in order to get a path avoiding $\mathcal{O}$, contradicting the fact that $\alpha$ intersects $\mathcal{O}$ essentially. 
\end{proof}

\subsection{The embedding theorem for full halos}\label{section:Full}

\noindent
Notice that, in the conclusion of Theorem~\ref{thm:EmbeddingThmGeneral}, if $\rho(Z) \subset \left\{ (\varphi,h) \mid \mathrm{supp}(\varphi_0^{-1}\varphi) \subset \{h\}^{+R} \right\}$, then $\rho(Z)$ is contained in a neighbourhood of the coset $\varphi_0 H$. This is the case we are the most interesting in. However, as shown by the following example, this may not be the case, even in very simple situations. 

\begin{ex}\label{ex:NotInALeaf}
Let $G$ denote the wreath product $\mathbb{Z}/2\mathbb{Z} \wr H$, where $H$ is an arbitrary infinite finitely generated group; and let $L \rtimes H$ be the index-two subgroup of $G$ where $L = \{ (a_h)_{h \in H} \mid \sum_{h \in H}a_h = 0 \}$. What happens here is that the $H$-coset $\delta_1H$ of $G$, where $\delta_1$ denotes the lamp configuration with a unique lighted lamp at $1 \in H$, is disjoint from $L \rtimes H$.  But because the latter has finite index in $G$, this ``ghost leaf'' provides a quasi-isometric embedding of $H$ into $L \rtimes H$. Since pairs of $H$-cosets in $G$ diverge from one another, it follows that this leaf is not at bounded distance from any other $H$-coset of $G$. This implies that our ghost leaf in $L \rtimes H$ does not lie at bounded distance from any $H$-coset in $L \rtimes H$.

\medskip \noindent 
A concrete way of describing this ghost leaf in $L \rtimes H$ is as follows: consider the injective morphism $\rho : H \to L \rtimes H$ defined by $\rho(h)=\left( \delta_1+\delta_h, \ h \right)$, whose image is the conjugate of $H$ by $\delta_1$ in $\mathbb{Z}/2\mathbb{Z} \wr H$. We claim that $\rho(H)$ does not lie in the neighbourhood of an $H$-coset. 
To see this, notice that, given two points $h_1,h_2 \in H$ very far apart and very far away from $1$, if $\rho(h_1)$ and $\rho(h_2)$ lies in a small neighbourhood of some coset $x H$, then the colouring $x$ must be supported at the same time in $\{1\} \cup \{h_1\}^{+\mathrm{cst}}$ and $\{1\} \cup \{h_2\}^{+ \mathrm{cst}}$. Consequently, the support of $x$ must be contained in $\{1\}$, and so be trivial since it cannot be equal to $\{1\}$ according to our definition of $L$. Therefore, the only $H$-coset that may contain $\rho(H)$ in some of its neighbourhoods is $H$. But $d(\rho(h),H) \geq d(h,1)$ for every $h \in H$, i.e.\ a point in $\rho(H)$ can be arbitrarily far from $H$.
\end{ex}

\noindent
In this section, we introduce specific halos of groups for which the phenomenon exhibited by Example~\ref{ex:NotInALeaf} does not occur.

\begin{definition}\label{def:Cool}
A halo of groups $\mathscr{L}$ on a metric space $X$ is \emph{full} if there exists some $K \geq 0$ such that, for all subsets $R,S,T,U \subset X$, if 
$$d(R,S),d(R,T), d(R,U), d(S,T) \geq K$$ 
then
$$L(R \cup S) L(U) \cap L(R \cup T) \subset L(R) L(T).$$
\end{definition}

\noindent
The intuition behind the definition is the following. Fix four zones $R,S,T,U \subset X$ such that $R,S,T$ are pairwise far apart and such that $U$ is far from $R$ (but it may intersect $S$ or $T$). Now, let $\varphi$ be a colouring supported on $R \cup T$ and assume that we can transform $\varphi$ into a colouring supported on $R \cup S$ just by right-multiplying by a colouring supported on $U$ (so very far from $R$). Then the colouring $\varphi$ has to decompose as a product of a colouring supported on $R$ with a colouring supported on $T$, and the previous operation made on $\varphi$ was just the replacement of the colouring supported on $T$ with a colouring supported on $S$. 

\begin{lemma}\label{lem:PseudoLeaves}
Let $\mathscr{L}H$ be a finitely generated halo product with $\mathscr{L}$ full and $H$ infinite. For all $(\varphi_0,h_0) \in \mathscr{L}H$ and $R \geq 0$, there exist $x \in \mathscr{L}H$ and $T \geq 0$ such that
$$\left\{ (\varphi,h) \in \mathscr{L}H \mid \mathrm{supp}(\varphi_0^{-1} \varphi) \subset \{h_0,h\}^{+R} \right\}$$
lies at finite Hausdorff distance from a union of $H$-cosets all intersecting the ball $B(x,T)$. Moreover, $T$ and the Hausdorff distance depend only on $\mathscr{L}H$ and $R$.
\end{lemma}

\begin{proof}
Let $\Xi$ be the subset from our statement. Notice that $\Xi$ contains $\varphi_0H$, so $\Xi$ is unbounded. Fix a constant $D \geq 1$ sufficiently large compared to $R$, $r_0$, and the constant given by Definition~\ref{def:Cool}. Because
$$\{( \varphi,h) \in \Xi \mid d(h,h_0) \leq K\}$$ 
is bounded for every $K \geq 0$ (with a bound depending only on $K$, $R$, and $\mathscr{L}H$), $\Xi$ can be decomposed as a union of a bounded set $\Theta$ with some $\Omega$ satisfying $d(h,h_0) \geq D$ for every $(\varphi,h) \in \Omega$. (Here, the diameter of $\Theta$ depends only on $\mathscr{L}H$ and $R$.) Let $\Omega_0$ be a path-connected component of $\Omega$.

\begin{claim}\label{claim:FullOne}
For every $(\varphi,h) \in \Omega_0$, we can write $\varphi = \varphi_0 c(\varphi) r(\varphi)$ for some $c(\varphi) \in L(\{h_0\}^{+R})$ and $r(\varphi) \in L(\{h\}^{+R})$.
\end{claim}

\noindent
Because $\Xi$ is unbounded, there must exist some $(\psi,k) \in \Omega_0$ such that $d(h,k) \geq D$. Because there exists a path in $\Omega$ connecting $(\psi,k)$ and $(\varphi,h)$, necessarily
$$\varphi \in \psi L \left( H \backslash \{h_0\}^{+D-r_0} \right),$$
Since $\varphi_0^{-1}\psi\in L \left( \{ h_0\}^{+R} \cup \{k\}^{+R} \right),$
we deduce that
$$\varphi_0^{-1} \varphi \in L \left( \{h_0\}^{+R} \cup \{h\}^{+R} \right) \cap L \left( \{ h_0\}^{+R} \cup \{k\}^{+R} \right) L \left( H \backslash \{h_0\}^{+D-r_0} \right).$$
Because we chose $D$ sufficiently large, it follows from the fact that our halo of groups is full that $\varphi_0^{-1} \varphi \in L\left( \{h_0\}^{+R} \right) L \left( \{h\}^{+R} \right).$ This concludes the proof of our claim.

\begin{claim}\label{claim:FullTwo}
For all $(\varphi_1,h_1),(\varphi_2,h_2) \in \Omega_0$, $c(\varphi_1)=c(\varphi_2)$.
\end{claim}

\noindent
Because there exists a path in $\Omega$ connecting $(\varphi_1,h_1)$ and $(\varphi_2,h_2)$, we must have $\varphi_2 \in \varphi_1 L \left( H \backslash \{h_0\}^{+D-r_0} \right)$, hence
$$c(\varphi_1)^{-1} c(\varphi_2) \in r(\varphi_1) L \left( H \backslash \{h_0\}^{+D-r_0} \right) r(\varphi_2)^{-1} \subset L \left( H \backslash \{h_0\}^{+R} \right),$$
where the inclusion holds because we chose $D$ sufficiently large. Thus, we have
$$c(\varphi_1)^{-1}c(\varphi_2) \in L \left( \{h_0\}^{+R} \right) \cap L \left( H \backslash \{h_0\}^{+R} \right) = L (\emptyset)= \{1\},$$
proving that $c(\varphi_1)= c(\varphi_2)$ as desired. 

\medskip \noindent
Claims~\ref{claim:FullOne} and~\ref{claim:FullTwo} imply that
$$\Xi \subset \Theta \cup \bigcup\limits_{\mathrm{supp}(c) \subset \{h_0\}^{+R}} \left\{ (\varphi_0cr ,h) \mid \mathrm{supp}(r) \subset \{h\}^{+R} \right\}.$$
Consequently, $\Xi$ lies in a neighbourhood of the union of the cosets $\varphi_0c H$ with $\mathrm{supp}(c) \subset \{h_0\}^{+R}$ (and the size of the neighbourhood only depends on $\mathscr{L}H$ and $R$). Notice that all these cosets intersect some ball centred at $(\varphi_0,h_0)$ (of radius depending only on $\mathscr{L}H$ and $R$). Conversely, it is clear that each $\varphi_0cH$ lies in $\Xi$, concluding the proof of our lemma.
\end{proof}

\noindent
Under the additional assumption that our halo of groups is full, Lemma~\ref{lem:PseudoLeaves} allows us to deduce the following improved embedding theorem:

\begin{thm}\label{thm:InALeaf}
Let $Z$ be a geodesic metric space satisfying the thick bigon property and let $\mathscr{L}H$ be a finitely generated halo product with $\mathscr{L}$ full and $L(H)$ locally finite. Every coarse embedding $\rho : Z \to \mathscr{L}H$ has its image contained in a neighbourhood of an $H$-coset. Moreover, the size of this neighbourhood only depends on $Z$, $\mathscr{L}H$, and the parameters of $\rho$. 
\end{thm}

\noindent
We begin by proving the following observation:

\begin{lemma}\label{lemma:InterBounded}
Let $\mathscr{M}A$ be a finitely generated halo product. For every $L \geq 0$, there exists some $D \geq 0$ such that, for all distinct $A$-cosets $P,Q$ in $\mathscr{M}A$, the intersection $P \cap Q^{+L}$ has diameter $\leq D$.
\end{lemma}

\begin{proof}
If $P \cap Q^{+L}$ is empty, there is nothing to prove. So let us assume that $P \cap Q^{+L}$ contains at least one point. In other words, there exist $(c,p) \in P$ and $(d,q) \in Q$ at distance $\leq L$. This implies that $d \in c L^+(B(p,L))$ and that $d(p,q) \leq L$. The elements $c$ and $d$ are uniquely determined by $P$ and $Q$. But there are only finitely many $p$ such that $c^{-1}d \in L^+(B(p,L))$. And, once $p$ is fixed, there are only finitely many $q$ such that $d(p,q) \leq L$. Therefore, $P \cap Q^{+L}$ is finite, or equivalently bounded. Because there are only finitely many $\mathscr{M}A$-orbits of pairs of leaves with intersecting $L$-neighbourhoods, we can bound the diameters of these intersections uniformly, concluding the proof the lemma.
\end{proof}

\begin{proof}[Proof of Theorem~\ref{thm:InALeaf}.]
According to Theorem~\ref{thm:EmbeddingThmGeneral}, there exist $R \geq 0$, $\varphi_0 \in L(H)$, and $h_0 \in H$ such that $\rho(Z) \subset \left\{ (\varphi,h) \mid \mathrm{supp}(\varphi_0^{-1}\varphi) \subset \{h_0,h\}^{+R} \right\}$, where $R$ only depends on $Z$, $\mathscr{L}H$, and the parameters of $\rho$. We deduce from Lemma~\ref{lem:PseudoLeaves} that $\rho(Z)$ must be contained in the neighbourhood of a union of cosets $\varphi_1H, \ldots, \varphi_n H$ all intersecting a given ball $B(x,T)$, where $T$ and the size of this neighbourhood only depends on $\mathscr{L}H$ and $R$. But we know from Lemma~\ref{lemma:InterBounded} that the coarse intersection between two $H$-cosets is always bounded, so this union is coarsely a pointed sum of copies of $H$. Because $\rho(Z)$ satisfies the thick bigon property, it must be one-ended, and consequently it must be contained in the neighbourhood of a single $H$-coset (whose size depends only on $\mathscr{L}H$, $Z$, and the parameters of $\rho$).
\end{proof}

\noindent
We conclude this section by showing that the main examples of halos of groups we consider in this paper are full. See also Lemma~\ref{lem:NilpotentFull} for $2$-nilpotent wreath products.

\begin{prop}\label{prop:AreFull}
The halos of groups given by wreath products, lampjuggler groups, lampdesigner groups, and lampcloner products are full.
\end{prop}

\begin{proof}
Recall from Remark \ref{rem:lamdesigner/lampcloner} that a lampdesigner product (and a fortiori a lamplighter group and a lampjuggler group) is naturally a subgroup of the lampcloner group. Through this inclusion, the halo of groups given by the lampdesigner appears as a ``sub-(halo of groups)'' of the one given by the lampcloner. In particular, the supports of elements are preserved under the inclusion map. Consequently, it is enough to prove the proposition for the lampcloner product. However, we include short proofs for lamplighter and lampjuggler groups as an illustration of the intuition behind the definition of full halos of groups.

\medskip \noindent
Let $H$ be a finitely generated group. Let us verify that our halos of groups satisfy the condition given by Definition~\ref{def:Cool} for $K=1$. So fix four finite subsets $R,S,T,U \subset H$ satisfying
$$R \cap S = R \cap T = R \cap U = S \cap T= \emptyset.$$
For wreath products, we clearly have
$$L(R \cup S) L(U) \cap L(R \cup T) \subset L(R \cup T)=L(R)L(T).$$
For lampjuggler groups, let $\sigma \in L(R \cup S)$ and $\nu \in L(U)$ be two permutations satisfying $\sigma \nu \in L(R \cup T)$. If there exists some $(a,i) \in R \times \{1, \ldots, n\}$ such that $\sigma(a,i) \notin R \times \{1, \ldots, n\}$, then $\sigma (\nu((a,i))) = \sigma((a,i)) \notin R \times \{1, \ldots, n\}$ because $\nu$ fixes $R \times \{1, \ldots, n\}$. Because $\sigma$ stabilises $(R \cup S) \times \{1, \ldots, n\}$, we must have $\sigma((a,i)) \in S \times \{1, \ldots, n\}$, but $\sigma(\nu((a,i))) \in S \times \{1, \ldots, n\}$ is impossible since $\sigma \nu$ fixes $S \times \{1, \ldots, n\}$. Thus, $\sigma(R \times \{1, \ldots, n\}) \subset R \times \{1, \ldots, n\}$, which implies that $\sigma(R \times \{1, \ldots, n\})=R \times \{1, \ldots, n\}$ because $R \times \{1, \ldots, n\}$ is finite. Because $\nu$ fixes $R \times \{1, \ldots, n\}$, we also have $\sigma(\nu(R \times \{1,\ldots, n\}))=R \times \{1, \ldots, n\}$. Because $\sigma \nu$ stabilises $(R \cup T) \times \{1, \ldots, n\}$ we also must have $\sigma(\nu(T\times \{1, \ldots, n\}))=T \times \{1, \ldots, n\}$. In other words, $\sigma \nu$ can be written as a product of a permutation of $R\times \{1, \ldots, n\}$ with a permutation of $T \times \{1, \ldots, n\}$, i.e.\ $\mu \nu \in L(R) L(T)$.

\medskip \noindent
Finally, let us consider lampcloner products. Let $\varphi \in L(R \cup S)$ and $\psi \in L(U)$ be three linear transformations satisfying $\varphi \psi \in L(R \cup T)$. Fixing a large finite subset $V \subset H$ such that $\varphi,\psi \in L(V)$, we can represent $\varphi$ and $\psi$ as (finite) matrices. Decomposing $V$ as the disjoint union $R \sqcup S \sqcup T \sqcup (V \backslash (R \cup S \cup T))$, our assumptions imply that
$$\varphi = \left( \begin{array}{cccc} I & X & 0 & 0 \\ Y & J & 0 & 0 \\ 0 & 0 & 1 & 0 \\ 0 & 0 & 0 & 1 \end{array} \right), \ \psi = \left( \begin{array}{cccc} 1 & 0 & 0 & 0 \\ 0 & A & B & C \\ 0 & E & F & G \\ 0 & L & M & N \end{array} \right), \ \varphi \psi = \left( \begin{array}{cccc} \ast & 0 & \ast & 0 \\ 0 & 1 & 0 & 0 \\ \ast & 0 & \ast & 0 \\ 0 & 0 & 0 & 1 \end{array} \right).$$
Computing the product of the first two matrices, it follows that
$$\left( \begin{array}{cccc} I & XA & XB & XC \\ Y & JA & JB & JC \\ 0 & E & F & G \\ 0 & L & M & N \end{array} \right) = \left( \begin{array}{cccc} \ast & 0 & \ast & 0 \\ 0 & 1 & 0 & 0 \\ \ast & 0 & \ast & 0 \\ 0 & 0 & 0 & 1 \end{array} \right).$$
Notice that, in the right matrix, the bottom-left star must be zero. Proving that the top-right star is also zero will prove the desired conclusion $\varphi \psi \in L(R) L(S)$. It suffices to observe that $X=0$. Indeed, because $Y=0$, the matrix representation of $\varphi$ implies that $J$ must be invertible. Thus, we deduce from the equality $JA=1$ that $A$ must be invertible as well. And, finally, we use the equality $XA=0$ to conclude that $X=0$. 
\end{proof}


\noindent
As an illustration of Theorem~\ref{thm:InALeaf}, let us record the following application. We emphasize that the statement does not follow from \cite{LampGT}.

\begin{cor}
Let $F$ be a non-trivial finite group, $I$ an infinite finitely generated group, and $H$ a one-ended finitely presented group. The wreath products $F \wr (I \wr H)$ and $I \wr (F \wr H)$ are not quasi-isometric.
\end{cor}

\begin{proof}
According to Proposition~\ref{prop:InfWreathFleshy}, $I \wr (F \wr H)$ satisfies the thick bigon property, so, if there exists a quasi-isometry $\Phi : I \wr (F \wr H) \to F \wr (I \wr H)$, then Theorem~\ref{thm:InALeaf} implies that its image must be contained in a coset of $I \wr H$. Because the image of $\Phi$ must be quasi-dense, this shows that $F$ must be finite.
\end{proof}

\subsection{Application: the Subgroup Theorem}\label{section:SubThm}

\noindent
Because subgroups in finitely generated groups are always coarsely embedded, the embedding theorem provided by Theorem~\ref{thm:InALeaf} applies to subgroups satisfying the thick bigon property, such as one-ended finitely presented subgroups. This already provides some non-trivial algebraic information.

\begin{cor}\label{cor:Algebraic}
Let $\mathscr{L}H$ be a finitely generated halo group with $\mathscr{L}$ full and $L(H)$ locally finite. Every one-ended finitely presented subgroup in $\mathscr{L}H$ is conjugate to a subgroup of~$H$.
\end{cor}

\begin{proof}
Let $K \leq \mathscr{L}H$ be a one-ended finitely presented subgroup. According to Lemma~\ref{lem:TBP}, $K$ satisfies the thick bigon property. So Theorem~\ref{thm:InALeaf} applies to the inclusion map $K \hookrightarrow \mathscr{L}H$ and shows that $K$ lies in a neighbourhood of an $H$-coset. Because the coarse intersection between two distinct $H$-cosets is always bounded according to Lemma~\ref{lemma:InterBounded}, there exists a unique such $H$-coset, which implies that $K$ stabilises it (with respect to the action of $K$ on $\mathscr{L}H$ by left multiplication). In other words, $K$ is contained in a conjugate of $H$, as desired.
\end{proof}

\begin{remark}
It is worth noticing that one-endedness cannot be removed from the assumption of Corollary~\ref{cor:Algebraic}. Indeed, if $H$ contains a free subgroup, say freely generated by $g$ and $h$, then, given a non-trivial element $a \in L(H)$, the subgroup $\langle aga^{-1},h \rangle$ is free but not contained in a conjugate of $H$. 
\end{remark}

\noindent
It is worth mentioning that Theorems~\ref{thm:EmbeddingThmGeneral} and~\ref{thm:InALeaf} do not only help us to distinguish halo groups up to quasi-isometry, but also provide non-trivial information on arbitrary finitely generated groups quasi-isometric to halo groups. 

\begin{thm}\label{thm:Peripheral}
Let $H$ be a finitely generated group satisfying the thick bigon property and let $\mathscr{L}H$ be a finitely generated halo group with $\mathscr{L}$ full and $L(H)$ locally finite. Every finitely generated group $G$ quasi-isometric to $\mathscr{L}H$ contains a finite collection of subgroups $\mathcal{H}$ such that:
\begin{itemize}
	\item all the groups in $\mathcal{H}$ are quasi-isometric to $H$;
	\item the collection $\mathcal{H}$ is almost malnormal;
	\item every finitely generated subgroup in $G$ satisfying the thick bigon property is contained in a conjugate of a subgroup from $\mathcal{H}$.
\end{itemize}
\end{thm}

\noindent
The following elementary observation will be needed in order to prove the theorem.

\begin{lemma}\label{lem:Commensurator}
Let $\mathscr{L}H$ be a halo product. For every $g \in \mathscr{L}H$, if $gHg^{-1} \cap H$ is infinite then $g \in H$. 
\end{lemma}

\begin{proof}
Observe that, if we write $g$ as $ch$ for some $c \in L(H)$ and $h \in H$, then $gHg^{-1}=cHc^{-1}$. If $k \in H$ belongs to $cHc^{-1}$, it can also be written as $cqc^{-1}$ for some $q \in H$. But the fact that $cqc^{-1}=cqc^{-1}q^{-1} \cdot q$ belongs to $H$ implies that $cqc^{-1}q^{-1}=1$. It follows that $q$ stabilises $\mathrm{supp}(c)$ for its action on $H$. Therefore, either $\mathrm{supp}(c)=\emptyset$ and $g$ belongs to $H$, or $gHg^{-1} \cap H$ is finite  (because there are finitely many elements of $H$ stabilising $\mathrm{supp}(c) \neq \emptyset$). 
\end{proof}

\begin{proof}[Proof of Theorem~\ref{thm:Peripheral}.]
Fix a quasi-isometry $q : G \to \mathscr{L}H$ and a quasi-inverse $\bar{q} : \mathscr{L}H \to G$. Then $G$ quasi-acts properly and cocompactly on $\mathscr{L}H$ via
$$g \mapsto \left( (c,p) \mapsto q( g \cdot \bar{q}(c,p)) \right), \ g \in G, (c,p) \in \mathscr{L}H.$$
We know from Lemma~\ref{lem:Commensurator} that $H$ has finite index in its commensurator, and we deduce from Theorem~\ref{thm:InALeaf} that every auto-quasi-isometry of $\mathscr{L}H$ sends every $H$-coset at (uniform) finite Hausdorff distance from another $H$-coset. It follows from \cite[Theorem~1.1]{MR4520684} (and its proof) that, with respect to our quasi-action $G \curvearrowright \mathscr{L}H$,
\begin{itemize}
	\item[(i)] there exist a constant $C \geq 0$ and finitely many $H$-cosets $\mathcal{H}_1, \ldots, \mathcal{H}_n$ such that, for every $H$-coset $\mathcal{H}$, there exist $g \in G$ and $1 \leq i \leq n$ such that the Hausdorff distance between $\mathcal{H}$ and $g \mathcal{H}_i$ is finite;
	\item[(ii)] for every $H$-coset $\mathcal{H}$, the quasi-stabiliser $$\mathrm{qstab}(\mathcal{H}):= \left\{ g \in G \mid \text{$\mathcal{H}$ and $g \mathcal{H}$ at finite Hausdorff distance} \right\}$$ quasi-acts cocompactly on $\mathcal{H}$.
\end{itemize}
Without loss of generality, we assume that $\mathcal{H}_1, \ldots, \mathcal{H}_n$ are pairwise distinct. For every $1 \leq i \leq n$, let $H_i$ denote the quasi-stabiliser of $\mathcal{H}_i$. It follows from $(ii)$ that $H_1, \ldots, H_n$ are all quasi-isometric to $H$. 

\medskip \noindent
Let $g \in G$ and $1 \leq i,j \leq n$ be such that $gH_ig^{-1} \cap H_j$ is infinite. Notice that $gH_ig^{-1}$ (resp. $H_j$) quasi-stabilises $g \mathcal{H}_i$ (resp. $\mathcal{H}_j$). Because the intersection between two neighbourhoods of distinct $H$-cosets must be bounded in $\mathscr{L}H$, according to Lemma~\ref{lemma:InterBounded}, we must have $g \mathcal{H}_i= \mathcal{H}_j$. In other words, $i=j$ and $g \in H_i$. Thus, we have proved that $\{H_1, \ldots, H_n\}$ is almost malnormal.

\medskip \noindent
Finally, let $K \leq G$ be a finitely generated subgroup satisfying the thick bigon property. If we denote by $\iota : K \hookrightarrow G$ the inclusion, then $q \circ \iota$ defines a coarse embedding $K \to \mathscr{L}H$. According to Theorem~\ref{thm:InALeaf}, $q(\iota(K))$ lies in the neighbourhood of an $H$-coset $\mathcal{H}$. As a consequence, for every $k \in K$, $k \cdot \mathcal{H}$ and $\mathcal{H}$ both contain $q(\iota(K))$ in a neighbourhood. But $k \cdot \mathcal{H}$ lies at finite Hausdorff distance from an $H$-coset, and again the intersection between two neighbourhoods of two distinct $H$-cosets is bounded, so $k$ has to quasi-stabilise $\mathcal{H}$. Therefore, $K \leq g H_ig^{-1}$ where $g \in G$ and $1 \leq i \leq n$ are given by (i), i.e.\ the Hausdorff distance between $\mathcal{H}$ and $g \mathcal{H}_i$ is finite.
\end{proof}

\noindent
We end this section by illustrating Theorem~\ref{thm:Peripheral} with a few concrete applications. First, we notice that a finitely generated group with an infinite centre cannot be quasi-isometric to the halo products we usually consider. (Notice that, for torsion-free groups, this also follows from Theorems~\ref{thm:AbSubThick} and~\ref{thm:InALeaf}.) 

\begin{cor}\label{cor:EasyQI}
Let $H$ be a finitely generated group satisfying the thick bigon property and let $\mathscr{L}H$ be a finitely generated halo product with $\mathscr{L}$ full and $L(H)$ locally finite but infinite. Every finitely generated group quasi-isometric to $\mathscr{L}H$ has a finite centre.
\end{cor}

\begin{proof}
Let $G$ be a group quasi-isometric to $\mathscr{L} H$. According to Theorem~\ref{thm:Peripheral}, $G$ contains an almost malnormal subgroup $\bar{H}$ quasi-isometric to $H$. As an infinite almost malnormal subgroup, $\bar{H}$ must coincide with its normaliser, so the centre of $G$ must lie in $\bar{H}$. If this centre is infinite, then $g \bar{H} g^{-1} \cap \bar{H}$ must be infinite as well for every $g \in G$ since it contains the centre, hence $G= \bar{H}$. This implies that $H$ is quasi-isometric to $\mathscr{L}H$, which is not possible since $H$ satisfies the thick bigon property but not $\mathscr{L} H$. 
\end{proof}

\noindent
Next, it is worth noticing that a permutational halo product whose associated action admits infinite stabilisers is usually not quasi-isometric to the halo products consider this paper. As an instance of this phenomenon:

\begin{cor}\label{cor:PermutationW}
Let $A$ be a finitely generated group satisfying the thick bigon property and let $\mathscr{M}A$ be a finitely generated halo product with $\mathscr{M}$ full and $M(A)$ locally finite but infinite. If a permutational wreath product $F \wr_X H$ is quasi-isometric to $\mathscr{M}A$, then $H$ acts on $X$ with finite stabilisers.
\end{cor}

\begin{proof}
Let $\mathcal{P}$ denote the collection of the conjugates of the subgroups in $F \wr_X H$ provided by Theorem~\ref{thm:Peripheral}. In particular, there exists some $P \in \mathcal{P}$ such that $H \leq P$. Assume that there exists a point $x \in X$ with an infinite stabiliser in $H$. We claim that $\bigoplus_{H \cdot x} F \leq P$. 

\medskip \noindent
We fix an $a\in \bigoplus_{H \cdot x} F$ and we argue by induction over the size of $\mathrm{supp}(a)$. If $\mathrm{supp}(a)$ is empty, there is nothing to prove. Otherwise, we can write $a$ as a product $bc$ where $b,c \in \bigoplus_{H \cdot x} F$ are such that the support of $c$ has size the support of $a$ minus one and such that the support of $b$ has size one. We know from our induction hypothesis that $c \in P$. Moreover, we have
$$P \cap bPb^{-1} \supset H \cap  bHb^{-1} \supset \mathrm{stab}_{H}(\mathrm{supp}(b))$$
where $\mathrm{stab}_{H}(\mathrm{supp}(b))$ is infinite since $\mathrm{supp}(b)$ contains only an $H$-translate of $x$. Because $\mathcal{P}$ is almost malnormal, necessarily $b \in P$. Therefore, we have $a=bc \in P$ as desired.

\medskip \noindent
Thus, $P$ contains the infinite normal subgroup $N:= \langle \bigoplus_{H \cdot x} F , H \rangle \lhd F \wr_{X} H$. For every $g \in F \wr_{X} H$, we must have $P \cap gPg^{-1} \supset N$. And, because $\mathcal{P}$ is almost malnormal, this implies that $g \in P$. In other words, we have proved that $P = F \wr_{X} H$. This implies that $F \wr_{X} H$ must quasi-isometric to $A$. So $F \wr_X H$ satisfies the thick bigon property, and Theorem~\ref{thm:InALeaf} shows that the image of a quasi-isometry $F \wr_X H \to \mathscr{M}A$ lies in a neighbourhood of an $A$-coset. But, if an $A$-coset is quasi-dense in $\mathscr{M}A$, necessarily $A$ has finite index in $\mathscr{M}A$, which amounts to saying that $M(A)$ is finite. 
\end{proof}

\noindent
Finally, let us show that a form of rigidity can be deduced from Theorem~\ref{thm:Peripheral} for elementary amenable groups. In our next statement, $\mathcal{C}$ refers to the class of torsion-free minimax solvable groups satisfying that any group quasi-isometric to one of its members is virtually torsion-free solvable minimax. This includes finitely generated nilpotent groups, $\mathrm{SOL}(\mathbb{Z})$, and solvable Baumslag-Solitar groups.

\begin{thm}\label{thm:QIrigidityEA}
Let $G$ be an elementary amenable group and let $H\in \mathcal{C}$. Assume that $G$ is quasi-isometric to a finitely generated halo product $\mathscr{L} H$ with $\mathscr{L}$ full and $L(H)$ locally finite. Then $G$ virtually splits as $L \rtimes \bar{H}$, where $\bar{H}$ is quasi-isometric to $H$, where $L$ is locally finite, and where $\bar{H}$ acts by conjugation on $L \backslash \{1\}$ with finite stabilisers.
\end{thm}

\noindent
We recall the following theorem of Hillmann and Linnel \cite{HL}. For the following precise statement, we refer to 
\cite[Lemma 2.1]{KLo}. 

\begin{thm}\label{thm:HL}
Let $G$ be an elementary amenable finitely generated group of finite Hirsch length. Then $G$ has a subgroup $\bar{G}$ of finite index which is (locally finite)$\rtimes$(torsion free minimax solvable). 
\end{thm}

\noindent
In \cite[Lemma 2.1]{KLo}, it is said that $\bar{G}$ is (locally finite)$\rtimes$(torsion free solvable). But it is well-known that a torsion free solvable group of finite Hirsch length is minimax (see for instance \cite[p. 2]{JK}). We shall deduce the following lemma.

\begin{lemma}\label{lem:HirschElemAM}
Let $G$ be an elementary amenable finitely generated group of Hirsch length $d\in \mathbb{N}$, and let $H$ be a minimax solvable subgroup of Hirsch length $d$. Then $KH$ has finite index in $G$, where $K$ is the locally finite radical of $G$. In particular, $G$ is virtually (locally finite)$\rtimes H$.
\end{lemma}

\begin{proof}
On passing to a finite-index subgroup, we may assume by Theorem~\ref{thm:HL} that $G/K$ is torsion-free minimax solvable. Since $H$ is torsion-free, $H\cap K$ is trivial, and hence the projection $p$ to $G/K$ is injective in restriction to $H$. Since the Hirsch length of $p(H)$ equals that of $G/K$, we deduce from \cite[Lemma 2.1]{KM} that $p(H)$ has finite index in $G/K$. Hence $KH$ has finite index in $G$.  
\end{proof}

\begin{proof}[Proof of Theorem~\ref{thm:QIrigidityEA}.]
We know from Theorem~\ref{thm:Peripheral} that $G$ contains an almost malnormal group $\bar{H}$ quasi-isometric to $H$. On replacing it by a finite-index subgroup, we can assume that $\bar{H}$ is torsion-free mini-max. By \cite{St}, the rational cohomological dimension of $G$ equals that of $\mathscr{L} H$, and therefore that of $H$. On the other hand, by \cite{St}, the Hirsch length of $G$ equals its cohomological dimension. Hence we deduce that the Hirsch lengths of $G$ and $H$ are equal. But since $\bar{H}$ is quasi-isometric to $H$, we deduce from \cite[Corollary~1.3.]{S} that the Hirsch lengths of $H$ and $\bar{H}$ are equal. In conclusion, the Hirsch length of $G$ equals the Hirsch length of $\bar{H}$, so we conclude by Lemmas~\ref{lem:HirschElemAM} and the following elementary observation:

\begin{fact}
Let $L \rtimes H$ be a semi-direct product of two groups. Then $H$ is almost malnormal in $L \rtimes H$ if and only if the action by conjugation of $H$ on $L$ has finite stabilisers. 
\end{fact}

\noindent
Fix an element $g \in L \rtimes H$. We can write $g$ as a product $\ell h$ where $\ell \in L$ and $h \in H$. Observe that
$$\begin{array}{lcl} \ell H\ell^{-1} \cap H & = & \{\ell h\ell^{-1} \mid h \in H\} \cap H= \{ \ell h\ell^{-1} h^{-1} \cdot h \mid H \} \cap H \\ \\ & = & \{ h \in H \mid \ell h\ell^{-1}h^{-1}=1 \} = \mathrm{stab}_H(\ell). \end{array}$$
Consequently, $H$ being almost malnormal amounts to saying that $\mathrm{stab}_H(\ell)$ is infinite if and only if $\ell \in H$, which only happens when $\ell=1$. 
\end{proof}

\subsection{Application: the graph of leaves}\label{section:QIinvariants}

\noindent
A major consequence of Theorem~\ref{thm:EmbeddingThmGeneral} is that, given two finitely generated halo groups $\mathscr{M}A,\mathscr{N}B$ with $\mathscr{M},\mathscr{N}$ full, $M(A),N(B)$ locally finite, and $A,B$ satisfying the thick bigon property, every quasi-isometry $\mathscr{M}A \to \mathscr{N}B$ is (up to finite distance) \emph{leaf-preserving}, i.e.\ it sends $A$-cosets to $B$-cosets and it admits a quasi-inverse that sends a $B$-coset to an $A$-coset. In this section, we collect some useful ways to distinguish various halo groups up to leaf-preserving quasi-isometry. Our first invariant encodes how $H$-cosets (referred to as \emph{leaves}) are organised in a halo product $\mathscr{L}H$. In order to describe it, we need the following notion:

\begin{definition}
An \emph{angular graph} is a graph $X$ endowed with an \emph{angle map} $\measuredangle$ that assigns to every vertex $o \in X$ and all neighbours $u,v \in X$ a real number $\measuredangle_o(u,v)$. 
\end{definition}

\noindent
Roughly speaking, the angular graph we are interested in has the leaves of our halo group $\mathscr{L}H$ as its vertices, its edges connect two leaves whenever they are close, and the angle between two leaves $P,Q$ from a third leaf $O$ measures the distance between the projections of $P,Q$ on $O$. More precisely:

\begin{definition}
Let $\mathscr{L}H$ be a finitely generated halo product. Given an $\epsilon\geq 0$, the \emph{graph of leaves} $\mathscr{G}_\epsilon(\mathscr{L}H)$ is the graph whose vertices are the $H$-cosets in $\mathscr{L}H$ and whose edges connect two leaves whenever they are at distance $\leq \epsilon$. It is endowed with an angle map $\measuredangle$ such that, for every leaf $P$ and for all neighbours $Q,R$ in $\mathscr{G}_\epsilon(\mathscr{L}H)$, $\measuredangle_P(Q,R)$ is the smallest distance between two points of $P$ respectively minimising the distance to~$Q,R$.  
\end{definition}

\noindent
Our next statement shows that, for various halo groups, if they are quasi-isometric then their graphs of leaves must be quasi-isometric as well.

\begin{prop}\label{prop:GraphLeaves}
Let $\mathscr{M}A$ and $\mathscr{N}B$ be two finitely generated halo products. For all $U,V,\epsilon \geq 0$, there exists some $\eta \geq 0$ such that the following holds. Every leaf-preserving $(U,V)$-quasi-isometry $\varphi : \mathscr{M}A \to \mathscr{N}B$ induces a bijective quasi-isometry $\phi : \mathscr{G}_\epsilon(\mathscr{M}A) \to \mathscr{G}_\eta(\mathscr{N}B)$ such that:
\begin{itemize}
	\item $\phi$ sends two adjacent vertices to two adjacent vertices;
	\item there exist constants $R,S >0$ such that $$\frac{1}{R} \measuredangle_O(P,Q) - S \leq \measuredangle_{\phi(O)}( \phi(P),\phi(Q)) \leq R \measuredangle_O(P,Q)+S$$ for every vertex $O \in \mathscr{G}_\epsilon(\mathscr{M}A)$ and all neighbours $P,Q \in \mathscr{G}_\epsilon(\mathscr{M}A)$. 
\end{itemize}
\end{prop}

\noindent
We recall the notation $L^+(S):= L(S^{+r_0})$ from Convention~\ref{Conv}. 
We start by proving the following observation:

\begin{lemma}\label{lem:Thin}
Let $\mathscr{L}H$ be a finitely generated halo group. Assume that $L(H)$ is generated by $L^+(1)$ as an $H$-module. Let $P,Q \subset \mathscr{L}H$ be two $H$-cosets and $y \in P$ a vertex minimising the distance to $Q$. For every $x \in P$, $d(x,y) \leq d(x,Q)+2d(P,Q)+4r_0$. 
\end{lemma}

\begin{proof}
Write $P=cH$, $Q=dH$, $x=(c,p)$, and $y=(c,q)$. Set $S:= \mathrm{supp}(c^{-1}d)$. The projection of a path in $\mathscr{L}H$ from a point in $P=cH$ to a point in $Q=dH$ is a path in $H$ that has to pass at distance $\leq r_0$ from every vertex in $S$. Otherwise, there exists some $s \in \mathrm{supp}(c^{-1}d)$ such that our path in $H$ does not cross $B(s,r_0)$, which implies that $c^{-1}d$ must belong to $L(H \backslash \{s\})$. But
$$c^{-1}d \in L(\mathrm{supp}(c^{-1}d) \cap L(H \backslash \{s\})= L(\mathrm{supp}(c^{-1}d) \backslash \{s\})$$
provides a contradiction. This observation implies that 
$$d((c,p),Q) \geq d(p,S^{+r_0}) \geq d(p,S) - r_0,$$
hence $d(p,S) \leq d(x,Q) + r_0$. Similarly, $d(q,S) \leq d(P,Q)+r_0$. It also follows from our observation that $\mathrm{diam}(S) \leq d(P,Q)+2r_0$, since the projection of a path in $\mathscr{L}H$ from $y$ to a point in $Q$ is a path in $H$ crossing the $r_0$-balls centred at two points realising the diameter of $S$. We conclude that
$$d(x,y) = d(p,q) \leq d(p,S) + \mathrm{diam}(S) + d(q,S) \leq d(x,Q)+2d(P,Q)+4r_0,$$
as desired. 
\end{proof}

\begin{proof}[Proof of Proposition~\ref{prop:GraphLeaves}.]
Fix a quasi-inverse $\bar{\varphi} : \mathscr{N} B \to \mathscr{M}A$ of $\varphi$ that sends a $B$-coset to an $A$-coset. Up to increasing $U$ and $V$, we assume that both $\varphi$ and $\bar{\varphi}$ are $(U,V)$-quasi-isometries. Finally, fix a constant $\eta \geq U\epsilon+V$. 

\medskip \noindent
First of all, let us show that $\varphi$ induces a bijection between the leaves of $\mathscr{M}A$ and the leaves of $\mathscr{N}B$. For every leaf $P$ of $\mathscr{M}A$, there exists a leaf $Q$ of $\mathscr{N}B$ such that $\varphi(P) \subset Q$. Similarly, there exists a leaf $R$ of $\mathscr{M}A$ such that $\bar{varphi}(Q) \subset R$. Therefore, the leaf $P$, which lies at finite Hausdorff distance from $\bar{\varphi}(\varphi(P))$, must lie in a neighbourhood of $R$. Because leaves are unbounded, it follows from Lemma~\ref{lemma:InterBounded} below that $R=P$. Thus, there exists a bijection $\phi$ from the leaves of $\mathscr{M}A$ to the leaves of $\mathscr{N}B$ such that, for every leaf $P$, $\varphi$ sends $P$ at finite Hausdorff distance from $\phi(P)$ (the distance being bounded independently of $P$). 

\medskip \noindent
From now on, we assume that $\varphi$ (resp. $\bar{\varphi}$) sends each leaf $P$ of $\mathscr{M}A$ (resp. $\mathscr{N}B$) exactly to $\phi(P)$ (resp. $\phi^{-1}(P)$). This is possible up to replacing $\varphi$ (resp. $\bar{\varphi}$) by a map at finite distance. 

\medskip \noindent
Given two adjacent vertices $P,Q\in \mathscr{G}_\epsilon(\mathscr{M}A)$, $\varphi$ sends the leaves $P$ and $Q$ to the leaves $\phi(P)$ and $\phi(Q)$, which are at distance $\leq U\epsilon+V \leq \eta$. Consequently, $\phi(P)$ and $\phi(Q)$ must be adjacent in $\mathscr{G}_\eta(\mathscr{N}B)$ as well.

\medskip \noindent
Now, if $P,Q \in \mathscr{G}_\eta(\mathscr{N}B)$ are two adjacent vertices, then $\bar{\varphi}$ sends the leaves $P$ and $Q$ to the leaves $\phi^{-1}(P)$ and $\phi^{-1}(Q)$, which are at distance $\leq U \eta +V$.

\medskip \noindent
The two previous paragraphs show that $\phi$ is a quasi-isometry $\mathscr{G}_\epsilon(\mathscr{M}A) \to \mathscr{G}_\eta(\mathscr{N}B)$ and they prove the first item from our proposition. It remains to show the second item.

\medskip \noindent
Fix a vertex $O \in \mathscr{G}_\epsilon(\mathscr{M}A)$ and two distinct neighbours $P,Q \in \mathscr{G}_\epsilon(\mathscr{M}A)$. There exist vertices $o_1,o_2 \in O$, $p \in P$, and $q \in Q$ satisfying $d(o_1,p),d(o_2,q) \leq \epsilon$ and $d(o_1,o_2)= \measuredangle_O(P,Q)$. Finally, let $a_1,a_2 \in \phi(O)$ be two vertices minimising the distance to $\phi(P),\phi(Q)$ such that $d(a_1,a_2)= \measuredangle_{\phi(O)}(\phi(P),\phi(Q))$. Thanks to Lemma~\ref{lem:Thin}, we deduce that
$$\begin{array}{lcl} \measuredangle_{\phi(O)}(\phi(P),\phi(Q)) & = & d(a_1,a_2) \leq d(a_1, \varphi(o_1))+ d(\varphi(o_1), \varphi(o_2))+d(\varphi(o_2),a_2)) \\ \\ & \leq & Ud(o_1,o_2)+V+2(3\eta+4r_0) \\ \\ & \leq & U \measuredangle_O(P,Q) +V+2(3\eta+4r_0) \end{array}$$
and similarly that
$$\begin{array}{lcl} \measuredangle_{\phi(O)}(\phi(P),\phi(Q)) & = & d(a_1,a_2) \geq d(\varphi(o_1),\varphi(o_2)) - d(a_1,\varphi(o_1))-d(a_2, \varphi(o_2)) \\ \\ & \geq & \frac{1}{U} d(o_1,o_2) -V-2(3\eta +4r_0) \\ \\ & \geq & \frac{1}{U} \measuredangle_O(P,Q) -V-2(3\eta +4r_0)  \end{array}$$
This concludes the proof of our proposition. 
\end{proof}

\noindent
Let us distinguish up to quasi-isometry a few halo groups as an application of Proposition~\ref{prop:GraphLeaves}. Our examples will focus on halo products given by specific halos of groups, namely:

\begin{definition}
A halo of groups $\mathscr{L}$ over a metric space $X$ is \emph{large-scale commutative} if there exists some constant $D \geq 0$ such that, for all subsets $R,S \subset X$ at distance $\geq D$, the subgroups $L(R)$ and $L(S)$ commute. 
\end{definition}

\noindent
Examples include many of the halo products introduced in Section~\ref{section:Halo}, such as lamplighter, lampjuggler, lampdesigner, lampcloner, and lampbraider groups. On the other hand, the halos of groups associated to verbal wreath products, such as nilpotent or metabelian wreath products, are usually not large-scale commutative. The key idea here is that, from some perspective, the graph of leaves of a halo product given by a large-scale commutative halo of groups looks like an infinite cube. In order to state precisely this idea, we need a couple of definitions. 
\begin{itemize}
	\item In a graph, we refer to a \emph{$k$-cube} as an induced subgraph isomorphic to the product of $k$ copies of the complete graph $K_2$.
	\item In an angular graph, a subgraph $Y$ is \emph{$R$-obtuse} if all the angles in $Y$ are $\geq R$. 
\end{itemize}
Observe that:

\begin{prop}\label{lem:SpanCube}
Let $\mathscr{L}H$ be a finitely generated halo group. We assume that $L(H)$ is locally finite and that $\mathscr{L}$ is large-scale commutative. Fix an $\epsilon>0$ and let $\mathscr{G}_\epsilon$ denote the corresponding graph of leaves. There exist constants $Q,R \geq 0$ such that the following holds. Given a vertex $O \in \mathscr{G}_\epsilon$ and some neighbours $L_1, \ldots, L_k \in \mathscr{G}_\epsilon$, if $\measuredangle_O(L_i,L_j) \geq Q$ for all $i \neq j$, then the edges $[O,L_1],\ldots,[O,L_k]$ span an $R$-obtuse $k$-cube.
\end{prop}

\noindent
It is worth noticing that, as a consequence of Proposition~\ref{prop:SquareLeaves} proved later, the $k$-cube given by the proposition is unique. 

\begin{proof}[Proof of Lemma~\ref{lem:SpanCube}.]
Fix a $Q \geq 0$ which is very large compared to $\epsilon$ and to the constant $C$ given by the large-scale commutativity of $\mathscr{L}$. Let $c$ denote the element of $L(H)$ given by $O$. Given a leaf $L_i$, it follows from the fact that $O$ and $L_i$ are adjacent in $\mathscr{G}_\epsilon$ that $L_i$ is labelled by $ce_i$ for some $e_i \in L(H)$ whose support has a small diameter (more precisely, comparable to $\epsilon$, so very small compared to $Q$). There exists a constant $R \geq 0$, large compared to $\epsilon$ and $C$, such that, given $i \neq j$, the fact that $\measuredangle_O(L_j,L_j) \geq Q$ implies that the supports of $e_i$ and $e_j$ are at distance $\geq R$. It follows, in particular, that $e_1, \ldots, e_k$ pairwise commute. Then the leaves given by
$$\{ ce_{i_1}\cdots e_{i_s} \mid \{i_1, \ldots, i_s\} \subset \{1, \ldots, k\} \} \subset L(H)$$
defines an $R$-obtuse $k$-cube spanned by $O$ and the $L_i$.  
\end{proof}

\noindent
As an application, let us distinguish some nilpotent wreath products from other halo products defined by large-scale commutative halos of groups. 

\begin{prop}\label{prop:NilWP}
Let $F \wr^{\mathfrak{n}_2} H$ be a $2$-nilpotent wreath product where $F$ is finite and where $H$ satisfies the thick bigon property. Let $\mathscr{L}K$ be a finitely generated halo product with $K$ satisfying the thick bigon property, $L(K)$ locally finite, and $\mathscr{L}$ full. If $F$ is not perfect and if $\mathscr{L}$ is large-scale commutative, then $F \wr^{\mathfrak{n}_2} K$ and $\mathscr{L}H$ are not quasi-isometric.
\end{prop}

\noindent
The restriction to finite groups that are not perfect is the due to the fact that a nilpotent product of perfect groups coincides with their direct sum. Thus, a nilpotent wreath product may coincide with a wreath product. This does not happen when our groups are not perfect:

\begin{lemma}\label{lem:NilpotentNotComm}
Let $G$ and $H$ be two groups. If $g \in G \backslash [G,G]$ and $h \in H \backslash [H,H]$, then $g$ and $h$ do not commute in the $2$-nilpotent product $G \ast^{\mathfrak{n}_2} H$.
\end{lemma}

\begin{proof}
Let $[G,H]$ denote the subgroup generated by the commutators $[a,b]$ with $a \in G$ and $b \in H$. Then we have the short exact sequence
$$1 \to [G,H] \to G \overset{\mathfrak{n}_2}{\ast} H \to G \oplus H \to 1.$$
According to \cite{MR0120270} (see also \cite{MR4053850}), the map $[a,b] \mapsto a \otimes b$ induces an isomorphism $[G,H] \to G^{\mathrm{ab}} \otimes H^\mathrm{ab}$. The desired conclusion follows.
\end{proof}

\noindent
In order to apply our embedding theorem, we also need to verify that halos of groups given by $2$-nilpotent wreath products are full:

\begin{lemma}\label{lem:NilpotentFull}
Let $X$ be a set and $F$ a group. Let $\mathscr{L}$ be the halo of groups over $X$ defined by
$$L(S):= \overset{\mathfrak{n}_2}{\underset{s \in S}{\ast}} F \text{ for every } S \subset X,$$
where $\ast^{\mathfrak{n}_2}$ denotes the $2$-nilpotent product of groups. Then $\mathscr{L}$ is full.
\end{lemma}

\begin{proof}
Notice that, for all $R,S \subset X$, the inclusion $L(R \cup S) \subset L(R)L(S) [R,S]$ holds, where $[R,S]$ denotes the subgroup generated by the commutators $[r,s]$ with $r \in R$ and $s \in S$. This is direct consequence of the equality $sr=rs[s^{-1},r^{-1}]$ for all $r \in R$, $s \in S$, and of the fact that $[R,S]$ is central in $L(X)$. Thus, in order to show that $\mathscr{L}$ is full, it suffices to verify that, given $R,S,T,U \subset X$ satisfying $R \cap S= R \cap T = S \cap T = R \cap U= \emptyset$, the inclusion
$$L(R)L(S) [R,S] L(U) \cap L(R)L(T) [R,T] \subset L(R)L(T)$$
holds. So assume that there exist $a_1,a_2 \in L(R)$, $b \in L(S)$, $c \in L(T)$, $e \in L(U)$, $d_1 \in [R,S]$, and $d_2 \in [R,TS]$ such that $a_1bd_1e=a_1cd_2$. Our goal is to show that $d_2=1$. By projecting to the subfactor $L(R)$, we get $a_1=a_2$, so our equality simplifies as $bd_1e=cd_2$. By projecting to the subfactor $L(X\backslash R)$, we get $be=c$, hence $d_1=d_2$. Finally, by projecting to the subfactor $L(X\backslash S)$, we conclude that $d_2=1$, as desired.
\end{proof}

\begin{proof}[Proof of Proposition~\ref{prop:NilWP}.]
According to \cite{MR0075947}, a nilpotent product of finite groups is finite. Consequently, also thanks to Lemma~\ref{lem:NilpotentFull}, Theorem~\ref{thm:InALeaf} applies and shows that every quasi-isometry between $F \wr^{\mathfrak{n}_2} K$ and $\mathscr{L}H$ is (up to finite distance) leaf-preserving. Then, Proposition~\ref{prop:GraphLeaves} applies and shows that, if $F \wr^{\mathfrak{n}_2} K$ and $\mathscr{L}H$ are quasi-isometric, then every graph of leaves $\mathscr{G}_\epsilon$ of $F \wr^{\mathfrak{n}_2} K$ must be quasi-isometric to some graph of leaves of $\mathscr{L}H$. Fix two non-trivial elements $f \in F\backslash [F,F]$ and $k \in K$ with $K$ far away from $1$ compared to $\epsilon$. As a consequence of Lemma~\ref{lem:SpanCube}, it suffices to show that the leaves indexed by $1, f, kfk^{-1} \in \overset{\mathfrak{n}_2}{\ast}_K F$ do not span an obtuse square in $\mathscr{G}_\epsilon$ in order to deduce that our two groups cannot be quasi-isometric.

\medskip \noindent
If such a square exists, then the element of $\overset{\mathfrak{n}_2}{\ast}_K F$ labelling the fourth leaf can be obtained both by right-multiplying $f$ with an element $c$ of small support far away from $1$ and by right-multiplying $fkf^{-1}$ with an element $d$ of small support far away from $k$. From the equality $f \cdot c = fkf^{-1} \cdot d$, equivalent to $cd^{-1}=f^{-1}kfk^{-1}$, we deduce that $$\{1,k\}\subset  \mathrm{supp}(c) \cup \mathrm{supp}(d).$$
But since $k$ and $1$ are far away, and the supports of $c$ and $d$ have small diameters, this forces the supports of $c$ and $d$ to be far away, one close to $1$ and the other close to $k$. Now since the support of $d$ is far away from $k$, we necessarily have  $\mathrm{supp}(c)=\{k\}$ and $\mathrm{supp}(d)= \{1\}$. By projecting respectively to the factors of $F$ indexed by $1$ and $k$ in the nilpotent product, we deduce that $c=fkf^{-1}$ and $d=f$.
Hence, $$f^{-1}fkf^{-1}=cd^{-1}=fkf^{-1}f^{-1}.$$
In other words, we have proved that $f$ and $fkf^{-1}$ commute, which contradicts the non-commutativity given by Lemma~\ref{lem:NilpotentNotComm}
\end{proof}

\noindent
We saw with Lemma~\ref{lem:SpanCube} that the graph of leaves of a halo product given by a large-scale commutative halo contains many obtuse cubes. The next observation shows that, in the case of wreath products, the whole graph of leaves is close from being an infinite cube.

\begin{lemma}\label{lem:LampGraphLeaves}
Let $F$ be a non-trivial finite group and $H$ a finitely generated group. Fix an $\epsilon\geq 1$ and let $\mathscr{G}_\epsilon$ denote the corresponding graph of leaves of the lamplighter group $F \wr H$. For every $R \geq 0$, there exists some $N \geq 1$ such that any two vertices of $\mathscr{G}_\epsilon$ belong to a chain of $R$-obtuse cubes of length $\leq N$.
\end{lemma}

\begin{proof}
Fix two leaves of $F \wr H$. Up to translating by an element of $F \wr H$, we assume without loss of generality that one of the leaves is $H$. Let $c \in \bigoplus_HF$ be such that $cH$ is the other leaf. According to \cite[Claim~4.14]{LampGT}, there exists a partition $H=H_1 \sqcup \cdots \sqcup H_N$ such that any two points in the same block are at distance $\geq R$. Consequently, we can write $c$ as a product $c_1 \cdots c_N$ such that each $c_i$ has support in $H_i$. We claim that, for every $0 \leq i \leq N-1$, the leaves $c_1 \cdots c_i H$ and $c_1 \cdots c_ic_{i+1}H$ belong to a common $R$-obtuse cube. Indeed, we can write $c_{i+1}$ as a product $d_1 \cdots d_k$ where the supports of the $d_i$ are vertices pairwise at distance $\geq R$. Then
$$\{ c_1 \cdots c_i d_{j_1} \cdots d_{j_r} H \mid \{j_1, \ldots, j_r\} \subset \{1 ,\ldots, k\}\}$$
is the cube we are looking for.
\end{proof}

\noindent
In the following application, we prove that lamplighters and lampshufflers are never quasi-isometric. An alternative argument will be given by Corollary~\ref{cor:ManyNotQI}.

\begin{cor}
Let $F$ be a finite group and $H,K$ two infinite finitely generated groups satisfying the thick bigon property. The lamplighter group $F \wr K$ and the lampshuffler group $\circledS H$ are not quasi-isometric.
\end{cor}

\begin{proof}
Fix an $\epsilon>0$ and let $\mathscr{G}_\epsilon$ be the corresponding graph of leaves of $\circledS H$. Assume that two leaves $\mu H$ and $\nu H$ belong to a common $R$-obtuse cube for some $R >2 \epsilon$. Then $\nu$ can be written as $\mu \sigma_1 \cdots \sigma_k$ where the supports of the $\sigma_i$ have diameters $\leq \epsilon$ and are pairwise disjoint. Consequently, the displacement of $\nu$, namely the quantity $\max \{ d(x,\mu(x)), x \in H\}$, cannot differ from the displacement of $\mu$ by more than $\epsilon$. It follows that, if any two vertices of $\mathscr{G}_\epsilon$ belong to a chain of obtuse cubes of uniform length, then the displacements of the elements in $\mathrm{FSym}(H)$ would be uniformly bounded, which is of course not the case. The desired conclusion then follows from Lemma~\ref{lem:LampGraphLeaves}. 
\end{proof}

\noindent
We conclude this section by noticing that being large-scale commutative is not preserved by quasi-isometry.

\begin{ex}
Let $n\geq 1$ be an integer and $H$ a finitely generated group with a fixed order $\leq$ invariant under conjugation. Define
$$L:= \left\langle a_h \ (h \in H) \mid a_h^{2n}=1 \ (h \in H), \ a_ha_k=a_k^{-1} a_h \ (h>k) \right\rangle.$$
Because $\leq$ is invariant under conjugation, $H$ acts on $L$ via
$$ha_kh^{-1} = a_{hk} \text{ for all } h,k \in H.$$
Let $L \rtimes H$ be the corresponding semi-direct product. We claim that $L \rtimes H$ is biLipschitz equivalent to $\mathbb{Z}_{2n} \wr H$ but not commensurable to any lamplighter over $H$. First, we need to investigate the structure of $L$.

\begin{fact}\label{fact:Locally}
Every element of $L$ can be uniquely written as $a_{h_1}^{r_1} \cdots a_{h_s}^{r_s}$ for some $s \geq 0$, $r_1, \ldots, r_s \in \{0, \ldots, 2n-1\}$, and $h_1< \cdots < h_s$ in $H$. 
\end{fact}

\noindent
Let $p,q \geq 0$, $r_1, \ldots, r_p,s_1, \ldots, s_q \in \{0, \ldots, 2n-1\}$, and $h_1< \cdots < h_p$, $k_1< \cdots < k_q$ in $H$ such that 
\begin{equation}\label{equation}
a_{h_1}^{r_1} \cdots a_{h_p}^{r_p} = a_{k_1}^{s_1} \cdots a_{k_q}^{s_q}.
\end{equation}
The abelianisation of $L$ is given by $\bigoplus_{H} \langle a_h \mid a_h^2 \rangle$, so we deduce from the equality (\ref{equation}) that $\{h_1, \ldots, h_p\} = \{ k_1, \ldots, k_q \}$, or equivalently that $p=q$ and $h_i=k_i$ for every $i$. 

\medskip \noindent
Next, for every $k \in H$, set
$$L_{\geq k} := \langle a_h \ (h \geq k) \mid a_h^{2n}=1 \ (h \geq k), a_ha_\ell = a_\ell^{-1}a_h \ (h>\ell \geq k) \rangle,$$
and similarly for $L_{>k}$. From these presentations, we deduce that each $L_{\geq k}$ naturally splits as a semi-direct product $\langle a_k \mid a_k^{2n}=1\rangle \rtimes L_{>k}$ where the action of $L_{>k}$ on $\langle a_k \mid a_k^{2n}=1 \rangle$ is such that each $a_\ell$ inverts $a_k$ (the action being well-defined because it factors through the abelianisation $\bigoplus_{h>k} \langle a_h \mid a_h^2=1 \rangle$ of $L_{>k}$), and $L$ is the direct limit of the $L_{\geq k}$ where $L_{\geq h} \to L_{\geq \ell}$ is defined for $h \geq \ell$ by sending each generator $a_k$ of $L_{\geq h}$ to the corresponding generator $a_k$ of $L_{\geq \ell}$. Then one can interpret the equality (\ref{equation}) in $L_{\geq h_1}= L_{\geq k_1}$, and deduce from the semi-direct product decomposition $\langle a_{h_1} \mid a_{h_1}^{2n}=1 \rangle \rtimes L_{> h_1}$ of $L_{\geq h_1}$ that $a_{h_1}^{r_1}= a_{h_1}^{s_1}$ in $\langle a_{h_1} \mid a_{h_1}^{2n}=1 \rangle$, i.e. $r_1=s_1$. Thus, the equality (\ref{equation}) simplifies as
$$a_{h_2}^{r_2} \cdots a_{h_p}^{r_p} = a_{k_2}^{s_2} \cdots a_{k_q}^{s_q},$$
and we can iterate the process in order to deduce that $r_2=s_2$, $r_3=s_3$, and so on. This concludes the proof of Fact~\ref{fact:Locally}.

\medskip \noindent
Fixing a finite generating set $S \subset H$ and letting $\delta_h$ denote the generator of the copy of $\mathbb{Z}_{2n}$ in $\bigoplus_H \mathbb{Z}_{2n}$ indexed by $h$, one can check that the map
$$\left\{ \begin{array}{ccc} L \rtimes H & \to & \mathbb{Z}_{2n} \wr H \\ a_{h_1}^{n_1} \cdots a_{h_r}^{n_r} \cdot h & \mapsto & \delta_{h_1}^{n_1} \cdots \delta_{h_r}^{n_r} \cdot h \end{array} \right.$$
where $n_1, \ldots, n_r \in \{0, \ldots, 2n-1\}$ and $h_1< \cdots < h_r$ in $H$, induces a isomorphism from the Cayley graph of $L \rtimes H$ with respect to $S \cup \langle a_1 \rangle$ to the Cayley graph of $\mathbb{Z}_{2n} \wr H$ with respect to $S \cup \langle \delta_1\rangle$. 
\end{ex}

\section{A word about finite presentability}\label{section:FP}

\noindent
It follows from Theorem~\ref{thm:EmbeddingThmGeneral} that, given a finitely generated halo product $\mathscr{L}H$ with $H$ satisfying the thick bigon property and $L(H)$ locally finite but infinite, the group $\mathscr{L}H$ cannot be finitely presented. This is compatible with the fact that wreath products are usually not finitely presented \cite{MR120269}. Notice, however, that halo products can be finitely presented in some cases. For instance, let $H$ be a finitely presented group and let $\mathscr{L}$ be the halo defined over $H$ by setting $L(S):= \ast_S F$ for every $S \subset H$, where $F$ is a fixed finitely presented group. Then the halo product $\mathscr{L}H$ coincides with the free product $F \ast H$, which is finitely presented. This section is dedicated to the following statement, which shows that this is essentially the only phenomenon which can lead to non-trivial examples of finitely presented halo products. Our proof relies on the techniques introduced in Section~\ref{section:EmbeddingTheorem}.

\begin{thm}\label{thm:NotFP}
Let $\mathscr{L}H$ be a halo product. Assume that there exists a finite subset $F\subset L(H)$ for which there exist infinitely many $h \in H$ such that $\langle F,hFh^{-1} \rangle \neq \langle F \rangle \ast \langle hFh^{-1} \rangle$. Then $\mathscr{L}H$ is not finitely presented. 
\end{thm}

\begin{proof}
If $\mathscr{L}H$ is not finitely generated, then there is nothing the prove. So, from now on, we assume that $\mathscr{L}H$ is finitely generated.

\medskip \noindent
Let us recall the notation from Convention~\ref{section:EmbeddingTheorem}, which we also use here. Because $\mathscr{L}H$ is finitely generated, we can fix finite generating set $X_H$ of $H$ and a finite subset $X_L \subset L(H)$ which generates $L(H)$ as an $H$-module. We also fix an $r_0 \geq 0$ such that $X_L$ is contained in $L(B(1,r_0))$, and we set $L^+(S):= L(S^{+r_0})$ for every $S \subset H$. We assume for convenience that $X_L=L^+(1)$ (which may no longer be finite). Finally, we denote by $\mathscr{C}$ the cube complex given by Definition~\ref{def:CubeComplex}. It is endowed with a height function, and, for every $k \geq 1$, we denote by $\mathscr{C}_k$ the subcomplex spanned by the vertices of height $\leq k$. 

\medskip \noindent
Let $\mathscr{W}$ denote the set of words written over $X_L \cup X_H$. To every $w \in \mathscr{W}$, we associate a (possibly degenerate) path $\Pi(w) \subset \mathscr{C}_2$ connecting the basepoint $o:= [1,\{1\}]$ to some vertex in $\mathscr{C}_1$. The construction is done by induction over the length of $w$:
\begin{itemize}
	\item First, $\Pi(1)$ is defined as the path reduced to the single vertex $o$.
	\item Next, assume that $w$ has length at least one, so that it can be written as a reduced product $w=u \ell$ with $\ell \in X_L \cup X_H$, and assume that $\Pi(u)$ is already defined. Let $[\varphi, \{h\}]$ denote the terminal point of $\Pi(u)$. Two cases may happen.
\begin{itemize}
	\item If $\ell \in X_H$, define $\Pi(w)$ by extending $\Pi(u)$ with the path of length two $[\varphi,\{h\}]-[\varphi,\{h,h\ell\}]-[\varphi,\{h \ell\}]$.
	\item If $\ell \in X_L$, set $\Pi(w):= \Pi(u)$. However, we think of the terminal point of $\Pi(w)$ as $[\varphi h \ell h^{-1},h]$, which coincides with $[\varphi,\{h\}]$ because $h \ell h^{-1}$ belongs to $L^+(h)$.
\end{itemize}
\end{itemize}
For instance, given $a,b \in X_L$ and $x \in X_H$,
$$\Pi(axb)= \left( [1, \{1\}]= [a,\{1\}], \ [a,\{1,x\}], \ [a,\{x\}] = [axbx^{-1},\{x\}] \right).$$
Let us record the following observation for future use:

\begin{claim}\label{claim:WhenLoop}
For every $w \in \mathscr{W}$, $\Pi(w)$ is a loop if and only if $w \in L^+(1)$ in $\mathscr{L}H$.
\end{claim}

\noindent
Let $[\varphi,\{h\}]$ denote the terminal vertex of $\Pi(w)$. So $\Pi(w)$ is a loop if and only if $[\varphi,\{h\}]= [1, \{1\}]$, i.e.\ $h=1$ and $\varphi \in L^+(1)$. On the other hand, it follows from the definition of $\mathscr{L}H$ and from the construction of $\Pi$ that the word $w$ represents the element $(\varphi,h)$ in $\mathscr{L}H$. This concludes the proof of Claim~\ref{claim:WhenLoop}. 

\medskip \noindent
The map $\Pi$ allows us to construct ``approximations'' of $\mathscr{L}H$. Namely, for every $k \in \mathbb{N} \cup \{ \infty\}$, we set
$$\mathscr{L}_kH := \left\langle X_L \cup X_H \mid \begin{array}{c} w=1 \text{ if $\Pi(w)$ is a loop homotopically} \\ \text{ trivial in }\mathscr{C}_k \text{ and $w=1$ in } \mathscr{L}H \end{array} \right\rangle.$$
The following observation will be important later:

\begin{claim}\label{claim:TrivialIffNullhomotopic}
For every $m \in \mathscr{W}$, $m=1$ in $\mathscr{L}_kH$ if and only if $\Pi(m)$ is a loop homotopically trivial in $\mathscr{C}_k$ and $m=1$ in $\mathscr{L}H$.
\end{claim}

\noindent
If $\Pi(m)$ is a nullhomotopic loop in $\mathscr{C}_k$ and $m=1$ in $\mathscr{L}H$, then $m$ is clearly trivial in $\mathscr{L}_kH$. Conversely, if $m$ is trivial in $\mathscr{L}_kH$, then $m$ coincides up to free reductions with a product
$$w:= a_1w_1a_1^{-1} \cdot a_2w_2a_2^{-1} \cdots a_rw_ra_r^{-1}$$
where each $w_i$ is such that $\Pi(w_i)$ is a nullhomotopic loop in $\mathscr{C}_k$ and $w_i=1$ in $\mathscr{L}H$. Clearly, we also have $w=1$ in $\mathscr{L}H$. Because $m$ coincides with $w$ up to free reductions, $\Pi(m)$ can be obtained from $\Pi(w)$ by adding and removing backtracks; in particular, $\Pi(m)$ and $\Pi(w)$ are homotopy equivalent. For each $1 \leq i \leq r$, $\Pi(a_iw_ia_i^{-1})$ is a loop based at $[1,\{1\}]$: it follows $\Pi(a_i)$, passes through the loop obtained from the same instructions as $\Pi(w_i)$ but based at the terminal point of $\Pi(a_i)$, and comes back to $[1,\{1\}]$ along $\Pi(a_i)$. In other words, $\Pi(a_iw_ia_i^{-1})$ is a translate of $\Pi(w_i)$ connected to $[1,\{1\}]$ by the path $\Pi(a_i)$. This implies that $\Pi(a_iw_ia_i^{-1})$ is nullhomotopic. We conclude that $\Pi(w)$, and a fortiori $\Pi(m)$, is nullhomotopic, concluding the proof of Claim~\ref{claim:TrivialIffNullhomotopic}.

\medskip \noindent
Because $\mathscr{C}_\infty= \mathscr{C}$ is simply connected, according to Proposition~\ref{prop:SimplyConnected}, it follows that $\mathscr{L}_\infty H$ coincides with $\mathscr{L}H$. Thus, each $\mathscr{L}_k H$ naturally surjects onto $\mathscr{L}H$. We also deduce from Proposition~\ref{prop:SimplyConnected} and Claim~\ref{claim:WhenLoop} that every finite set of relations of $\mathscr{L}H$ appears in the presentation defining $\mathscr{L}_kH$ for some $k$ sufficiently large. Therefore, if $\mathscr{L}H$ is finitely presented, then there must exists some $k \geq 1$ such that the quotient map $\mathscr{L}H_k \twoheadrightarrow \mathscr{L}H$ is an isomorphism. 

\medskip \noindent
It is worth noticing that, so far, we did not use any assumption about our halo group $\mathscr{L}H$ except that it is finitely generated. But now we need an extra assumption in order to prove that $\mathscr{C}_2$ is not homotopically trivial in $\mathscr{C}_k$ for some $k$ arbitrarily large. 

\medskip \noindent
Let $F$ be the finite subset given by our assumptions. There exists an $R \geq 0$ such that $F$ is contained in $\langle L^+(p), p \in B(1,R) \rangle$. Fix a $k \geq 1$. By assumption, there exists an $h \in H$ satisfying $d(1,h) \geq 2(R+2r_0) +k$ such that $\langle F, hFh^{-1} \rangle \neq \langle F \rangle \ast \langle hFh^{-1} \rangle$. Therefore, we can find a reduced alternating word written over $\langle F \rangle \sqcup \langle hFh^{-1} \rangle$ which is trivial in $\mathscr{L}H$. Up to conjugacy, assume that our alternating word is of the form $a_1b_1 \cdots a_nb_n=1$ with $a_1, \ldots, a_n \in \langle F \rangle$ and $b_1,\ldots, b_n \in \langle hFh^{-1} \rangle$ all non-trivial. For every $1 \leq i \leq n$, let $\alpha_i \in \mathscr{W}$ be a word representing $a_i$, corresponding to a product of conjugates of generators in $X_L$ by elements of $H$ of length $\leq R$; and let $\beta_i \in \mathscr{W}$ be a word representing $b_i$, corresponding to the conjugate by $h$ of a product of conjugates of generators in $X_L$ by elements of $H$ of length $\leq R$. 

\begin{claim}\label{claim:HomotopicallyNonTrivial}
The loop $\Pi(\alpha_1\beta_1 \cdots \alpha_n \beta_n) \subset \mathscr{C}_2$ is not homotopically trivial in $\mathscr{C}_k$. 
\end{claim}

\noindent
We begin by constructing a cover of $\mathscr{C}_k$. Set $E(1):= \left( B(1,R)^{+r_0} \right)^c$ and similarly $E(h):= \left( B(h,R)^{+r_0} \right)^c$. For every coset $U$ of $L(E(h))$, set
$$\mathscr{C}(U):= \left\{ [\varphi,T] \mid d(h,T) >R+2r_0, \ \varphi L(E(h))=U \right\}.$$
Observe that this subcomplex is well-defined. Indeed, a vertex $[\varphi,T]$ in $\mathscr{C}(U)$ can be rewritten as $[\varphi \psi, T]$ for an arbitrary $\psi \in L^+(T)$. But, because $d(h,T)>R+2r_0$, $L^+(T) \subset L(E(h))$, so $\varphi \psi L(E(h)) = \varphi L(E(h))$. Similarly, for every coset $V$ of $L(E(1))$, set
$$\mathscr{C}(V):= \left\{ [\varphi,T] \mid d(1,T) \geq R+2r_0, \ \varphi L(E(1))=V \right\}.$$
Because $d(1,h) > 2(R+2r_0) +k$, every connected subgraph of size $\leq k$ in $H$ lies at distance $>R+2r_0$ from $1$ or $h$. This implies that every cube in $\mathscr{C}_k$ is contained in $\mathscr{C}(U)$ for some coset $U$ for $L(E(h))$ or $L(E(1))$. In other words,
$$\mathcal{O}:= \{ \mathscr{C}(U) \mid U \text{ coset of $L(E(1))$ or $L(E(h))$}\}$$
covers $\mathscr{C}_k$. The nerve complex of $\mathcal{O}$ is a bipartite graph, and, by construction, the image of our path $\Pi(\alpha_1\beta_1 \cdots \alpha_n \beta_n)$ in this graph is
$$\mathscr{C}(L(E(h)) - \mathscr{C}(a_1 L(E(1))) - \mathscr{C}(a_1b_1 L(E(h))) - \cdots - \mathscr{C}(a_1 b_1 \cdots a_n b_n L(E(h))).$$
In this path, there are two types of subpath of length two. One is
$$\mathscr{C}( a_1 b_1 \cdots a_i b_i L(E(h))) - \mathscr{C}(a_1 b_1 \cdots a_i b_i a_{i+1} L(E(1))) - \mathscr{C} (a_1 b_1 \cdots a_{i+1} b_{i+1} L(E(h)));$$
and the other is
$$\mathscr{C}(a_1 b_1 \cdots  a_i L(E(1))) - \mathscr{C}(a_1 b_1 \cdots a_i b_i L(E(h))) - \mathscr{C}(a_1 b_1 \cdots a_i b_{i}a_{i+1} L(E(1))).$$
If the former is a backtrack, then we must have $a_{i+1}b_{i+1} \in L(E(h))$, hence $b_{i+1} \in L(E(h))$ since $a_{i+1}$ already belongs to $L(E(h))$. It follows that
$$b_{i+1} \in L(E(h)) \cap L\left( B(h,R)^{+r_0} \right) = L\left( E(h) \cap B(h,R)^{+r_0} \right) = L(\emptyset) = \{1\},$$
which is impossible since our alternating word is reduced. The same argument shows that the latter subpath is not a backtrack. Thus, the projection of $\Pi(\alpha_1\beta_1 \cdots \alpha_n \beta_n)$ the nerve complex is homotopically non-trivial, which implies that our path in $\mathscr{C}_2$ is homotopically non-trivial in $\mathscr{C}_k$. This concludes the proof of Claim~\ref{claim:HomotopicallyNonTrivial}. 

\medskip \noindent
Thinking of the $\alpha_1 \beta_1 \cdots \alpha_n \beta_n$ as an element of $\mathscr{L}_kH$, it follows from Claims~\ref{claim:HomotopicallyNonTrivial} and~\ref{claim:TrivialIffNullhomotopic} that one gets a non-trivial element of $\mathscr{L}_k H$ that projects trivially in $\mathscr{L}H$ of $N_k$. Thus, we have proved that the quotient $\mathscr{L}_k H \twoheadrightarrow \mathscr{L}H$ has a non-trivial kernel for every $k \geq 1$, which proves that $\mathscr{L}H$ is not finitely presented. 
\end{proof}

\section{Aptolic quasi-isometries}\label{section:Aptolic}

\noindent
In this section, our goal is to show that, under reasonable assumptions, quasi-isometries between halo groups admit a specific structure, namely:

\begin{definition}
Let $\mathscr{M}A,\mathscr{N}B$ be two finitely generated halo products. A quasi-isometry $\Phi : \mathscr{M}A \to \mathscr{N}B$ is \emph{aptolic} if there exist a bijection $\alpha : M(A) \to N(B)$ and a quasi-isometry $\beta : A \to B$ such that $\Phi : (c,p) \mapsto (\alpha(c), \beta(p))$.  
\end{definition}

\noindent
We emphasize that, given a bijection $\alpha : M(A) \to N(B)$ and a quasi-isometry $\beta : A \to B$, the map $(c,p) \mapsto (\alpha(c),\beta(p))$ does not necessarily defines a quasi-isometry $\mathscr{M}A \to \mathscr{N}B$. Some compatibility between $\alpha$ and $\beta$ is required. For wreath products, this is fully characterised by \cite[Proposition~3.1]{LampGT}.

\medskip \noindent
In Section~\ref{section:Ladders}, we introduce \emph{ladders of leaves}, the main technical tool which we will use in the subsequent sections. In Section~\ref{section:Morse}, we use these ladders to prove that a weak form of aptolicity always holds for large-scale commutative halos of groups. A true aptolicity can then be deduced for halo groups over many groups, including all the acylindrically hyperbolic groups, as shown in Section~\ref{section:Morse}. However, basic examples are not covered by this criterion, including halo groups over free abelian groups. In Section~\ref{section:Sufficient}, we introduce the concept of \emph{altitude} and we show the aptolicity of quasi-isometries by assuming in addition that our halo of groups has finite altitude, which allows us to cover all the particular examples we are interested in. Aptolicity will be used in Section~\ref{section:ApplicationsTwo} in order to construct new quasi-isometric invariants.

\subsection{Ladders}\label{section:Ladders}

\noindent
In order to show that a quasi-isometry between two halo groups is aptolic (up to finite distance), we need to be able to recognize geometrically when two points in a given halo group $\mathscr{L}H$ have close projections onto $H$. In this section, we introduce our main tool to solve this problem:

\begin{definition}
Let $\mathscr{L}H$ be a finitely generated halo group. Given $\epsilon,R >0$, an \emph{$(\epsilon,R)$-ladder} is a collection of leaves $P_1,Q_1, \ldots, P_k,Q_k$ such that:
\begin{itemize}
	\item each $P_i$ (resp. $Q_i$) lies at distance $\leq \epsilon$ from $P_{i-1},Q_i,P_{i+1}$ (resp. $Q_{i-1},P_i,Q_{i+1}$);
	\item each square of leaves $P_i,P_{i+1},Q_{i+1},Q_i$ is $R$-obtuse.
\end{itemize}
\end{definition}

\noindent
The rest of the section is dedicated to the proof of the next criterion. Recall from Convention~\ref{Conv} that we always endow our finitely generated halo products with specific finite generating sets. We also refer the reader to Section~\ref{section:QIinvariants} for the definition of obtuse configurations of leaves.

\begin{thm}\label{thm:Ladders}
Let $\mathscr{L}H$ be a finitely generated halo group. Assume that $\mathscr{L}$ is large-scale commutative and let $D$ denote the corresponding constant. For all $\epsilon,\eta>0$ and $R > D+4\epsilon+2r_0$, the following holds. For every $(\epsilon,R)$-ladder $P_1,Q_1, \ldots, P_k,Q_k$, if $(x,p) \in P_1^{+\eta} \cap Q_1^{+\eta}$ and $(y,q) \in P_k^{+\eta} \cap Q_k^{+ \eta}$ then $d(p,q) \leq 6\eta +\epsilon$.
\end{thm}

\noindent
Our theorem will be an easy consequence of the following crucial observation:

\begin{prop}\label{prop:SquareLeaves}
Let $\mathscr{L}H$ be a finitely generated halo group. Fix three constants $L,D, \epsilon \geq 0$ satisfying $L > D +4 \epsilon+2r_0$. Assume that, for all $R,S \subset H$, $L(R)$ and $L(S)$ commute whenever $d(R,S) \geq D$. Let $\{P_i, i \in \mathbb{Z}_4 \}$ be a collection of leaves such that 
\begin{itemize}
	\item $d(P_i,P_{i+1}) \leq \epsilon$ for every $i \in \mathbb{Z}_4$;
	\item $d(P_i \cap P_{i-1}^{+ \epsilon}, P_i\cap P_{i+1}^{+ \epsilon}) \geq L$ for every $i \in \mathbb{Z}_4$.
\end{itemize}
There exist $\sigma, \alpha, \beta \in L(H)$ such that $\alpha$ and $\beta$ commute, $P_0= \sigma H$, $P_1= \sigma \alpha H$, $P_3= \sigma \beta H$, and $P_2 = \sigma \alpha \beta H= \sigma \beta \alpha H$. 
\end{prop}

\begin{proof}
Let $\sigma \in L(H)$ be such that $P_0 = \sigma H$. Our configuration of leaves is illustrated by Figure~\ref{SquareLeaves}.
\begin{figure}[h!]
\begin{center}
\includegraphics[width=0.6\linewidth]{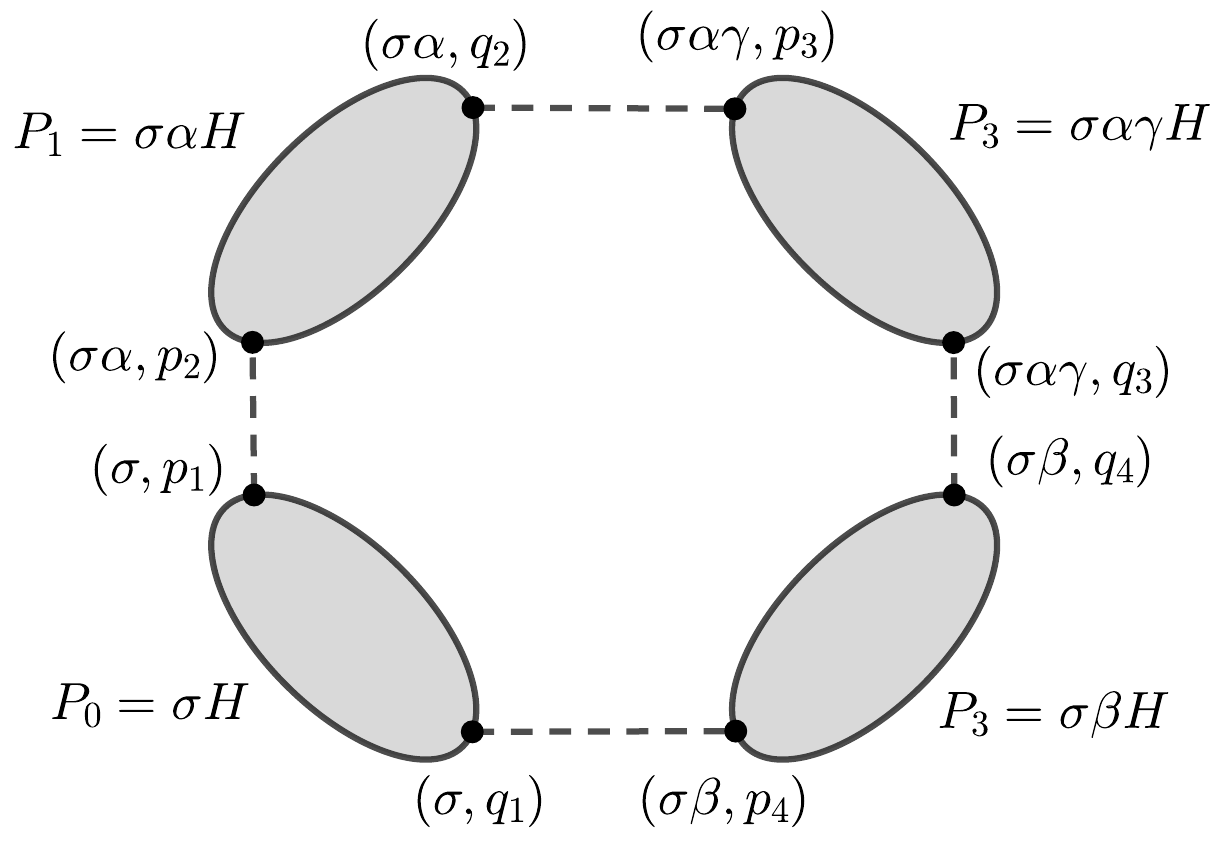}
\caption{Configuration from the proof of Proposition~\ref{prop:SquareLeaves}.}
\label{SquareLeaves}
\end{center}
\end{figure}
Our assumptions imply that $\alpha \in L(B(p_1,\epsilon))$, $\beta \in L(B(q_1,\epsilon))$, $\gamma \in L(B(q_2,\epsilon))$, $\beta \in \alpha \gamma L(B(q_3,\epsilon))$, and $d(p_i,q_i) \geq L$ for every $i \in \mathbb{Z}_4$. 

\medskip \noindent
Observe that, because $d( p_1, q_2) \geq d(p_1,q_1) - d(q_1,q_2) \geq L - \epsilon > D + 2\epsilon$, the fact that $\alpha$ and $\gamma$ respectively belong to $L^+(B(p_1,\epsilon))$ and $L^+(B(q_2,\epsilon))$ implies that $\alpha$ and $\gamma$ commute. Thus, it follows from $\beta \in \alpha \gamma L^+(B(q_3,\epsilon))$ that $\gamma^{-1} \beta \in \alpha L^+(B(q_3,\epsilon))$. Hence
$$\gamma^{-1} \beta \in L^+(B(q_1,\epsilon) \cup B(q_2, \epsilon)) \cap L^+(B(p_1,\epsilon) \cup B(q_3,\epsilon)).$$
But $B(p_1,\epsilon)$ is at distance $>2r_0$ from $B(q_1,\epsilon)$, because $d(p_1,q_1) \geq L > 2 (\epsilon+r_0)$; and at distance $>2r_0$ from $B(q_2,\epsilon)$, because $d(p_1,q_2) \geq d(p_1,q_1)- d(q_1,q_2) \geq L- \epsilon> 2 (\epsilon+r_0)$. Similarly, $B(q_3,\epsilon)$ is at distance $>2r_0$ from $B(q_2,\epsilon)$, because $d(q_2,q_3) \geq d(p_3,q_3)-d(q_2,p_3) \geq L - \epsilon > 2 (\epsilon+r_0)$; and at distance $>2r_0$ from $B(q_1, \epsilon)$, because $d(q_1,q_3) \geq d(p_4,q_4)-d(q_1,p_4)-d(q_3,q_4) \geq L - 2 \epsilon> 2 (\epsilon+r_0)$. Consequently, we must have $\gamma^{-1}\beta \in L(\emptyset)= \{1\}$, i.e.\ $\beta= \gamma$. 

\medskip \noindent
We conclude that $P_2= \sigma \alpha \gamma H = \sigma \alpha \beta H$, which coincides with $\sigma \beta \alpha H$ since we have seen that $\alpha$ commutes with $\gamma= \beta$.
\end{proof}

\begin{proof}[Proof of Theorem~\ref{thm:Ladders}.]
For every $1 \leq i \leq k$, let $a_i$ (resp. $b_i$) denote the element of $L(H)$ labelling $P_i$ (resp. $Q_i$). By applying Proposition~\ref{prop:SquareLeaves} iteratively, we deduce that 
$$a_1^{-1}b_1 = a_2^{-1}b_2= \cdots = a_{k-1}^{-1}b_{k-1}= a_k^{-1}b_k.$$
If a point $(x,p)$ lies in the $\eta$-neighbourhoods of both $P_1=a_1H$ and $Q_1=b_1H$, then there must exist $p' \in H$ such that $d((x,p),(a_1,p')) \leq \eta$ and $d((a_1,p'),Q_1)) \leq 2\eta$. Necessarily, $p'$ must belong to the $2\eta$-neighbourhood of $\mathrm{supp}(a_1^{-1}b_1)$. Since $d(p,p') \leq \eta$, $p$ belongs to $\mathrm{supp}(a_1^{-1}b_1)^{+3\eta}$. Similarly, $q$ belongs to $\mathrm{supp}(a_k^{-1}b_k)^{+3\eta}$. But $\mathrm{supp}(a_1^{-1}b_1)$ coincides with $\mathrm{supp}(a_k^{-1}b_k)$, a subset which has diameter $\leq \epsilon$ since $d(P_1,P_1),d(P_k,P_k) \leq \epsilon$. Consequently, $d(p,q) \leq 6\eta+ \epsilon$. 
\end{proof}

\subsection{A word about stiff groups}\label{section:Morse}

\noindent
In this section, we show that Theorem~\ref{thm:Ladders} is sufficient to deduce that quasi-isometries between specific halo groups are always aptolic (up to finite distance). We will focus on halo groups over finitely generated groups satisfying the following form of rigidity:

\begin{definition}
A metric space $X$ is \emph{stiff} if, for all $K_1,K_2>0$, there exists some $D \geq 0$ such that every $(K_1,K_2)$-quasi-isometry $X \to X$ which coincides with the identity outside some bounded subset automatically lies at distance $\leq K$ from the identity. 
\end{definition}

\noindent
The property is clearly stable under quasi-isometry, so it makes sense to say that a finitely generated groups is (or is not) stiff. 

\begin{thm}\label{thm:Stiff}
Let $\mathscr{M}A, \mathscr{N}B$ be two finitely generated halo products with $\mathscr{M},\mathscr{N}$ large-scale commutative. If $A$ is stiff, then every leaf-preserving quasi-isometry $\mathscr{M}A \to \mathscr{N}B$ lies at finite distance from an aptolic quasi-isometry.
\end{thm}

\noindent
In order to prove our theorem, we start by deducing from Theorem~\ref{thm:Ladders} that a weak form of aptolicity always hold for quasi-isometries between halo groups given by large-scale commutative halos of groups. 

\begin{prop}\label{prop:WeakAptolic}
Let $\mathscr{M}A,\mathscr{N}B$ be two finitely generated halo groups with $\mathscr{M},\mathscr{N}$ large-scale commutative. For every leaf-preserving quasi-isometry $\Psi : \mathscr{M}A \to \mathscr{N}B$, there exist a bijection $\alpha : M(A) \to N(B)$ and quasi-isometries $\beta_c : A \to B$, $c \in M(A)$, such that
\begin{itemize}
	\item $\Psi$ is of the form $(c,p) \mapsto (\alpha(c), \beta_c(p))$;
	\item there exists $K \geq 0$ such that $d(\beta_a(p),\beta_b(p)) \leq K$ for all $a,b \in M(A)$ and for every $p \in A$ outside some bounded subset $S(a,b)$.
\end{itemize}
\end{prop}

\begin{proof}
Because $\Psi$ is leaf-preserving, there exist a bijection $\alpha : M(A) \to N(B)$ and quasi-isometries $\beta_c : A \to B$, $c \in M(A)$, such that $\Psi : (c,p) \mapsto (\alpha(c), \beta_c(p))$. 

\medskip \noindent
Fix two elements $a,b \in M(A)$. Write $a^{-1}b$ as a product $c_1 \cdots c_n$ of elements with small supports. Let $S(a,b)$ denote the $Q$-neighbourhood of the union of the supports of all the $c_i$, with $Q$ very large, and let $p \in A$ be a point outside $S(a,b)$. Also, fix a non-trivial element $m \in M(A)$ with a small support around $p$. Then the leaves $ac_1 \cdots c_i A$, $ac_1 \cdots c_i m A$ for $0 \leq i \leq n$ define a chain of squares. The vertex $(a,p)$ is close to the leaves $aA$ and $amA$, and the vertex $(b,p)$ is close to the leaves $bA$ and $bmA$. Therefore, $(\alpha(a),\beta_a(p))$ is close to the leaves $\alpha(a)B$ and $\alpha(am)B$, and $(\alpha(b),\beta_b(p))$ is close to the leaves $\alpha(b)B$ and $\alpha(bm)B$. In other words, $\beta_a(p)$ is close the support of $\alpha(a)^{-1}\alpha(am)$ and $\beta_b(p)$ is close to the support of $\alpha(b)^{-1}\alpha(bm)$. By applying Proposition~\ref{prop:SquareLeaves} to our chain of squares of leaves $\alpha(ac_1 \cdots c_i)B$, $\alpha(ac_1 \cdots c_im)B$ with $0 \leq i \leq n$, we deduce that $\alpha(a)^{-1}\alpha(b) = \alpha(b)^{-1}\alpha(bm)$. Therefore, $\beta_a(p)$ and $\beta_b(p)$ must be close. 
\end{proof}

\begin{proof}[Proof of Theorem~\ref{thm:Stiff}.]
As a consequence of Proposition~\ref{prop:WeakAptolic}, a quasi-isometry lies at finite distance from a quasi-isometry of the form $(c,p) \mapsto (\alpha(c), \beta_c(p))$ where $\alpha : M(A) \to N(B)$ is a bijection, where each $\beta_c : A \to B$ is a quasi-isometry, and where any two $\beta_c$ coincide outside some bounded subset. Observe that the parameters of the $\beta_c$ are all controlled by the parameters of our initial quasi-isometry. Consequently, because $A$ is stiff, the $\beta_c$ must all lie at a uniform finite distance. Setting for instance $\beta:= \beta_1$, we conclude that our quasi-isometry lies at finite distance from the aptolic quasi-isometry $(c,p) \mapsto (\alpha(c), \beta(p))$. 
\end{proof}

\noindent
Observe that free abelian groups are not stiff. Indeed, fix a compactly supported piecewise linear homeomorphism $f : \mathbb{R}^n \to \mathbb{R}^n$ (distinct from the identity). For every $k \geq 1$, let $D_k$ denote the map $x\mapsto kx$ and set $f_k:= D_k \circ f \circ D_k^{-1}$. Clearly, $f$ is $L$-biLipschitz for some $L>0$, which implies that each $f_k$ is also $L$-biLipschitz. But $d(f_k, \mathrm{Id}) = k \cdot d(f, \mathrm{Id})$ will get arbitrarily large when $k \to + \infty$. Thus, $\mathbb{R}^n$, and a fortiori $\mathbb{Z}^n$, is not stiff. The same phenomenon occurs for many other nilpotent groups.

\medskip \noindent
In the opposite direction, hyperbolic groups are easily seen to be stiff, because a quasi-isometry that fixes pointwise the boundary lies at a finite distance from the identity controlled by the hyperbolicity constant. In the same spirit, but in a more general context:

\begin{prop}\label{prop:MorseLocallyRigid}
Let $X$ be a geodesic metric space admitting cobounded action. If the Morse boundary of $X$ has at least three points, then $X$ is stiff. 
\end{prop}

\noindent
The proposition is well-known by specialists, but we include a proof for completeness. Since Morse boundaries play no role in the rest of the article, we refer to \cite{MR3964503} for the definitions and notations used in the following proof. 

\begin{proof}[Proof of Proposition~\ref{prop:MorseLocallyRigid}.]
Given three pairwise distinct points $a,b,c$ in the Morse boundary $\partial_*X$ of $X$ and a Morse gauge $N$ such that $a,b,c \in \partial_*X^{(N)}$, we know from \cite[Lemma~2.5]{MR3964503} that, given a sufficiently large constant $K \geq 0$ (depending only on $N$), the set $E_{K,N}(a,b,c)$ of all the points lying in the $K$-neighbourhoods of three $N$-Morse geodesics connecting $a,b,c$ is non-empty and has bounded diameter $\leq D_{K,N}$ controlled by $K$ and $N$. 

\medskip \noindent
Now, let $\varphi : X \to X$ be a quasi-isometry which is the identity outside some bounded set. In particular, $\varphi$ induces the identity $\partial_*X \to \partial_*X$. We want to prove that the distance between $\varphi$ and the identity $X \to X$ is controlled by the parameters of $\varphi$ (and $X$). Fix three pairwise distinct points $a,b,c \in \partial_*X$ and a Morse gauge $N$ such that $a,b,c \in \partial_*X^{(N)}$. Up to translating by an isometry of $X$, we can assume that a given vertex $x \in X$ belongs to the $D$-neighbourhood of $E_{K,N}(a,b,c)$ for some $D$ depending only on $X$. There exist a constant $K'\geq K$ and Morse gauge $N' \geq N$ depending only on $K$, $N$, and the parameters of $\varphi$ such that $\varphi$ sends $E_{K,N}(a,b,c)$ in $E_{K',N'}(a,b,c)$. Consequently, $\varphi$ sends $x$ in the $D'$-neighbourhood of $E_{K',N'}(a,b,c)$ for some $D'$ depending only on $D$ and the parameters of $\varphi$. Because $E_{K',N'}(a,b,c)$ contains $E_{K,N}(a,b,c)$ and has diameter $\leq D_{K',N'}$, it follows that
$$d(x, \varphi(x)) \leq D + D_{K',N'} + D'.$$
Thus, $\varphi$ lies at distance $\leq D+D_{K',N'}+D'$ from the identity $X \to X$. 
\end{proof}

\noindent
Acylindrically hyperbolic groups are typical examples of groups with infinite Morse boundaries \cite{MR3519976}, and encompass a wide variety of examples (see for instance \cite{MR3966794} and \cite{MR4057355}). 

\begin{cor}
Finitely generated acylindrically hyperbolic groups are stiff.
\end{cor}

\noindent
However, as already said, basic examples such as free abelian groups are not stiff, motivating the need for the search of another criterion.

\subsection{Altitude of a halo}\label{section:Sufficient}

\noindent
Roughly speaking, we define the \emph{altitude} of a halo group $\mathscr{L}H$ as the least number of ladders needed to connect an arbitrary point of $\mathscr{L}H$ to a point of the leaf $H$. Here, given a $\eta>0$, we say that two points $x,y \in \mathscr{L}H$ are \emph{$\eta$-connected} by a ladder $P_1,Q_1, \ldots, P_k,Q_k$ if $x \in P_1^{+ \eta} \cap Q_1^{+ \eta}$ and $b \in P_k^{+\eta} \cap Q_k^{+ \eta}$. The altitude may be infinite.

\begin{definition}
Let $\mathscr{L}H$ be a finitely generated halo group. Given $\epsilon,\eta,R>0$, the \emph{$(\epsilon,\eta,R)$-altitude} of $\mathscr{L}$ is the minimal number $N \in \mathbb{N} \cup \{\infty\}$ such that, for every $a \in \mathscr{L}H$, there exist $N+1$ points $x_0=a,x_1, \ldots, x_N$ successively $\eta$-connected by an $(\epsilon,R)$-ladder with $x_N \in H$. 
\end{definition}

\noindent
As a rather straightforward consequence of Theorem~\ref{thm:Ladders}:

\begin{cor}\label{cor:AptoAltitude}
Let $\mathscr{M}A,\mathscr{N}B$ be two finitely generated halo groups with $\mathscr{N}$ large-scale commutative. For all $U,V,\epsilon,\epsilon'>0$, there exist $Q >0$ (independent of $\epsilon'$) such that the following holds. If $\mathscr{M}$ has finite $(\epsilon,\epsilon',Q)$-altitude, then every leaf-preserving $(U,V)$-quasi-isometry $\mathscr{M}A \to \mathscr{N}B$ lies at finite distance from an aptolic quasi-isometry. 
\end{cor}

\begin{proof}
Let $R,S,\eta$ denote the constants given by Proposition~\ref{prop:GraphLeaves} and let $D$ denote the constant given by the large-scale commutativity of $\mathscr{N}$. Fix a constant $Q$ satisfying $Q > R(D+4\eta+2r_0+S)$. Let $\Psi : \mathscr{M}A \to \mathscr{N}B$ be a leaf-preserving $(U,V)$-quasi-isometry. There exist a bijection $\alpha : M(A) \to N(B)$ and quasi-isometries $\beta_c : A \to B$, $c \in M(A)$, such that $\Psi : (c,p) \mapsto (\alpha(c), \beta_c(p))$. 

\medskip \noindent
Fix a point $(x,p) \in \mathscr{M}A$. Given the $(\epsilon,\epsilon',Q)$-altitude $N$ of $\mathscr{M}$, we know that there exist $k \leq N+1$ points $z_0=(x,p),z_1, \ldots, z_k:=(1,q)$ such that $z_i$ and $z_{i+1}$ are $\epsilon'$-connected by an $(\epsilon,Q)$-ladder for every $0 \leq i \leq k-1$. By applying Theorem~\ref{thm:Ladders} $k$ times, we know that $d(p,q) \leq (6\epsilon + \epsilon')N$. Next, we deduce from Proposition~\ref{prop:GraphLeaves} that $\Psi(z_i)$ and $\Psi(z_{i+1})$ are $(U\epsilon'+V)$-connected by a $(\eta,Q/R-S)$-ladder. By applying Theorem~\ref{thm:Ladders} $k$ times, it follows that that $d(\beta_x(p), \beta_1(q)) \leq (6\eta + U\epsilon'+V)N$. Hence
$$\begin{array}{lcl} d(\beta_x(p), \beta_1(p)) & \leq & d(\beta_x(p),\beta_1(q)) + d(\beta_1(q),\beta_1(p)) \\ \\ & \leq & (6\eta+U \epsilon'+V)N + UN(6 \epsilon+\epsilon') +V. \end{array}$$
Thus, we have proved that all the quasi-isometries $\beta_c : A \to B$ lies at a uniform bounded distance from $\beta_1$. We conclude that our $\Psi$ lies at finite distance from the aptolic quasi-isometry $(c,p) \mapsto (\alpha(c),\beta(p))$. 
\end{proof}

\noindent
We conclude this section by showing that the halo groups we are the most interested in, namely lamplighter, lampjuggler, lampdesigner, and lampcloner groups, all have finite altitude. We begin with lamplighter groups, the easiest case.

\begin{lemma}\label{lem:LighterAltitude}
Let $F$ and $H$ be two finitely generated groups. For all $\epsilon \geq 1$, $R \geq 0$, and $\eta \geq 2R+\epsilon$, the $(\epsilon,\eta,R)$-altitude of $F \wr H$ is $ \leq 2$.
\end{lemma}

\begin{proof}
Let $(x,p)$ be a point of $F \wr H$. We begin by observing that:

\begin{claim}\label{claim:ParticularCase}
If $x$ is trivial on $B(p,R)$, then $(x,p)$ and $(1,p)$ are $\epsilon$-connected by an $(\epsilon,R)$-ladder.
\end{claim}

\noindent
Write $x^{-1}$ as a product $z_1 \cdots z_n$ such that each $z_i$ has support a single point in the ball $B(p,R)^c$ and has length $\leq \epsilon$ in $F$. Also, fix a colouring $c$ whose support is $\{p\}$ and whose length in $F$ is $\leq \epsilon$. Then 
$$xH, xcH, xz_1H, xz_1cH, \ldots, xz_1 \cdots z_nH=H, xz_1 \cdots z_ncH=cH$$
defines an $(\epsilon,R)$-ladder that $\epsilon$-connects $(x,p)$ and $(1,p)$, concluding the proof of Claim~\ref{claim:ParticularCase}.

\medskip \noindent
Now, let $x'$ denote the colouring that coincides with $x$ on $B(p,R)$ and that is trivial outside $B(p,R)$. We know from Claim~\ref{claim:ParticularCase} that $(x,p)$ and $(x',p)$ are $\epsilon$-connected by an $(\epsilon,R)$-ladder. Given a point $q \in H$ satisfying $d(p,q)=2R$, we apply again Claim~\ref{claim:ParticularCase} in order to deduce that $(x',q)$ and $(1,q)$ are $\epsilon$-connected by an $(\epsilon,R)$-ladder. Because $(x',q)$ lies at distance $2R$ from $(x',p)$ and that $(1,q)$ lies at distance $2R$ from $(1,p)$, we conclude that our two ladders provide a chain of two $(\epsilon,R)$-ladders that $(2R+\epsilon)$-connects $(x,p)$ and $(1,p)$. 
\end{proof}

\noindent
Next, we treat lampjuggler and lampdesigner groups in a uniform way.

\begin{lemma}\label{lem:AltitudeJ}
Let $H$ be a finitely generated one-ended group and $n\geq 1$ an integer. For every $R \geq 0$, there exists a constant $\eta_0(R)$ such that the following holds. For all $\epsilon \geq 1$, $R \geq 0$, and $\eta > \eta_0(R)$, the $(\epsilon,\eta,R)$-altitude of $\circledS_nH$ is $\leq 2$.
\end{lemma}

\begin{lemma}\label{lem:AltitudeD}
Let $F$ be a finite group and $H$ a finitely generated one-ended group. For every $R \geq 0$, there exists a constant $\eta_0(R)$ such that the following holds. For all $\epsilon \geq 1$, $R \geq 0$, and $\eta > \eta_0(R)$, the $(\epsilon,\eta,R)$-altitude of $F \boxplus H$ is $\leq 2$.
\end{lemma}

\begin{proof}[Proof of Lemmas~\ref{lem:AltitudeJ} and~\ref{lem:AltitudeD}.]
Let $\mathscr{L}H$ denote our halo group. If $\mathscr{L}H$ is a lampjuggler group (resp. a lampdesigner group), $L(H)$ is a permutation group of $H \times \{1, \ldots, n\}$ (resp. $H \times F$). For every subset $S \subset H$, we define its \emph{closure} as $\bar{S}:= S \times \{1, \ldots, n\}$ (resp. $S \times F$). Given an $R$, we define $c(R)$ as the diameter of the union of a ball of radius $R$ with the bounded components of its complement. Then, we define $\eta_0(R)$ as the smallest radius of a ball containing $N+2$ pairwise disjoint balls of radius $c(R)$, where $N$ denotes the cardinality of the closure of a ball of radius $c(R)$. From now on, we fix $\epsilon \geq 1$, $R \geq 0$, and $\eta\geq \eta_0(R)$.

\medskip \noindent
Fix an arbitrary point $(\sigma,p) \in \mathscr{L}H$. We denote by $\bar{B}^+(p,R)$ the closure of the union of $B(p,R)$ with the bounded components of its complement in $H$. Assume that there exists a point $q \in \bar{H} \backslash \bar{B}^+(p,R)$, not fixed by $\sigma$, such that $\sigma(q) \notin \bar{B}^+(p,R)$, and let $\tau$ denote the transposition $(q~\sigma(q))$. Observe that $\sigma \tau$ fixes $q$ as well as all the points fixed by $\sigma$. Moreover, because $q$ and $\sigma(q)$ belong to the complement of $\bar{B}^+(p,R)$, which is connected, we can write $\tau$ as a product of permutations whose supports in $L(H)$ are pairs of adjacent vertices of $H$ disjoint from $\bar{B}^+(p,R)$. By iterating the argument, we find permutations $\tau_1, \ldots, \tau_k$ whose supports in $L(H)$ are pairs of adjacent vertices of $H$ disjoint from $\bar{B}^+(p,R)$ such that $\zeta:= \sigma \tau_1 \cdots \tau_k$ fixes every point in the complement of $\bar{B}^+(p,R)$ which is not sent in $\bar{B}^+(p,R)$. 

\medskip \noindent
Fix a permutation $\tau$ that switches $p$ with one of its neighbours, and observe that 
$$\sigma H, \sigma \tau H, \sigma \tau_1H, \sigma \tau_1 \tau H, \ldots, \sigma \tau_1 \cdots \tau_k H = \zeta H, \sigma \tau_1 \cdots \tau_k \tau  H= \zeta \tau H$$
defines a $(1,R)$-ladder $1$-connecting $(\sigma,p)$ and $(\zeta,p)$. Now, by definition of $\eta_0(R)$ and as a consequence of the pigeonhole principle, the ball $B(p, \eta_0(R))$ contains a ball $B(q,c(R))$ such that $\zeta$ fixes all the points in $\bar{B}(q,c(R))$, and in particular in $\bar{B}^+(q,R)$. Because $d(p,q) \leq \eta_0(R)$, our previous $(1,R)$-ladder $(\eta_0(R)+1)$-connects $(\sigma,p)$ and $(\zeta,q)$.

\medskip \noindent
The fact that $\zeta$ fixes all the points in $\bar{B}^+(q,R)$ implies that, by applying the construction above to $(\zeta,q)$, there exists a $(1,R)$-ladder $1$-connecting $(\zeta,q)$ and $(1,q)$. Thus, we have proved that $(\sigma,p)$ can be $(\eta_0(R)+1)$-connected to $(1,q)$ by a sequence of two $(1,R)$-ladders. 
\end{proof}

\noindent
Finally, let us show that lampcloner groups have finite altitude.

\begin{lemma}\label{lem:AltitudeCloner}
Let $\mathfrak{k}$ be a finite field and $H$ a finitely generated one-ended group. For every $R \geq 0$, there exists a constant $\eta_0(R)$ such that the following holds. For all $\epsilon \geq 1$, $R \geq 0$, and $\eta > \eta_0(R)$, the $(\epsilon,\eta,R)$-altitude of $\oslash_\mathfrak{k} H$ is $\leq 4$.
\end{lemma}

\begin{proof}
Recall from Section~\ref{section:Halo} that a point $(\varphi,p)$ in $\oslash_\mathfrak{k}H$ is described by a colouring of $H$, such that each lamp $q \in H$ is labelled by the vector $\varphi(q)$, and by an arrow pointing to $p$. For all but finitely many values of $q$, we have $\varphi(q)=\mathbf{q}$ where we denote by $\mathbf{q}$ the vector of the canonical basis indexed by $q$. Moving to a neighbour in $\oslash_\mathfrak{k} H$ amounts either to moving the arrow from $p$ to a neighbour; or to multiplying the label $\varphi(p)$ by a non-trivial element of $\mathfrak{k}$; or to \emph{cloning} (i.e.\ duplicating) the vector $\varphi(p)$ and adding it to the label of a neighbour of $p$ after multiplication by a non-trivial element of $\mathfrak{k}$. By abuse of notation, given a subset $S \subset H$, we denote by $\langle S \rangle$ the subspace spanned by the vectors in $\{ \mathbf{q} \mid q \in S \}$. 

\medskip \noindent
For every $p \in H$ and $R \geq 0$, we denote by $B^+(p,R)$ the union of the ball $B(p,R)$ with the bounded components of its complement. We denote by $c(R)$ the diameter of $B^+(p,R)$, which does not depend on $R$. Observe that, because $H$ is one-ended, every $B^+(p,R)$ has connected complement. From now on, we fix an $R \geq 0$ and we want to prove that the $(1,2c(R)+1,R)$-altitude of $\oslash_\mathfrak{k} H$ is at most $4$.

\medskip \noindent
We begin by recording the following construction, which will be applied several times in the rest of the proof.

\begin{claim}\label{claim:LampLadder}
Let $(\varphi,p) \in \oslash_\mathfrak{k}H$. If $\psi \in E(H,\mathfrak{k})$ satisfies the conditions:
\begin{itemize}
	\item $\varphi$ and $\psi$ agree on $B^+(p,R)$
	\item for every $q \notin B^+(p,R)$, $\psi(q) \in \langle \varphi (B^+(p,R)^c) \rangle$,
\end{itemize}
then $(\varphi,p)$ and $(\psi,p)$ are $1$-connected by a $(1,R)$-ladder. 
\end{claim}

\noindent
As already noticed in Remark~\ref{rem:subgraphs}, it follows from the connectedness of $B^+(p,R)^c$ that $L(B^+(p,R)^c)$ is generated by transvections and diagonal matrices supported in $B^+(p,R)^c$. Now, we make the following trivial observation: 
given a subset $Y\in H$,  $\varphi L(Y)$ coincides with the set of the $\psi\in E(H, \mathfrak{k})$ that coincide with $\varphi$ on $Y^c$ and such that for all $q\in Y$, $\psi(q) \in \langle \varphi (Y) \rangle$. Applying this to $Y=B^+(p,R)^c$, we find transvections between adjacent vertices in $B^+(p,R)^c$ and diagonal matrices supported by vertices in $B^+(p,R)^c$, say $e_1, \ldots, e_k$, such that $\psi = \varphi e_1 \cdots e_k$. Fixing a non-trivial $\lambda \in \mathfrak{k}$, we obtain a $(1,R)$-ladder
$$\varphi, \varphi \delta_p(\lambda), \varphi e_1, \varphi e_1\delta_p(\lambda), \ldots, \varphi e_1 \cdots e_k= \psi, \varphi e_1 \cdots e_k \delta_p(\lambda)=\psi \delta_p(\lambda)$$
that $1$-connects $(\varphi,p)$ and $(\psi,p)$, concluding the proof of Claim~\ref{claim:LampLadder}. 

\medskip \noindent
We shall apply the claim in the following situation. 

\begin{claim}\label{claim:LampLadderBis}
Let $(\varphi,p) \in \oslash_\mathfrak{k}H$. There exist a subset $A\subset H$ of cardinality $\leq |B^+(p,R)|$ and an element $\psi \in E(H,\mathfrak{k})$ such that $\psi(q)=\varphi(q)$ for all $q\in B^+(p,R)$ and $\psi(q)\in \mathbf{q}+ \langle \varphi (B^+(p,R)) \rangle$ for all $q\notin B^+(p,R)\cup A$. Moreover, $(\varphi,p)$ and $(\psi,p)$ are $1$-connected by a $(1,R)$-ladder.
\end{claim}

\noindent
We first prove that there exists some $A \subset H$ of cardinality $\leq |B^+(p,R)|$ such that $\{\mathbf{q}\mid q\notin A\}$ is linearly independent modulo $\langle \varphi (B^+(p,R)) \rangle$. By the incomplete basis theorem, any basis of $\langle \varphi (B^+(p,R)) \rangle$ can be extended to a basis of the whole vector space by a subset $C \subset \{ \mathbf{q} \mid q \in H\}$. Let $A \subset H$ be such that $C= \{ \mathbf{q} \mid q \notin A\}$.
By construction, $\{\mathbf{q}\mid q\notin A\}$ is linearly independent modulo $\langle \varphi (B^+(p,R)) \rangle$. It remains to notice that $|A|=\mathrm{codim}(\langle C \rangle) = \dim (\langle \varphi(B^+(p,R)) \rangle)= |B^+(p,R)|$, proving our observation. 

\medskip \noindent
Now, we set $\psi(q):=\varphi(q)$ for all $q\in B^+(p,R)$. For $q \notin A \cup B^+(p,R)$, we know that there exists a unique pair of elements $w\in  \langle \varphi (B^+(p,R)^c) \rangle$ and $v\in \langle \varphi (B^+(p,R)) \rangle$ such that $w=\mathbf{q}+v$, and we set $\psi(q):=w$. So far, we have defined $\psi$ on $A^c\cup  B^+(p,R)$, and, by the defining property of $A$, the family $\{\psi(q)\mid q\in A^c\cup  B^+(p,R)\}$ is free. Our goal now is to extend $\psi$ on $A':=A\cap B^+(p,R)^c$ in order to get an element of $E(H,\mathfrak{k})$ satisfying $\psi(q)\in \langle \varphi (B^+(p,R)^c) \rangle$ for all $q\in A$.

\medskip \noindent 
It suffices to define $\psi$ on $A'$ such that $\{\psi(q)\mid q\in B^+(p,R)^c\}$ is a basis of $\langle \varphi (B^+(p,R)^c) \rangle$. In order to do so, we merely need to show that the codimension of  $\langle\psi(B^+(p,R)^c\setminus A')\rangle$ in $\langle \varphi (B^+(p,R)^c) \rangle$ is exactly the size of $A'$. Since  $\varphi$ is contained in some $E(Z,\mathfrak{k})$ for a finite subset $Z\in H$, we can assume without loss of generality that $H$ is finite, in which case the codimension of  $\langle\psi(B^+(p,R)^c\setminus A')\rangle$ in $\langle \varphi (B^+(p,R)^c) \rangle$ is simply the difference between their respective dimensions, which are respectively $|B^+(p,R)^c\setminus A'|$ and $|B^+(p,R)^c|$. Hence the codimension is $|A'|$, as required.

\medskip \noindent
Thanks to Claim~\ref{claim:LampLadder}, we conclude the proof of Claim~\ref{claim:LampLadderBis}.

\medskip \noindent
We are now ready to compute the altitude of $\oslash_\mathfrak{k} H$. So we fix an arbitrary element $(\varphi,p) \in \oslash_\mathfrak{k}H$, and we want to construct a short sequence of ladders connecting $(\varphi,p)$ to the coset $H$. According to Claim~\ref{claim:LampLadderBis}:
\begin{itemize}
	\item There exist a finite subset $A_0\subset H$ of cardinality $\leq |B^+(p,R)|$ and an element $\varphi_0 \in E(H,\mathfrak{k})$ such that $\varphi_0(q) \in \mathbf{q} + \langle \varphi(B^+(p,R)) \rangle$ for all $q\notin B^+(p,R)\cup A_0$. Moreover, $(\varphi,p)$ is $1$-connected to $(\varphi_0,p)$ by a $(1,R)$-ladder.	
\end{itemize}
Next, notice that, by the pigeonhole principle, there exists $R_1$ only depending on $R$ such that there is a point $p_1 \in H\setminus A_0$ at distance $R_1$ from $p$ such that 
$$B^+(p_1,R)\cap (A_0\cup B^+(p,R))=\emptyset.$$
Let us construct a new element of $E(H,\mathfrak{k})$ satisfying the following properties.
\begin{itemize}
	\item There exists $\varphi_1 \in E(H,\mathfrak{k})$ such that $\varphi_1(q)= \mathbf{q}$ for every $q \notin D:= B^+(p_1,R) \cup B^+(p,R) \cup A_0$. Moreover, $(\varphi_0,p_1)$ is $1$-connected to $(\varphi_1,p_1)$ by a $(1,R)$-ladder. 
\end{itemize}
Set $\varphi_1(q)= \mathbf{q}$ for $q \notin D$ and $\varphi_1(q)= \varphi_0(q)$ for $q \in B^+(p_1,R)$. So far, $\varphi_1$ is defined on the complement of $C:= B^+(p,R) \cup A_0$. Once again, we aim to extend $\varphi_1$ to an element of $E(H,\mathfrak{k})$ such that $ \langle\varphi_0(B^+(p_1,R)^c) \rangle= \langle\varphi_1(B^+(p_1,R)^c) \rangle$. This will allow us to apply Claim~\ref{claim:LampLadder} to connect $\varphi_0$ to $\varphi_1$ by a ladder, as desired. 

\medskip \noindent
First of all, notice that $\langle \varphi_0(B^+(p_1,R)^c \cap C^c) \rangle \subset \langle \varphi_1(B^+(p_1,R)^c)$. Indeed, for every $q \in B^+(p_1,R)^c \cap C^c = D^c$, we have $\varphi_0(q) \in \mathbf{q} + \langle \varphi(B^+(p,R)) \rangle$, hence
$$\varphi_1(q)= \mathbf{q} \in \varphi_0(q) + \langle \varphi(B^+(p,R)) \rangle \subset \langle \varphi_0(B^+(p_1,R)^c) \rangle.$$
In order to conclude, we just need to show that the codimension of $\langle \varphi_1(B^+(p_1,R)^c \cap C^c) \rangle$ in $\langle \varphi_0(B^+(p_1,R)^c) \rangle$ is exactly $|C|$. Since  $\varphi_0$ is contained in some $E(Z,\mathfrak{k})$ for a finite subset $Z\in H$, we can assume without loss of generality that $H$ is finite, in which case the codimension of $\langle \varphi_1(B^+(p_1,R)^c \cap C^c) \rangle$ in $\langle \varphi_0(B^+(p_1,R)^c) \rangle$ is just the difference of the two dimensions, namely
$$| B^+(p_1,R)^c| - |B^+(p_1,R)^c \cap C^c| = |C|$$
as desired. 

\medskip \noindent
By the pigeonwhole principle,  there exists $R_2$ only depending on $R$ such that, for some point $p_2 \in H$ at distance $\leq R_2$ from $p$, $B^+(p_2,R)$ is disjoint from $D$.

\begin{itemize}
	\item There exists $\varphi_2 \in E(H,\mathfrak{k})$ such that $\varphi_2(q) = \mathbf{q}$ for every $q \notin D$ and $\varphi_2(q) \in \mathbf{q} + \langle D \cup B^+(p_2,R) \rangle$ for every $q \in D$. Moreover, $(\varphi_2,p_2)$ and $(\varphi_1,p_2)$ are $1$-connected by a $(1,R)$-ladder.
\end{itemize}
Recall that, by construction, $\varphi_1(q)=\mathbf{q}$ on $D^c$. This implies that we can find a basis of $\langle \varphi_1(B^+(p_2,R)^c) \rangle$ of the form $(\mathbf{q}+a(q))_{q\in B^+(p_2,R)^c}$ where $a(q)$ belongs to $\langle D\cup B^+(p_2,R) \rangle$ and is trivial whenever $q\in D^c$. We set $\varphi_2(q)=\mathbf{q}+a(q)$ for $q\in B^+(p_2,R)^c$.
Extending $\varphi_2$ such that $\varphi_2(q)=\varphi_1(q)=\mathbf{q}$ for $q\in  B^+(p_2,R)$ defines an element of $E(H,\mathfrak{k})$ such that $\langle \varphi_2(B^+(p_2,R)^c) \rangle=\langle \varphi_1(B^+(p_2,R)^c) \rangle$. We deduce from
Claim~\ref{claim:LampLadder} that $(\varphi_2,p_2)$ is $1$-connected to $(\varphi_1,p_2)$ by a $(1,R)$-ladder. 

\medskip \noindent
Finally, notice that, by the pigeonwhole principle, there exists $R_3$ only depending on $R$ such that some point $p_3 \in H$ at distance $\leq R_3$ from $p$ is such that $B(p_3,R)^+$ is disjoint from $D\cup B^+(p_2,R)$. The subspace spanned by $\varphi_2(B^+(p_3,R)^c)$ coincides with the space spanned by the basis vectors $\mathbf{q}$ for $q\in B^+(p_3,R)^c$. Moreover, $\varphi_2(q)=\mathbf{q}$ for $q\in B^+(p_3,R)$. So Claim~\ref{claim:LampLadder} applies and shows that $(\varphi_2,p_3)$ is $1$-connected to $(\mathrm{id},p_3)$ by a $(1,R)$-ladder.  

\medskip \noindent
In summary, we have constructed a sequence a points
$$(\varphi,p),(\varphi_0,p), (\varphi_1,p_1), (\varphi_2,p_2), (\mathrm{id},p_3)$$
such that any two consecutive points are $\max\{R_1,R_2,2c(R)+1)\}$-connected by a $(1,R)$-ladder. This concludes the proof of Lemma \ref{lem:AltitudeCloner}.
\end{proof}

\subsection{Aptolic coarse embeddings}\label{section:CoarseEmb}

\noindent
In the previous subsections, our goal was to show that, under reasonable assumptions, leaf-preserving quasi-isometries between halo products are aptolic. In this last subsections, we shift our attention to coarse embeddings.

\medskip \noindent
Given two finitely generated halo products $\mathscr{M}A,\mathscr{N}B$, a coarse embedding $\rho : \mathscr{M}A \to \mathscr{N}B$ is \emph{aptolic} if there exists an injective map $\alpha : M(A) \to N(B)$ and a coarse embedding $\beta : A \to B$ such that
$$\rho : (c,p) \mapsto (\alpha(c), \beta(p)), \ (c,p) \in \mathscr{M}A.$$
Our key observation is that, given a coarse embedding $\rho : \mathscr{M}A \to \mathscr{N}B$ that sends leaves into leaves, if $\rho$ is not aptolic then it has to send a large ``sub-halo product'' of $\mathscr{M}A$ inside a single leaf of $\mathscr{N}B$, i.e.\ a copy of $B$. In full generality, the structure of this subspace is not clear, unless $\mathscr{M}$ is large-scale commutative, in which case our subspace coarsely coincides with a \emph{wreath product of graphs}, defined as follows:

\begin{definition}\label{def:WPGraphs}
Given a pointed graph $(X,o)$ and a graph $Y$, the \emph{wreath product} $(X,o) \wr Y$ is the graph 
\begin{itemize}
	\item whose vertices are the pairs $(c,y)$ with $y \in Y$ and with $c: Y \to X$ satisfying $c(z)=o$ for all but finitely many $z \in Y$;
	\item whose edges connect $(c_1,p_1)$ and $(c_2,p_2)$ whenever $c_1,c_2$ differ only at $p_1=p_2$ where they adjacent values in $Y$ or when $c_1=c_2$ and $p_1,p_2$ are adjacent in $X$.
\end{itemize}
When $X$ is the complete graph $K_n$, $(X,o) \wr Y$ is the \emph{lamplighter graph} $\mathcal{L}_n(Y)$. For simplicity, we note $\mathcal{L}(Y):= \mathcal{L}_2(Y)$. 
\end{definition}

\noindent
As a consequence, if we forbid $B$ from containing coarse copies of specific subspaces, we are able to prove that leaf-preserving coarse embeddings are automatically aptolic.

\begin{thm}\label{thm:AptoCoarseEmbedding}
Let $\mathscr{M}A,\mathscr{N}B$ be two finitely generated halo products such that $\mathscr{N}$ is large-scale commutative and $\mathscr{M}$ has a finite altitude. Assume that at least one of the following conditions is satisfied:
\begin{itemize}
	\item there is no coarse (resp.\ quasi-isometric) embedding of the pointed sum $A \vee A$ in~$B$;
	\item $\mathscr{M}$ is large-scale commutative and there exists a finitely generated subgroup $G \leq M(A)$ such that there is no coarse (resp.\ quasi-isometric) embedding of $(G',o) \wr A'$ in $B$ for some graph of bounded degree $A'$ quasi-isometric to~$A$ and some locally finite Cayley graph $G'$ of $G$. 
\end{itemize}
Then every coarse (resp.\ quasi-isometric) embedding $\rho : \mathscr{M}A \to \mathscr{N}B$ that sends $A$-cosets inside $B$-cosets is aptolic.
\end{thm}

\noindent
The key point in the proof of Theorem~\ref{thm:AptoCoarseEmbedding} is that a leaf-preserving coarse embedding always sends a ladder either bijectively to another ladder or entirely inside a single leaf. This will be a consequence of the following proposition:

\begin{prop}\label{prop:ImageFourCycle}
Let $\mathscr{M}A, \mathscr{N}B$ be two finitely generated halo groups and $\rho : \mathscr{M}A \to \mathscr{N}B$ a coarse embedding that sends $A$-cosets in $B$-cosets. For every $\epsilon>0$, there exists some $R>0$ such that the following holds. If $L_1, \ldots, L_4$ is an $R$-obtuse $4$-cycle in $\mathscr{G}_\epsilon(\mathscr{M}A)$, then $\rho(L_1), \ldots, \rho(L_4)$ are contained in pairwise distinct $B$-cosets or they are all contained in the same $B$-coset.
\end{prop}

\noindent
We begin by proving the following easy observation:

\begin{lemma}\label{lem:No3Cycle}
Let $\mathscr{L}H$ be a finitely generated halo product. For all $\epsilon>0$ and $R> 3\epsilon +2r_0$, the graph of leaves $\mathscr{G}_\epsilon =\mathscr{G}_\epsilon(\mathscr{L}H)$ does not contain an $R$-obtuse $3$-cycle.
\end{lemma}

\begin{proof}
Assume for contradiction that $\mathscr{G}_\epsilon$ contains an $R$-obtuse $3$-cycle. So there exist three leaves $H$, $aH$, $bH$ and points $(1,1),(1,t) \in H$, $(a,p),(a,q) \in aH$, $(b,r),(b,s) \in bH$ such that
\begin{itemize}
	\item the distances $d((1,1),(a,p))$, $d((a,q),(b,r))$, and $d((b,s),(1,t))$ are $\leq \epsilon$;
	\item the distances $d(1,t)$, $d(p,q)$, and $d(r,s)$ are $\geq R$.
\end{itemize}
The first item implies that $a \in L^+(B(1,\epsilon))$, $b \in a L^+(B(q,\epsilon))$, and $1 \in b L^+(B(s,\epsilon))$. (Where our notation follows Convention~\ref{Conv}.) Notice that $\mathrm{supp}(a)$ is disjoint from $B(q,\epsilon+r_0)$ because
$$d(1,q) \geq d(p,q)-d(1,p) \geq R- \epsilon> 2(\epsilon+r_0).$$
Therefore, Fact~\ref{fact:DisjointSupp} applies and shows that $\mathrm{supp}(a) \subset \mathrm{supp}(b)$. But we know that $\mathrm{supp}(b)$ is contained in $B(s,\epsilon+r_0)$. Hence
$$a \in L \left( B(1,\epsilon+r_0) \right) \cap L \left( B(s,\epsilon+r_0) \right) = L \left( B(1,\epsilon+r_0) \cap B(s,\epsilon+r_0) \right) = L(\emptyset)= \{1\},$$
where the fact that $B(1,\epsilon+r_0)$ and $B(s,\epsilon+r_0)$ are disjoint is justified by
$$d(1,s) \geq d(1,t)-d(s,t) \geq R - \epsilon > 2(\epsilon+r_0).$$
But, if $a=1$, then the leaves $H$, $aH$, $bH$ do not define a $3$-cycle in the graph of leaves. 
\end{proof}

\begin{proof}[Proof of Proposition~\ref{prop:ImageFourCycle}.]
For every $1 \leq i \leq 4$, fix two points $a_i=(c_i,p_i)$ and $b_i = (c_i,q_i)$ in $L_i$ such that $d(a_j,b_{j-1}) \leq \epsilon$ for every $j$ mod $4$. Also, let $\mu,\nu$ be two non-decreasing functions such that $\rho$ is a $(\mu,\nu)$-coarse embedding. From now on, we fix an arbitrary $\epsilon>0$ and an $R \geq 2(\epsilon+2r_0)$ sufficiently large so that 
\begin{itemize}
	\item for any two distinct $B$-cosets $P,Q$, the diameter of $P^{+ \nu(\epsilon)} \cap Q^{+\nu(\epsilon)}$ is $<\mu(R)$;
	\item there is no $\mu(R)$-obtuse $3$-cycle in $\mathscr{G}_{\nu(\epsilon)}(\mathscr{N}B)$.
\end{itemize}
This is possible according to Lemma~\ref{lem:No3Cycle} and Lemma~\ref{lemma:InterBounded}.

\medskip \noindent
First, consider the case where two opposite $A$-cosets in our $4$-cycle are sent by $\rho$ in the same $B$-coset, say $\rho(L_1),\rho(L_3) \subset B$. Let $cB$ denote the $B$-coset containing $\rho(L_4)$. Notice that 
$$\rho(a_4), \rho(b_4) \in B^{+\nu(\epsilon)} \cap cB$$
and
$$d(\rho(a_4), \rho(b_4)) \geq \mu (d(a_4,b_4)) \geq \mu(R).$$
Because we chose $R$ is sufficiently large, it follows that $cB=B$. Thus, $L_4$ is sent under $\rho$ in $B$ as well. The same argument implies the same conclusion for $L_2$.

\medskip \noindent
Next, consider the case where two consecutive $A$-cosets in our $4$-cycle are sent by $\rho$ in the same $B$-coset, say $\rho(L_1),\rho(L_2) \subset B$. If $\rho(L_3)$ or $\rho(L_4)$ is also contained in $B$, then the problem reduces the previous case. So, from now on, we assume that $\rho(L_3)$ and $\rho(L_4)$ are sent inside $B$-cosets distinct from $B$. If they are moreover distinct, then $B$ and the two $B$-cosets containing $\rho(L_3)$ and $\rho(L_3)$ define a $\mu(R)$-obtuse $3$-cycle in $\mathscr{G}_{\nu(\epsilon)}(\mathscr{N}B)$. This is not possible because we chose $R$ sufficiently large. Therefore, it only remains to consider the case where $\rho(L_3)$ and $\rho(L_4)$ are sent in the same $B$-coset, say $cB$. Notice that
$$\rho(a_3), \rho(b_4) \in B^{+\nu(\epsilon)} \cap cB$$
and 
$$d(\rho(a_3), \rho(b_4)) \geq \mu(d(a_3,b_4)) \geq \mu(2(R-\epsilon-2r_0)) \geq \mu(R),$$
where the inequality $d(a_3,b_4) \geq 2(R-\epsilon-2r_0)$ is justified as follows. We have
$$\begin{array}{lcl} d(a_3,b_4) & \geq & d(p_3, \mathrm{supp}(c_3^{-1} c_4))+ d(q_4,\mathrm{supp}(c_3^{-1} c_4)) - 2r_0\\ \\ & \geq & d(p_3,q_3)-d(q_3,\mathrm{supp}(c_3^{-1} c_4)) + d(p_4,q_4)- d(p_4,\mathrm{supp}(c_3^{-1} c_4)) - 2r_0.\end{array}$$
But we know that
$$d(p_3,q_3) = d(a_3,b_3) \geq R \text{ and similarly } d(p_4,q_4) \geq R,$$
and
$$d(q_3,\mathrm{supp}(c_3^{-1} c_4))-r_0 \leq d(b_3,a_4) \leq \epsilon \text{ and similarly } d(p_4,\mathrm{supp}(c_3^{-1} c_4)) -2r_0\leq \epsilon.$$
Then, the desired inequality $d(a_3,b_4) \geq 2(R-\epsilon-2r_0)$ follows. Thus, we have shown that $B^{+\nu(\epsilon)} \cap cB$ has diameter at least $\mu(R)$. Because we chose $R$ sufficiently large, we conclude that $B=cB$.
\end{proof}

\begin{proof}[Proof of Theorem~\ref{thm:AptoCoarseEmbedding}.]
Let $\alpha : M(A) \to N(B)$ be the map such that, for every $c \in M(A)$, $\rho$ sends the coset $cA$ inside the coset $\alpha(c)B$. We distinguish two cases depending on whether or not $\alpha$ is injective.

\medskip \noindent
First, let us assume that $\alpha$ is injective. Then $\rho$ naturally sends ladders of $A$-cosets to ladders of $B$-cosets, and we can reproduce the proof of Corollary~\ref{cor:AptoAltitude} in order to conclude that $\rho$ is aptolic. 

\medskip \noindent
Next, assume that $\alpha$ is not injective. Up to pre- and post-composing $\rho$ with translations, let us assume for simplicity that $\alpha(1)=\alpha(c)=1$ for some non-trivial $c \in M(A)$. Then $\rho$ maps $A \cup cA$, which coarsely coincides with the pointed sum $A \vee A$, inside $B$. Thus, if the first item from the statement of Theorem~\ref{thm:AptoCoarseEmbedding} holds, we know that $\alpha$ has to be injective. From now on, we assume that $\mathscr{M}$ is large-scale commutative.

\medskip \noindent
Let $G \leq M(A)$ be a finitely generated group, with a fixed finite generating set. Let $r_0$ be defined as in Convention~\ref{Conv}. Up to increasing $r_0$, we assume that $G$ is contained in $M(B(1,r_0)$. Set $\epsilon:=d(A,cA)$ and let $R$ be the constant given by Proposition~\ref{prop:ImageFourCycle} with respect to $\epsilon$ and $\mathscr{N}B$. Up to increasing $R$ if necessary, we assume that any two elements of $M(A)$ with supports at distance $\geq R$ commute. Also, fix a set of points $Y \subset A \backslash \mathrm{supp}(c)^{+R+r_0}$, pairwise at distance $>R+2r_0$, and such that every point of $A \backslash \mathrm{supp}(c)^{+R}$ lies in the $2(R+2r_0)$-neighbourhood of $Y$. Fix a point $y_0 \in Y$. 

\begin{claim}\label{claim:InTheSameCoset}
The equality $\alpha(f)=1$ holds for every $f \in M(A)$ whose support is contained in $\bigcup_{y \in Y \backslash \{y_0\}} B(y,r_0)$. 
\end{claim}

\noindent
So let us fix such an $f$. Notice that $f$ can be written as a product of elements (say syllables) supported on balls of radius $r_0$ centred at a point of $Y \backslash \{y_0\}$. Each factor in this product can be decomposed as a product of elements (say letters) of lengths $\leq \epsilon$ in $A$. Up to adding possible cancellations, we can assume that all these products have the same number of letters. Because any two distinct syllables commute, we can rearrange the letters in order to write $f$ as a product $f_1 \cdots f_{k}$ such that:
\begin{itemize}
	\item each $f_i$ is non-trivial and supported on a ball of radius $r_0$ centred at a point of $Y \backslash \{y_0\}$;
	\item for every $1 \leq i \leq k-1$, the supports of $f_i$ and $f_{i+1}$ are contained in balls centred at distinct points of $Y$.
\end{itemize}
Now, fix a non-trivial element $g \in M(A)$ supported on a ball of radius $r_0$ centred at $y_0$ and of length $\leq \epsilon$ in $A$. Essentially by construction, we get an $(\epsilon,R)$-ladder

\begin{center}
\includegraphics[width=0.8\linewidth]{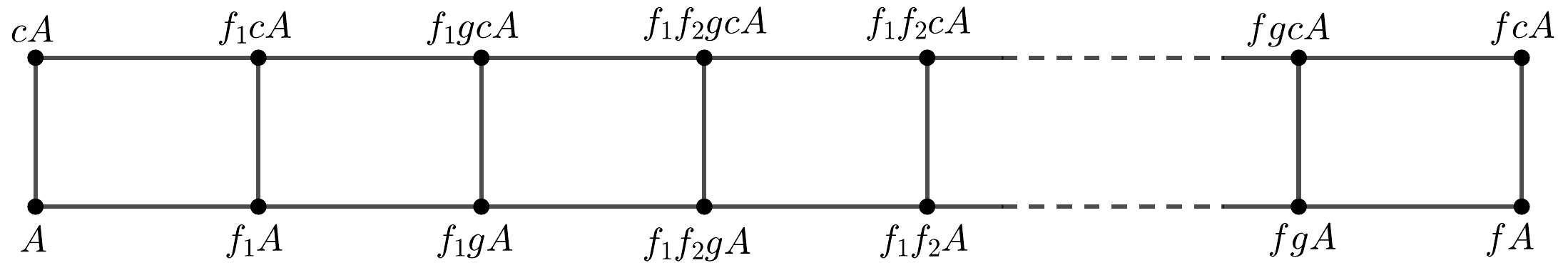}
\end{center}

\noindent
Because $\rho$ sends $A$ and $cA$ in the same $B$-cosets, we can apply Proposition~\ref{prop:ImageFourCycle} iteratively in order to deduce that $fA$ is sent in the same $B$-coset, i.e.\ $\alpha(f)=1$. This concludes the proof of Claim~\ref{claim:InTheSameCoset}.

\medskip \noindent
Now, let $A'$ be the graph whose vertex-set is $Y$ and whose edges connect two vertices $y_1,y_2$ whenever they lie at distance at most $2(R+2r_0)+1$ in $Y$. Our graph $A'$ is connected and quasi-isometric to $A$ because $Y$ is quasi-dense in $A$. For every $y \in Y$, fix a non-trivial element $c_y \in M(A)$ supported in $B(y,r_0)$. Define the map
$$\eta : \left\{ \begin{array}{ccc} (\mathrm{Cayl}(G),1) \wr A' & \to & \mathscr{M}A \\ (c,a) & \mapsto & \left( y \mapsto \left\{ \begin{array}{cl} c(y) & \text{if } y \in Y \\ 1 & \text{otherwise} \end{array} \right., \ a \right) \end{array} \right..$$
Notice that, because the vertices in $Y$ are pairwise at distance $>R$, the copies of $\mathrm{Cayl}(G)$ in $(\mathrm{Cayl}(G),1) \wr A' $ are sent under $\eta$ in pairwise commuting copies of $G$ in $M(A)$. It is clear that $\eta$ is a quasi-isometric embedding. Because Claim~\ref{claim:InTheSameCoset} implies that the image of $\rho \circ \eta$ is contained in $B$, we conclude that $\rho \circ \eta$ induces a coarse (resp.\ quasi-isometric) embedding from $(\mathrm{Cayl}(G),1) \wr A'$ to $B$.
\end{proof}

\begin{cor}
Let $\mathscr{M}A,\mathscr{N}B$ be two finitely generated halo products with $M(A)$ and $N(B)$ locally finite. Assume that $A$ satisfies the thick bigon property, that $\mathscr{N}$ is full, and that at least one of the following two conditions is satisfies:
\begin{itemize}
	\item there is no coarse (resp.\ quasi-isometric) embedding of the pointed sum $A \vee A$ in~$B$;
	\item $\mathscr{M}$ is large-scale commutative and there exists a finitely generated subgroup $G \leq M(A)$ such that there is no coarse (resp.\ quasi-isometric) embedding of $(G',o) \wr A'$ in $B$ for some graph of bounded degree $A'$ quasi-isometric to~$A$ and some locally finite Cayley graph $G'$ of $G$. 
\end{itemize}
Then every coarse (resp.\ quasi-isometric) embedding $\mathscr{M}A \to \mathscr{N}B$ is aptolic.
\end{cor}

\begin{proof}
Let $\rho : \mathscr{M}A \to \mathscr{N}B$ be a coarse (resp.\ quasi-isometric) embedding. We know from Theorem~\ref{thm:InALeaf} that, up to finite distance, $\rho$ sends $A$-cosets inside $B$-cosets. Then Theorem~\ref{thm:AptoCoarseEmbedding} yields the desired conclusion.
\end{proof}

\noindent 
In view of Theorem~\ref{thm:AptoCoarseEmbedding}, it is natural to ask how one can determine whether or not a given finitely generated group contains a coarsely or quasi-isometrically embedded copy of a wreath product (of graphs). For instance, a group of subexponential growth does not contain a coarsely embedded copy of a wreath product over an unbounded graph, because such a graph must have exponential growth. In the realm of amenable groups, F\o lner functions \cite{MR2011120} also provide useful invariants. We conclude this section by showing that, if a finitely generated group has finite-dimensional asymptotic cones, which is rather common, then it cannot contain a quasi-isometrically embedded copy of a wreath product over a graph of superlinear growth.

\begin{prop}\label{prop:cones}
Let $X,Y$ be two graphs of bounded degree. Fix a basepoint $x \in X$ and assume that $Y$ is quasi-isometric to a finitely generated group. Also, assume that $X$ contains at least two vertices and that $Y$ is infinite. The asymptotic cones of $W:= (X,x) \wr Y$ have finite topological dimension if and only if $X$ is bounded and $Y$ is bounded or a quasi-line.
\end{prop}

\begin{proof}
If $X$ is infinite, then $W$ contains a copy of $\mathbb{N}^k$ for every $k \geq 1$, so the asymptotic cones of $W$ are infinite-dimensional. So, from now on, we assume that $X$ is finite. We distinguish two cases. First, assume that $Y$ has linear growth. In other words, $Y$ is a quasi-line of bounded degree, and it must be biLipschitz equivalent to $\mathbb{Z}$, which imposes that $W$ is quasi-isometric to $\mathbb{Z}/n\mathbb{Z} \wr \mathbb{Z}$ with $n:= |X|$. It is well-known that $\mathbb{Z}/n\mathbb{Z} \wr \mathbb{Z}$ quasi-isometrically embed into a product of two (simplicial) trees, so the asymptotic cones of $\mathbb{Z}/n\mathbb{Z} \wr \mathbb{Z}$ topologically embed into a product of two real trees. We conclude that the asymptotic cones of $\mathbb{Z}/n\mathbb{Z} \wr \mathbb{Z}$ have finite topological dimension. (In fact, it has topological dimension precisely one, but this is not something we need to know here.)

\medskip \noindent
Next, assume that $Y$ has super-linear growth. Fix an ultrafilter $\omega$ over $\mathbb{N}$, a sequence of basepoints $o=(o_n)$, and a sequence of scaling factors $\lambda=(\lambda_n)$. Without loss of generality, we assume that $\lambda_n \in \mathbb{N}$ for every $n \geq 0$. Given a $k \geq 1$, our goal is to construct a topological embedding of $[0,1]^k$ into $\mathrm{Cone}_\omega( W,o,\lambda)$. Taking $k$ arbitrarily large will prove that the topological dimension of the asymptotic cone is infinite.

\medskip \noindent
For every $n \geq 0$, write $o_n=(c_n,p_n)$ and let $R_n$ denote the smallest integer such that the ball $B(p_n,R_n)$ in $Y$ has cardinality $\geq k \lambda_n$. Notice that $(R_n/\lambda_n)$ tends to zero as $n \to + \infty$ because the fact that $Y$ has super-linear growth implies that
$$\frac{R_n}{\lambda_n}= \frac{R_n}{ |B(p_n,R_n-1)|} \cdot \frac{| B(p_n,R_n-1)|}{\lambda_n} \leq k \frac{R_n}{ |B(p_n,R_n-1)|} \underset{n \to + \infty}{\longrightarrow} 0.$$
Now, fix an index $n \geq 0$. Because $B(p_n,R_n)$ has size $\geq k\lambda_n$, there exist $k$ pairwise disjoint subsets $L_n(1), \ldots, L_n(k) \subset B(p_n,R_n)$ of size $\lambda_n$. For every $1 \leq i \leq k$, we fix an enumeration $L_n(i)= \{\ell_n^1(i), \ldots, \ell_n^{\lambda_n}(i) \}$. Define
$$\Psi : \left\{ \begin{array}{ccc} [0,\lambda_n]^k & \to & W \\ (r_1,\ldots, r_k) & \mapsto & \left( y \mapsto \left\{ \begin{array}{cl} c_n(y)' & \text{if } y \in \{ \ell_n^j(i), \ 1 \leq i \leq k, 1 \leq j \leq r_i \} \\ c_n(y) & \text{otherwise} \end{array} \right., \ p_n \right) \end{array} \right..$$
where $c_n(y)'$ denotes an arbitrary neighbour of $c_n(y)$. Observe that
$$\sum\limits_{i=1}^k |r_i-s_i| \leq d \left( \Psi(r), \Psi(s) \right) \leq \sum\limits_{k=1}^k |r_i-s_i| + 2kR_n$$
for all $r=(r_i),s=(s_i) \in [0,\lambda_n]^k$. As a consequence, $\Psi$ induces a biLipschitz embedding
$$\Psi_\infty : \underset{\simeq [0,1]^k}{\underbrace{\mathrm{Cone}_\omega \left( [1,\lambda_n]^k, 0, \lambda \right)}} \hookrightarrow \mathrm{Cone}_\omega (W, o,\lambda),$$
concluding the proof of our proposition.
\end{proof}

\section{Growth of lamps}\label{section:ApplicationsTwo}

\noindent
In this section, we introduce and study an invariant for halo products $\mathscr{L}H$, which, roughly speaking, quantifies the growth of the finite subgroups of $L(H)$ in function of their supports. More precisely, let $\mathscr{L}$ be a halo of groups over some metric space $X$ with $L(X)$ locally finite. Given a parameter $r \geq 0$, set
$$L_r(S):= \left\langle \bigcup_{s \in S} L\left( B(s,r) \right) \right\rangle \text{ for every } S \subset X$$
and define the \emph{lamp growth function} as the map
$$\Lambda_\mathscr{L}^r : S \subset X \text{ finite} \mapsto \# L_r(S).$$
In other words, $\Lambda_\mathscr{L}^r(S)$ quantifies the size of the subgroup generated by the elements of small supports around $S$. As shown in Section~\ref{section:GrowthApplications}, for the examples we are the most interested in, $\Lambda_\mathscr{L}^r(S)$ essentially coincides with the size of $L(S)$ when $S$ is coarsely connected.

\medskip \noindent
The first observation is that the lamp growth function $\Lambda_\mathscr{L}^r$ does not depend on the parameter $r$ in an essential way.

\medskip \noindent
We emphasize that, in all this section, we use the notation from Convention~\ref{Conv}.

\begin{lemma}\label{lem:Equivalence}
Let $\mathscr{L}H$ be a finitely generated halo product with $L(H)$ locally finite. Fix an $r_0 \geq 0$ such that $L_{r_0}(X)=L(X)$. Then, for all $s \geq r \geq r_0$, there exists some $K \geq 0$ such that
$$L_r(S) \subset L_s(S) \subset L_r(S^{+K}), \text{ hence } \Lambda_\mathscr{L}^r(S) \leq \Lambda_\mathscr{L}^s(S) \leq \Lambda_\mathscr{L}^r(S^{+K}),$$
for every $S \subset H$ finite.
\end{lemma}

\begin{proof}
The inclusion $L_r(S) \subset L_s(S)$ is clear. Next, because $L_{r_0}(X)=L(X)$, we find a constant $K \geq 0$ such that $L(B(1,s)) \subset L_{r_0}(B(1,K))$. Then we have
$$\begin{array}{lcl} L_s(S) & = & \displaystyle \left\langle \bigcup\limits_{x \in S} L(B(x,s)) \right\rangle \subset \left\langle \bigcup\limits_{x \in S} L_{r_0}(B(x,K)) \right\rangle = \left\langle \bigcup\limits_{x \in S} \bigcup\limits_{y \in B(x,K)} L(B(y,r_0)) \right\rangle \\ \\ & \subset & \displaystyle \left\langle \bigcup\limits_{y \in S^{+K}} L(B(y,r_0)) \right\rangle = L_{r_0}(S^{+K}) \subset L_r(S^{+K}) \end{array}$$
which concludes the proof of our lemma.
\end{proof}

\noindent
Using the domination relation introduced in Section~\ref{section:VolumeGrowth}, Lemma~\ref{lem:Equivalence} shows that $\Lambda_\mathscr{L}^r \prec \Lambda_\mathscr{L}^s$ and $\Lambda_\mathscr{L}^s \prec \Lambda_\mathscr{L}^r$ for all sufficiently large $r,s \geq 0$. 

\medskip \noindent
In Section~\ref{section:VolumeGrowth}, we show that the asymptotic behaviour of lamp growths is monotonic with respect to aptolic coarse embeddings. In particular, this will allow us to distinguish lamplighters, lampdesigners, and lampcloners up to quasi-isometry in Section~\ref{section:GrowthApplications}. However, this invariant does not allow us to distinguish, say, two lampdesigners or two lampcloners. Indeed, their lamp growths have the same asymptotic behaviour. This is why, in Section~\ref{section:GrowthArithm}, we introduce finer invariants by noticing that, if there exists an aptolic quasi-isometry between two halo products, then the lamp growths are related arithmetically. Applications are included in Section~\ref{section:GrowthApplications}.

\subsection{Preservation of volume growth}\label{section:VolumeGrowth}

\noindent
In order to compare two lamp growths, which are defined on subsets of distinct sets, we introduce the following relation:

\begin{definition}
Let $X,Y$ be two sets and $\mathfrak{X} \subset \mathfrak{P}(X)$, $\mathfrak{Y} \subset \mathfrak{P}(Y)$ two collections of subsets. A map $\zeta : \mathfrak{X} \to \mathbb{R}$ is \emph{dominated} by a map $\xi : \mathfrak{Y} \to \mathbb{R}$, which we write $\zeta \prec \xi$, if there exists $Q>0$ such that the following holds. For every finite $S \in \mathfrak{X}$, there exists some $R \in \mathfrak{Y}$ such that $\zeta(S) \leq \xi(R)$ and $|R|/Q \leq |S| \leq Q |R|$.
\end{definition}

\noindent
Roughly speaking, $\zeta$ is dominated by $\xi$ if each $\zeta(S)$ is bounded above by some $\xi(R)$ where the size of $R$ is comparable to the size of $S$. The rest of the section is dedicated to the proof of the following statement:

\begin{prop}\label{prop:LampGrowth}
Let $\mathscr{M}A$ and $\mathscr{N}B$ be two finitely generated halo products with $M(A)$ and $N(B)$ locally finite. If there exists an aptolic coarse embedding $\mathscr{M}A \to \mathscr{N}B$, then the following holds for every sufficiently large $r$. For every $C \geq 0$, $\Lambda_\mathscr{M}^r \prec \Lambda^r_\mathscr{N}$ on the $C$-coarsely connected subspaces of $X$ and $Y$. 
\end{prop}

\noindent
The proposition will be a straightforward consequence of the following fundamental observation:

\begin{lemma}\label{lem:InclusionAlpha}
Let $\mathscr{M}A$ and $\mathscr{N}B$ be two finitely generated halo products. Assume that there exists a coarsely Lipschitz map $\Phi : \mathscr{M}A \to \mathscr{N}B$ of the form $\Phi : (c,p) \mapsto (\alpha(c),\beta(p))$ where $\alpha : M(A) \to N(B)$ and $\beta : A \to B$. The following holds for every sufficiently large $r \geq 0$. For every $C \geq 0$, there exists some $K \geq 0$ such that
$$\alpha(cM_r(S)) \subset \alpha(c) N_r \left( \beta(S)^{+K} \right)$$
for every $c \in M(A)$ and every $C$-coarsely connected $S \subset A$.
\end{lemma}

\noindent
Recall that a map $\varphi : X \to Y$ between two metric spaces $X,Y$ is \emph{$(K_1,K_2)$-coarsely Lipschitz} if
$$d(\varphi(x), \varphi(y)) \leq K_1 d(x,y)+ K_2$$
for all $x,y \in X$. A map is \emph{coarsely Lipschitz} if it is $(K_1,K_2)$-coarsely Lipschitz for some $K_1,K_2>0$. 

\begin{proof}[Proof of Lemma~\ref{lem:InclusionAlpha}.]
Let $K_1,K_2$ be such that $\Phi$ is $(K_1,K_2)$-coarsely Lipschitz. Notice that $\beta$ must also be $(K_1,K_2)$-coarsely Lipschitz. Fix some $r \geq 0$ such that the inclusion $N^+(B(1,K_1+K_2)) \subset N_r(1)$ holds, i.e.\ $r \geq K_1+K_2+r_0$. Finally, let $K_3$ be such that $S^{+K_3}$ is connected (i.e.\ $K_3 \geq C$) and such that 
$$M(B(1,r)) \subset \left\langle \bigcup\limits_{x \in B(1,K_3)} M^+(x) \right\rangle.$$
Given an element $e \in M_r(S)$, there is a path in $\mathscr{M}A$ connecting the leaves $cA$ and $ceA$ of the form
$$(c_0,s_0), \ (c_1,s_1), \ldots, \ (c_{n-1},s_{n-1}),\ (c_n,s_n)$$
where $c_0=c$, $c_n=ce$, and $s_i \in S^{+K_3}$ for every $0 \leq i \leq n$. The image
$$(\alpha(c_0),\beta(s_0)), \ (\alpha(c_1),\beta(s_1)), \ldots, \ (\alpha(c_{n-1}),\beta(s_{n-1})), \ (\alpha(c_n),\beta(s_n))$$
under $\Phi$ of this path connects the leaves $\alpha(c)B$ and $\alpha(ce)B$. Because $\beta$ is $(K_1,K_2)$-coarsely Lipschitz, each $\beta(s_i)$ must belong to $\beta(S^{+K_3}) \subset \beta(S)^{+K_1K_3+K_2}$. For every $1 \leq i \leq n$, we know that
$$d((\alpha(c_i),\beta(s_i)),(\alpha(c_{i-1}), \beta(s_{i-1}))) \leq K_1 d((c_i,s_i),(c_{i-1},s_{i-1})) + K_2 = K_1+K_2,$$
so we must have
$$\alpha(c_i) \in \alpha(c_{i-1}) N^+(B(\beta(s_{i-1}),K_1+K_2) \subset \alpha(c_{i-1})N_r( \beta(s_{i-1})) \subset \alpha(c_{i-1}) N_r \left( \beta(S)^{+K} \right)$$
where $K:=K_1K_3+K_2$. We conclude that
$$\alpha(ce)=\alpha(c_n) \in \alpha(c_0) N_r(\beta(S)^{+K}) = \alpha(c) N_r \left( \beta(S)^{+K} \right)$$
as desired.
\end{proof}

\begin{proof}[Proof of Proposition~\ref{prop:LampGrowth}.]
Assume that there exists an aptolic coarse embedding 
$$\left\{ \begin{array}{ccc} \mathscr{M}A & \to & \mathscr{N}B \\ (c,p) & \mapsto & (\alpha(c),\beta(p)) \end{array} \right.$$
where $\alpha : M(A) \to N(B)$ is injective and $\beta : A \to B$. Fix an $r \geq 0$ sufficiently large so that Lemma~\ref{lem:InclusionAlpha} applies, and fix some $C \geq 0$. There exists some $K \geq 0$ such that, for every $S \subset X$ that is $C$-connected, we have
$$\Lambda_\mathscr{M}^r(S)= |M_r(S)| = |\alpha(M_r(S))| \leq |N_r(\beta(S)^{+K})| = \Lambda_\mathscr{N}^r \left( \beta(S)^{+K} \right).$$
Up to increasing $K$ (compared to $C$ and the parameters of $\beta$), we assume that $\beta(S)^{+K}$ is $C$-coarsely connected. In order to conclude our proof, it remains to compare the size of $S$ with the size of $\beta(S)^{+K}$. Notice that
$$|\beta(S)^{+K}| \leq Q_1 \cdot |\beta(S)| \leq Q_1 \cdot |S|$$
where $Q_1$ denotes the cardinality of a ball of radius $K$ in $B$. We also have
$$|\beta(S)^{+K}| \geq |\beta(S)| \geq |S|/Q_2$$
where $Q_2$ denotes the maximal size of the pre-image of a single point under $\beta$ (which is finite because $\beta$ is a coarse embedding). 
\end{proof}

\begin{cor}
Let $\mathscr{M}A$ and $\mathscr{N}B$ be two finitely generated halo products with $M(A)$ and $N(B)$ locally finite. If there exists an aptolic quasi-isometry $\mathscr{M}A \to \mathscr{N}B$, then the following holds for every sufficiently large $r$. For every $C \geq 0$, $\Lambda_\mathscr{M}^r \prec \Lambda^r_\mathscr{N}$ and $\Lambda_\mathscr{N}^r \prec \Lambda_\mathscr{M}^r$ on the $C$-coarsely connected subspaces of $X$ and $Y$. 
\end{cor}

\begin{proof}
It suffices to apply Proposition~\ref{prop:LampGrowth} to an aptolic quasi-isometry $\mathscr{M}A \to \mathscr{N}B$ and to an aptolic quasi-inverse.
\end{proof}

\subsection{Preservation of arithmetic properties}\label{section:GrowthArithm}

\noindent
In this section, our goal is to show that, if there exists an aptolic quasi-isometry between two halo products, then we do not only know that their lamp growth functions have similar asymptotic growths, we also know that they have similar arithmetic properties. They are \emph{arithmetically interlaced} in the following sense:

\begin{definition}\label{def:Interlaced}
Let $X,Y$ be two sets, $\mathfrak{X}\subset \mathfrak{P}(X)$, $\mathfrak{Y} \subset \mathfrak{P}(Y)$ two collections of subsets, and $\Delta : \mathfrak{X} \to \mathbb{N}$ a map. A map $\zeta : \mathfrak{X} \to \mathbb{N}$ is \emph{$\Delta$-interlaced} with a map $\xi : \mathfrak{Y} \to \mathbb{N}$ if there exists $Q>0$ such that the following holds. For every finite subset $S \in \mathfrak{X}$, there exist $R \in \mathfrak{Y}$ and $T \in \mathfrak{X}$ such that
\begin{itemize}
	\item $\zeta(S)$ divides $\xi(R)$, and $\xi(R)$ divides $\zeta(T)$;
	\item $|S|/Q \leq |R| \leq Q \cdot |S|$ and $|S| \leq |T| \leq |S| +Q \cdot \Delta(S)$. 
\end{itemize}
\end{definition}

\noindent
Roughly speaking, $\zeta$ is $\Delta$-interlaced with $\xi$ if each $\zeta(S)$ divides some $\xi(R)$, which in turn divides some $\zeta(T)$, where the size of $R$ is comparable to the size of $S$ and where the size of $T$ differs from the size of $S$ only by an additive term depending on $\Delta(S)$. The point is that, if $\Delta(S)$ grows slowly with the size of $S$, then $S$ and $T$ have pretty much the same size.

\medskip \noindent
The rest of the section is dedicated to the proof of the following statement:

\begin{prop}\label{prop:ArithmEquivalent}
Let $\mathscr{M}A,\mathscr{N}B$ be two finitely generated halo products with $M(A)$ and $N(B)$ locally finite. If there exists an aptolic quasi-isometry $\mathscr{M}A \to \mathscr{N}B$, then the following holds for every sufficiently large $r$. For every $C \geq 0$, the lamp growth $\Lambda_\mathscr{M}^r$ is $|\partial (\cdot )|$-interlaced with $\Lambda_\mathscr{N}^r$ on the $C$-coarsely connected subspaces of $X$ and $Y$.
\end{prop}

\noindent
The divisibility between values of the growth lamps is recorded by the following lemma:

\begin{lemma}\label{lem:CosetDivision}
Let $\mathscr{M}A,\mathscr{N}B$ be two finitely generated halo products with $M(A)$ and $N(B)$ locally finite. If there exists an aptolic quasi-isometry
$$\left\{ \begin{array}{ccc} \mathscr{M}A & \to & \mathscr{N}B \\ (c,p) & \mapsto & (\alpha(c),\beta(p)) \end{array} \right.,$$
where $\alpha : M(A) \to N(B)$ is bijective and $\beta : A \to B$, then the following holds for every sufficiently large $r$. For every $C \geq 0$, there exists $K \geq 0$ such that $|M_r(S)|$ divides $|N_r(\beta(S)^{+K})|$ for every finite $S \subset A$ that is $C$-coarsely connected.
\end{lemma}

\noindent
We emphasize that, because $N_r(\beta(S)^{+K'})$ contains $N_r(\beta(S)^{+K})$ as a subgroup whenever $K' \geq K$, which implies that $|N_r(\beta(S)^{+K})$ divides $|N_r(\beta(S)^{+K'})|$, the conclusion of Lemma~\ref{lem:CosetDivision} still holds if we increase $K$. 

\begin{proof}[Proof of Lemma~\ref{lem:CosetDivision}.]
Fix some $r\geq 0$ sufficiently large so that Lemma~\ref{lem:InclusionAlpha} applies. Fix some $C \geq 0$ and let $K$ be the constant given by Lemma~\ref{lem:InclusionAlpha}. Given a finite subset $S \subset A$ that is $C$-coarsely connected, set $E:= \alpha^{-1}(\alpha(1) N_r(\beta(S)^{+K}))$. As a consequence of Lemma~\ref{lem:InclusionAlpha}, for every $e \in E$ we have
$$\begin{array}{lcl} \alpha(eM_r(S)) & \subset & \displaystyle \alpha(e) N_r \left( \beta(S)^{+K} \right) \subset \alpha(1) N_r \left( \beta(S)^{+K} \right) N_r \left( \beta(S)^{+K} \right) \\ \\ & \subset & \displaystyle \alpha(1) N_r \left( \beta(S)^{+K} \right), \end{array}$$
hence $eM_r(S) \subset E$. In other words, $E$ is stable by right-multiplication by elements in $M_r(S)$, which amounts to saying that $E$ is a union of $M_r(S)$-cosets, and which implies that $|M_r(S)|$ divides $|E|= |N_r(\beta(S)^{+K})|$. 
\end{proof}

\begin{proof}[Proof of Proposition~\ref{prop:ArithmEquivalent}.]
Assume that there exist a bijection $\alpha : M(A) \to N(B)$ and a quasi-isometry $\beta : A \to B$ such that $\Phi : (c,p) \mapsto (\alpha(c), \beta(p))$ defines an aptolic quasi-isometry $\mathscr{M}A \to \mathscr{N}B$. Fixing a quasi-inverse $\bar{\beta} : B \to A$ of $\beta$, the map $\bar{\Phi} : \mathscr{N}B \to \mathscr{M}A$ defined by $(c,p) \mapsto (\alpha^{-1}(c),\bar{\beta}(p))$ yields an aptolic quasi-inverse of $\Phi$. Fix an $r$ sufficiently large so that Lemma~\ref{lem:CosetDivision} applies to $\Phi$ and $\bar{\Phi}$. Fix a $C \geq 0$. Let $K \geq 0$ be a constant sufficiently large so that Lemma~\ref{lem:CosetDivision} applies to $\Phi$ and $\bar{\Phi}$, such that the $K$-neighbourhood of the image under $\beta$ (resp.\ $\bar{\beta}$) of a $C$-coarsely connected subspace is still $C$-coarsely connected, and such that every subset $S \subset X$ is contained in $\bar{\beta}(\beta(S)^{+K})^{+K}$. 

\medskip \noindent
Given a finite $S \subset X$ that is $C$-coarsely connected, we know from Lemma~\ref{lem:CosetDivision} that $\Lambda_\mathscr{M}^r(S)$ divides $\Lambda_\mathscr{N}^r(\beta(S)^{+K})$, which in turn divides $\Lambda_\mathscr{M}^r( \bar{\beta}(\beta(S)^{+K})^{+K})$. In order to conclude our proof, it suffices to observe that 
$$ |S| / N \leq \left| \beta(S)^{+K} \right| \leq  \mathrm{deg}(B)^K |S|$$
where $N$ denotes the maximal number of pre-images of a point under $\beta$; and that
$$S \subset \bar{\beta}( \beta(S)^{+K})^{+K} \subset S^{+Q}$$
for some constant $Q \geq 0$ depending only on $K$ and the parameters of $\beta$ and $\bar{\beta}$, hence 
$$|S| \leq |\bar{\beta}( \beta(S)^{+K})^{+K}| \leq |S| + E \cdot | \partial S|$$
where $E$ denotes the constant given by \cite[Fact~2.8(iv)]{MR4419103}.
\end{proof}

\subsection{Applications to concrete examples}\label{section:GrowthApplications}

\noindent
In this section, we exploit the quasi-isometric invariants built in Sections~\ref{section:VolumeGrowth} and~\ref{section:GrowthArithm} to concrete examples of halo products. The following definition records the convenient properties that are satisfied by the examples we are the most interested in, namely lamplighters, lampjugglers, lampdesigners, lampcloners, and verbal wreath products.
Recall our notation $L_D(S):= \left\langle \bigcup_{s \in S} L\left( B(s,D) \right) \right\rangle$
\begin{definition}
Let $X$ be a connected graph. A halo of groups $\mathscr{L}$ over $X$ is \emph{charming} if the following conditions are satisfied:
\begin{itemize}
	\item there exists some $D \geq 0$ such that $L(S) \subset L_D(S)$ for every $S$ connected;
	\item for $S$ connected, the size of $L(S)$ only depends on the size of $S$.
\end{itemize}
\end{definition}

\noindent
Focusing on these specific halos of groups allows us to replace lamp growth functions with sequences. This is not necessary, but this will simplify the exposition and the computations.

\medskip \noindent
Let $\mathscr{L}$ be a charming halo of groups over some metric space $X$ with $L(X)$ locally finite. The \emph{lamp growth sequence} is
$$\Lambda_\mathscr{S} : n \mapsto |L(S)|, \ S \subset X \text{ of size } n.$$
The sequence is well-defined because $\mathscr{L}$ is charming.

\medskip \noindent
The next proposition yields the analogue of Proposition~\ref{prop:LampGrowth} for lamp growth sequences.

\begin{prop}\label{prop:LampGrowthBis}
Let $\mathscr{M}A,\mathscr{N}B$ be two finitely generated halo products with $\mathscr{M},\mathscr{N}$ charming and with $M(A),N(B)$ locally finite. If there exists an aptolic coarse embedding $\mathscr{M}A \to \mathscr{N}B$, then $\Lambda_\mathscr{M} \prec \Lambda_\mathscr{N}$. 
\end{prop}

\noindent
Here, we say that a sequence $\psi : \mathbb{N} \to \mathbb{N}$ \emph{dominates} another sequence $\varphi : \mathbb{N} \to \mathbb{N}$, written $\varphi \prec \psi$, if there exists some $C \geq 0$ such that $\varphi(n) \leq \psi(Cn)$ for every $n \geq 1$. The sequences are \emph{equivalent}, written $\varphi \sim \psi$, if $\varphi \prec \psi$ and $\psi \prec \varphi$. 

\begin{proof}[Proof of Proposition~\ref{prop:LampGrowthBis}.]
Because $\mathscr{M}$ (resp.\ $\mathscr{N}$) is charming, we know that there exists some $D \geq 0$ (resp.\ $E \geq 0$) such that $M(S) \subset M_D(S)$ (resp.\ $N(S) \subset N_E(S)$) for every connected $S \subset A$ (resp.\ $S \subset B$). Fix an $r \geq D,E$ sufficiently large such that Proposition~\ref{prop:LampGrowth} applies. Given an $n \geq 1$, fix a connected subgraph $S_n \subset A$ of size $n$. We have
$$\Lambda_\mathscr{M}(n)= |M(S_n)| \leq |M_D(S_n)| \leq |M_r(S_n)|= \Lambda_\mathscr{M}^r(S_n).$$
We know from Proposition~\ref{prop:LampGrowth} that there exists a connecting subgraph $R_n \subset B$ satisfying $\Lambda_\mathscr{M}^r(S_n) \leq \Lambda_\mathscr{N}^r(R_n)$ and $|S_n|/Q \leq |R_n| \leq Q | S_n|$ for some constant $Q>0$ independent of $S_n$. So
$$\Lambda_\mathscr{M}(n) \leq \Lambda_\mathscr{N}^r(R_n) \leq |N(R_n^{r+r_0})| = \Lambda_\mathscr{N}( |R_n^{+r+r_0}| ).$$
If $N$ denotes the degree of $B$, then 
$$|R_n^{+r+r_0}| \leq N^{+r+r_0} |R_n| \leq QN^{r+r_0} |S_n| = QN^{r+r_0} n,$$
hence $\Lambda_\mathscr{M}(n) \leq \Lambda_\mathscr{N} (Cn)$ for $C:= QN^{r+r_0}$. 
\end{proof}

\begin{cor}\label{cor:LampGrowthBis}
Let $\mathscr{M}A,\mathscr{N}B$ be two finitely generated halo products with $\mathscr{M},\mathscr{N}$ charming and with $M(A),N(B)$ locally finite. If there exists an aptolic quasi-isometry $\mathscr{M}A \to \mathscr{N}B$, then $\Lambda_\mathscr{M} \sim \Lambda_\mathscr{N}$. 
\end{cor}

\begin{proof}
It suffices to apply Proposition~\ref{prop:LampGrowthBis} to an aptolic quasi-isometry $\mathscr{M}A \to \mathscr{N}B$ and to an aptolic quasi-inverse.
\end{proof}

\noindent
In order to apply Proposition~\ref{prop:LampGrowthBis}, let us compute the lamp growth sequences of our main examples of halo products. Recall that, given two sequences $f$ and $g$, we note $f = \Theta(g)$ if 
$$C_1 \cdot f(n) \leq g(n) \leq C_2 \cdot f(n) \text{ for every } n \geq 0$$
for some constants $C_1,C_2>0$.

\begin{fact}\label{fact:GrowthWreath}
Let $F$ be a finite group and $H$ a finitely generated group. The lamp growth sequence of the wreath product $F \wr H$ is
$$\Lambda : n \mapsto |F|^n.$$
Consequently, $\log \Lambda(n) = \Theta (n)$. 
\end{fact}

\begin{proof}
For every $n \geq 1$, we have $\Lambda(n)= |\bigoplus_S F| = |F|^{|S|}=|F|^n$ where $S \subset B$ is an arbitrary subset of size $n$. 
\end{proof}

\begin{fact}\label{fact:2Nil}
Let $F$ be a finite group and $H$ a finitely generated group. The lamp growth sequence of the $2$-nilpotent wreath product $F \wr^{\mathfrak{n}_2} H$ is
$$\Lambda : n \mapsto |F|^n \cdot |F_{\mathrm{ab}}|^{n(n-1)/2}.$$
Consequently, $\log \Lambda(n) = \Theta(n^2)$ if $F$ is not perfect.
\end{fact}

\begin{proof}
For every $n \geq 1$, $\Lambda(n)$ is the size of the $2$-nilpotent product of $n$ copies of $F$, which we denote here as $F^{\ast n}$. As a consequence of \cite{MR0120270} (see also \cite{MR4053850}), there is a short exact sequence
$$1 \to (F^{\ast n-1})_{\mathrm{ab}} \otimes F_\mathrm{ab} \to F^{\ast n} \to F^{\ast n-1} \times F \to 1.$$
Observe that
$$| (F^{\ast n-1})_\mathrm{ab} \otimes F_\mathrm{ab} | = | F_\mathrm{ab}^{ n-1} \otimes F_\mathrm{ab} | = |F_\mathrm{ab} \otimes F_\mathrm{ab}|^{n-1} = |F_\mathrm{ab}|^{n-1}.$$
Thus, we obtain that
$$|F^{\ast n}| = | F^{\ast n-1}| \cdot |F| \cdot |F_\mathrm{ab}|^{n-1},$$
from which the equality $|F^{\ast n}| = |F|^n \cdot |F_{\mathrm{ab}}|^{n(n-1)/2}$ follows. 
\end{proof}

\begin{fact}
Let $H$ be a finitely generated group and $p \geq 1$ an integer. The lamp growth sequence of the lampjuggler $\circledS_pH$ is
$$\Lambda : n \mapsto (np)!.$$
Consequently, $\log \Lambda(n) = \Theta( n \log n)$.
\end{fact}

\begin{proof}
For every $n \geq 1$, we have $\Lambda(n)= |\mathrm{Sym}(S \times \{1,\ldots, p\})| = (p|S|)!=(np)!$ where $S \subset H$ is an arbitrary subset of size $n$. The estimate $\Lambda(n)= \Theta(n \log n)$ follows from Stirling's formula.
\end{proof}

\begin{fact}
Let $F$ be a finite group and $H$ a finitely generated group. The lamp growth sequence of the lampdesigner $F \boxplus H$ is
$$\Lambda : n \mapsto |F|^n n!.$$
Consequently, $\log \Lambda(n) = \Theta( n \log n)$.
\end{fact}

\begin{proof}
For every $n \geq 1$, we have $\Lambda(n) = | \bigoplus_S F \rtimes \mathrm{Sym}(S)| = |F|^{|S|} |S|! = |F|^nn!$ where $S \subset H$ is an arbitrary subset of size $n$. The estimate $\log \Lambda(n) = \Theta( n \log n)$ follows from Stirling's formula.
\end{proof}

\begin{fact}\label{fact:GrowthCloner}
Let $H$ be a finitely generated group and $\mathfrak{k}$ a finite field. The lamp growth sequence of the lampcloner $\oslash_\mathfrak{k} H$ is
$$\Lambda : n \mapsto  \prod\limits_{i=0}^{n-1} (|\mathfrak{k}|^n-|\mathfrak{k}|^i).$$
Consequently, $\log\Lambda(n) = \Theta(n^2)$.
\end{fact}

\begin{proof}
For every $n \geq 1$, $\Lambda(n)= | \mathrm{GL}(n, \mathfrak{k})|$, which classically equals the product given above. In the sequel, set $q:= |\mathfrak{k}|$ for convenience. Notice that, for every $n \geq 1$, 
$$\begin{array}{lcl} \displaystyle \log \left( \prod\limits_{i=0}^{n-1} (q^n-q^i) \right) & = & \displaystyle \sum\limits_{i=0}^{n-1} \log(q^n-q^i) = \sum\limits_{i=0}^{n-1} \log(q^n) + \sum\limits_{i=0}^{n-1} \log\left( 1 - \frac{1}{q^{n-i}} \right) \\ \\ & = & \displaystyle \frac{n(n-1)}{2} \log(q) + \sum\limits_{i=1}^n \log\left( 1 - \frac{1}{q^{i}} \right). \end{array}$$
But the second term in the latter sum is at most linear since
$$\left| \sum\limits_{i=1}^n \log\left( 1 - \frac{1}{q^{i}} \right) \right| \leq n \left| \log \left( 1- \frac{1}{q} \right) \right|,$$
hence
$$\log \left( \prod\limits_{i=0}^{n-1} (q^n-q^i) \right) =  n^2 \log(q) + O(n).$$
We conclude that $\log\Lambda(n) = \Theta(n^2)$, as desired.
\end{proof}

\noindent
We are now ready to show that there do not exist coarse (resp.\ quasi-isometric) embeddings or quasi-isometries between some of our favourite halo products.

\begin{cor}\label{cor:NoCoarseEmb}
Let $G,H,I,J$ be four finitely generated groups, $E,F$ two finite groups, $\mathfrak{k}$ a finite field, and $p \geq 1$ an integer. The following statements hold.
\begin{itemize}
	\item Assume that $H,I,J$ satisfy the thick bigon property and that there is no coarse (resp.\ quasi-isometric) embedding $\mathcal{L}(X) \to G$ with $X$ quasi-isometric to one of $H,I,J$. Then there is no coarse (resp.\ quasi-isometric) embedding of $\circledS_p H$, $F \boxplus I$, $\oslash_\mathfrak{k}J$ in $E \wr F$. 
	\item Assume that $J$ satisfies the thick bigon property and that there is no coarse (resp.\ quasi-isometric) embedding $\mathcal{L}(X) \to H,I$ with $X$ quasi-isometric to $J$. Then there is no coarse (resp.\ quasi-isometric) embedding of $\oslash_\mathfrak{k}J$ in $\circledS_p H$, $F \boxplus I$.
\end{itemize}
\end{cor}

\noindent
Recall from Definition~\ref{def:WPGraphs} that $\mathcal{L}(\cdot)$ denotes a lamplighter graph.

\begin{proof}[Proof of Corollary~\ref{cor:NoCoarseEmb}.]
As a consequence of Theorem~\ref{thm:InALeaf} (which applies according to Proposition~\ref{prop:AreFull} and Lemma~\ref{lem:NilpotentFull}), our hypothetical coarse (resp.\ quasi-isometric) embeddings must send leaves inside leaves (up to finite distance). Next, it follows from Theorem~\ref{thm:AptoCoarseEmbedding} that they must be aptolic. We deduce from Proposition~\ref{prop:LampGrowthBis} a domination between the lamp growth sequences of the halo products under consideration, and we conclude from the estimates given by Facts~\ref{fact:GrowthWreath}--\ref{fact:GrowthCloner} that our coarse (resp.\ quasi-isometric) embedding do not exist.
\end{proof}

\begin{cor}\label{cor:ManyNotQI}
Let $G,H,I,J$ be four finitely generated groups satisfying the thick bigon property, $E,F$ three finite groups, $\mathfrak{k}$ a finite field, and $p \geq 1$ an integer. The following statements hold.
\begin{itemize}
	\item The wreath product $E \wr G$ is not quasi-isometric to $\circledS_p H$, $C \boxplus I$, $\oslash_\mathfrak{k}J$. 
	\item The lampjuggler $\circledS_p H$ and the lampdesigner $F \boxplus I$ are not quasi-isometric to $\oslash_\mathfrak{k}J$.
\end{itemize}
\end{cor}

\begin{proof}
As a consequence of Theorem~\ref{thm:InALeaf} (which applies according to Proposition~\ref{prop:AreFull}), our hypothetical quasi-isometry must be leaf-preserving. Next, it follows from Corollary~\ref{cor:AptoAltitude} (which applies according to Lemmas~\ref{lem:LighterAltitude}--\ref{lem:AltitudeCloner}) that they must be aptolic. We deduce from Corollary~\ref{cor:LampGrowthBis} an equivalence between the lamp growth sequences of the halo products under consideration, and we conclude from the estimates given by Facts~\ref{fact:GrowthWreath}--\ref{fact:GrowthCloner} that our quasi-isometry actually does not exist.
\end{proof}

\noindent
Recall from Proposition~\ref{prop:NilWP} that we already know that $2$-nilpotent wreath products are typically not quasi-isometric to lamplighters, lampjugglers, lampdesigner and lampcloners. Regarding Fact~\ref{fact:2Nil}, it would be tempting to claim that most of the time there do not even exist a coarse embedding from a $2$-nilpotent wreath product to a lamplighter, a lampjuggler, or a lampdesigner. However, Theorem~\ref{thm:AptoCoarseEmbedding} does not apply because the halo of groups given by a $2$-nilpotent wreath product usually has not a finite altitude (unless it coincides with a standard wreath product). Proving this claim would require new techniques to prove that leaf-preserving coarse embeddings between halo products are aptolic.

\medskip \noindent
Now, following Proposition~\ref{prop:ArithmEquivalent}, we want to show that the lamp growth sequences of two quasi-isometric halo products are arithmetically interlaced. We begin by adapting Definition~\ref{def:Interlaced} to sequences.

\begin{definition}
Given a function $\Delta : \mathbb{N} \to \mathbb{N}$ and two sequences $\varphi, \psi : \mathbb{N} \to \mathbb{N}$, $\varphi$ is \emph{$\Delta$-interlaced} with $\psi$ if there exist a constant $C \geq 0$ and a sequence $(x_n)$ satisfying $x_n= \Theta(n)$ such that
$$\varphi(n) \text{ divides } \psi(x_n), \text{ which divides } \varphi(n+ C \cdot \Delta(n))$$
for every $n \geq 1$.
\end{definition}

\noindent
The functions that will be relevant for saying that two lamp growth sequences are interlaced are the \emph{boundary growths}, which we define now:

\begin{definition}
Let $X$ be locally finite graph. Its \emph{boundary growth} is the map
$$\mathbb{B} : n \mapsto \min \{ |\partial S| \mid S \subset X \text{ connected of size }  n\}.$$
\end{definition}

\noindent
As justified by the following lemma, the boundary growths are always at most linear, and more or less sublinear for amenable graphs.

\begin{lemma}\label{lem:BoundaryGrowth}
Let $X$ be a graph of bounded degree. The following assertions hold:
\begin{itemize}
	\item for every $n \geq 0$, $\mathbb{B}(n) \leq \mathrm{deg}(X) \cdot n$;
	\item $X$ is amenable if and only if $\mathbb{B}(n)/n$ subconverges to $0$.
\end{itemize}
\end{lemma}

\begin{proof}
The first assertion is clear. The second assertion follows from the characterisation of amenability in terms of F\o lner sequences (which can always be taken connected).
\end{proof} 

\noindent
We are now ready to state and prove the analogue of Proposition~\ref{prop:ArithmEquivalent} for lamp growth sequences.

\begin{prop}\label{prop:LampGrowthArithm}
Let $\mathscr{M}A,\mathscr{N}B$ be two finitely generated halo products with $\mathscr{M},\mathscr{N}$ charming and with $M(A),N(B)$ locally finite. If there exists an aptolic quasi-isometry $\mathscr{M}A \to \mathscr{N}B$, then $\Lambda_\mathscr{M}$ is $\mathbb{B}$-interlaced with $\Lambda_\mathscr{N}$ where $\mathbb{B}$ is the boundary growth of $A$.
\end{prop}

\begin{proof}
Because $\mathscr{M}$ (resp.\ $\mathscr{N}$) is charming, we know that there exists some $D \geq 0$ (resp.\ $E \geq 0$) such that $M(S) \subset M_D(S)$ (resp.\ $N(S) \subset N_E(S)$) for every connected $S \subset A$ (resp.\ $S \subset B$). Fix an $r \geq D,E$ sufficiently large such that Proposition~\ref{prop:ArithmEquivalent} applies. Given an $n \geq 1$, fix a connected subgraph $S_n \subset A$ of size $n$ such that $\mathbb{B}(n)= |\partial S_n|$. We have:
$$\Lambda_\mathscr{M}(n) = |M(S_n)| \div |M_D(S_n)| \div |M_r(S_n)|= \Lambda_\mathscr{M}^r(S_n).$$
Let $\beta$ be the quasi-isometry $A \to B$ induced by an aptolic quasi-isometry $\mathscr{M}A \to \mathscr{N}B$. Fix a quasi-inverse $\bar{\beta} : B \to A$. According to Lemma~\ref{lem:CosetDivision}, there exists some $K \geq 0$ independent of $S_n$ such that $\Lambda_\mathscr{M}^r(S)$ divides $\Lambda_\mathscr{N}^r(\beta(S_n)^{+K})$. We know that the latter quantity divides $\Lambda_\mathscr{N}( \beta(S_n)^{+K+r+r_0})$. Up to increasing $K$ if necessary, we can assume that $R_n:= \beta(S_n)^{+K+r+r_0}$ is connected. 

\medskip \noindent
The same argument shows that $\Lambda_\mathscr{N}(R_n)$ divides $\Lambda_\mathscr{M}(T_n)$ where $T_n:=\bar{\beta}(R_n)^{+K+r+r_0}$ (which we can assume connected). 

\medskip \noindent
So far, we have proved that $\Lambda_\mathscr{M}(n) = \Lambda_\mathscr{M}(S_n)$ divides $\Lambda_\mathscr{N}(R_n)$, which in turn divides $\Lambda_\mathscr{M}(T_n)$. In order to conclude the proof of our proposition, it remains to estimate the sizes of $R_n$ and $T_n$. First, observe that
$$n/N = |S_n| / N \leq \left| \beta(S_n)^{+K+r+r_0} \right| \leq N \mathrm{deg}(B)^{+K+r+r_0} |S_n| = N \mathrm{deg}(B)^{+K+r+r_0} n$$
where $N$ denotes the maximal number of pre-images of a point under $\beta$; and that, choosing $K$ sufficiently large compared to the parameters of $\beta$ and $\bar{\beta}$, we have
$$S_n \subset T_n= \bar{\beta}( \beta(S_n)^{+K+r+r_0})^{+K+r+r_0} \subset S^{+Q}$$
for some constant $Q \geq 0$ depending only on $K$ and the parameters of $\beta$ and $\bar{\beta}$, hence 
$$n=|S_n| \leq |T_n| \leq |S_n| + E \cdot | \partial S_n| = n+E \cdot \mathbb{B}(n)$$
where $E$ denotes the constant given by \cite[Fact~2.8(iv)]{MR4419103}.
\end{proof}

\noindent
We conclude this section with concrete applications of the criterion given by Proposition~\ref{prop:LampGrowthArithm}. We begin with the simplest case, namely lamplighter groups. The following theorem generalises a similar statement from \cite{LampGT} proved for lamplighters over one-ended finitely presented groups.

\begin{thm}\label{thm:LighterPrime}
Let $E,F$ be two finite groups and $H,K$ two finitely generated groups satisfying the thick bigon property. If $E \wr H$ and $F \wr K$ are quasi-isometric, then $|E|$ and $|F|$ must have the same prime divisors.
\end{thm}

\noindent
The converse of the theorem turns out to hold when $H$ and $K$ are non-amenable; see \cite{LampGT} for more details. The amenable case will be considered in Section~\ref{section:AmenableLighter}.

\begin{proof}[Proof of Theorem~\ref{thm:LighterPrime}.]
For convenience, set $e:= |E|$ and $f:=|F|$. As a consequence of Theorem~\ref{thm:InALeaf}, a quasi-isometry $E \wr H \to F \wr K$ can be chosen leaf-preserving; and, as a consequence of Corollary~\ref{cor:AptoAltitude} and Lemma~\ref{lem:LighterAltitude}, the quasi-isometry can be chosen even aptolic. Thus, Proposition~\ref{prop:LampGrowthArithm} implies that there exists a sequence $(x_n)$ such that
$$\Lambda_e(n) = e^n \text{ divides } \Lambda_f(x_n)=f^{x_n} \text{ for every } n \geq 0.$$
Thus, every prime divisor of $e$ must divide $f$ as well. By symmetry, we also know that every prime divisor of $f$ must divide $e$. 
\end{proof}

\noindent
It is worth noticing that the proof of Theorem~\ref{thm:LighterPrime} does not use the full strength of Proposition~\ref{prop:LampGrowthArithm}, since, when the cardinalities of the finite groups do not have the same prime divisors, there is no divisibility at all between the numbers given by the lamp growths. The case of lampcloner groups is more interesting with this respect. 

\begin{thm}\label{thm:Lampcloner}
Let $\mathfrak{h},\mathfrak{k}$ be two finite fields and $A,B$ two finitely generated groups satisfying the thick bigon property. If the lampcloner groups $\oslash_\mathfrak{h} A$ and $\oslash_\mathfrak{k} B$ are quasi-isometric, then $\mathrm{char}(\mathfrak{h})= \mathrm{char}(\mathfrak{k})$. 
\end{thm}

\noindent
We begin by proving the following estimate, which will be crucial in our proof of the theorem.

\begin{lemma}\label{lem:CloneEstimate}
Let $p \geq 1$ be a prime number and $q \geq 1$ an integer not divisible by $p$. For every $m \geq 1$, the $p$-valuation of $N(m):= (q-1)(q^2-1) \cdots (q^m-1)$ is a $O(m\ln(m))$. 
\end{lemma}

\begin{proof}
For every $k \geq 1$, let $J_k$ be the set of powers $j \leq m$ such that $p^k$ divides $q^j-1$. Because the cardinality of $J_k \backslash J_{k+1}$ coincides with the number of factors in $N(m)$ whose $p$-valuation is exactly $k$, we have
$$\mathrm{val}_p(N(m)) = \sum\limits_{k \geq 0} k( |J_k|-|J_{k+1}|) = \sum\limits_{k \geq 0 } k |J_k| - \sum\limits_{k \geq 1} (k-1)|J_k| = \sum\limits_{k \geq 1} |J_k|.$$
Setting $Q:= \ln(q) / \ln(p)$, notice that $J_k= \emptyset$ if $k> \lfloor mQ \rfloor$. Indeed, if $J_k \neq \emptyset$, there exists $j \leq m$ such that $p^k$ divides $q^j-1$, hence $p^k \leq q^j-1 \leq q^m$ and finally $k \leq mQ$. Therefore,
\begin{equation}\label{valp}
\mathrm{val}_p(N(m)) = \sum\limits_{k=1}^{\lfloor mQ \rfloor} |J_k|.
\end{equation}
Now, notice that, if $i<j$ both belong to $J_k$, then $p^k$ divides $(q^j-1)-(q^i-1)= q^{i} (q^{j-i}-1)$. From the fact that $p^k$ has to divide $q^{j-i}-1$, we know that $j-i$ also belongs to $J_k$. As a consequence, setting $j_k:= \min(J_k)$, any two consecutive points in $J_k$ must be at distance $\geq j_k$, hence $|J_k| \leq m/j_k$. But, because $p^k$ divides $p^{j_k}-1$, we must have $k \leq Qj_k$. So $|J_k| \leq mQ/k$. Combining this observation with Equation~(\ref{valp}), it follows that
$$\mathrm{val}_p(N(m)) = \sum\limits_{k=1}^{\lfloor mQ \rfloor} |J_k| \leq mQ \sum\limits_{k=1}^{\lfloor mQ \rfloor} \frac{1}{k} = mQ \ln(m) + O(m),$$
where the last equality is justified by the well-known estimate $\sum\limits_{s=1}^r 1/s = \ln(r)+ O(1)$. 
\end{proof}

\begin{proof}[Proof of Theorem~\ref{thm:Lampcloner}.]
Assume for contradiction that $\mathfrak{h},\mathfrak{k}$ have distinct characteristic but that $\oslash_\mathfrak{h} A, \oslash_\mathfrak{k}B$ are quasi-isometric. As a consequence of Theorem~\ref{thm:InALeaf}, such a quasi-isometry can be chosen leaf-preserving; and, as a consequence of Corollary~\ref{cor:AptoAltitude} and Lemma~\ref{lem:AltitudeCloner}, the quasi-isometry can be chosen even aptolic. Thus, Proposition~\ref{prop:LampGrowthArithm} implies that there exists a sequence $(x_n)$ satisfying $x_n= \Theta(n)$ such that, for every $n \geq 0$, 
$$\Lambda_\mathfrak{h}(n)= |\mathfrak{h}|^{n(n-1)/2} (|\mathfrak{h}|-1) (|\mathfrak{h}|^2-1) \cdots (|\mathfrak{h}|^n-1)$$ 
divides 
$$\Lambda_\mathfrak{k}(x_n) = |\mathfrak{k}|^{x_n(x_n-1)/2} (|\mathfrak{k}|-1) (|\mathfrak{k}|^2-1) \cdots (|\mathfrak{k}|^{x_n}-1).$$ 
Setting $p:= \mathrm{char}( \mathfrak{h})$ and $q:= |\mathfrak{k}|$, it follows that $n(n-1)/2 \leq \mathrm{val}_p(N(x_n))$. On the other hand, it follows from Lemma~\ref{lem:CloneEstimate} that $\mathrm{val}_p(N(x_n)) = O(x_n \ln(x_n)) = O(n\ln(n))$. We get a contradiction, proving that our two lampcloner groups cannot be quasi-isometric if their underlying finite fields have distinct characteristic. 
\end{proof}

\section{The case of amenable groups}\label{section:amenable}

\noindent
We saw in Section~\ref{section:ApplicationsTwo} that, if there exists an aptolic quasi-isometry between two finitely generated halo groups $\mathscr{M}A \to \mathscr{N}B$, then the lamp growth $\Lambda_\mathscr{M}$ is $\mathbb{B}$-interlaced with $\Lambda_\mathscr{N}$, where $\mathbb{B}$ denotes the boundary growth of $A$. This observation provided useful information about lamplighter and lampcloner groups, but it does not bring anything interesting for lampjuggler and lampdesigner groups in full generality. In this section, we show that, in the case where $A$ is amenable, it is possible to exploit the sublinearity of $\mathbb{B}$ given by Lemma~\ref{lem:BoundaryGrowth} to get new information. 

\medskip \noindent
Typically, the strategy is the following. 
\begin{description}
	\item[(Step 1)] Given a quasi-isometry $\Psi : \mathscr{M}A \to \mathscr{N}B$, we apply Theorem~\ref{thm:InALeaf} in order to justify that $\Psi$ can be chosen leaf-preserving. 
	\item[(Step 2)] Next, we apply Corollary~\ref{cor:AptoAltitude} and the associated Lemmas~\ref{lem:LighterAltitude}--~\ref{lem:AltitudeCloner} in order to justify that $\Psi$ can be chosen aptolic. 
	\item[(Step 3)] Applying Proposition~\ref{prop:LampGrowthArithm} and using the sublinearity of the boundary growth of $A$ given by Lemma~\ref{lem:BoundaryGrowth}, we find two sequences $(x_n), (y_n)$ satisfying $x_n= \Theta(y_n)$ such that, for every prime number $p$, $$\mathrm{val}_p \Lambda_\mathscr{M}(y_n) \leq \mathrm{val}_p \Lambda_\mathscr{N}(x_n) \leq \mathrm{val}_p \Lambda_\mathscr{M}(y_n+o(n)).$$ Thanks to good estimates, we deduce that the sequence $(x_n/y_n)$ converges to some expression depending on $p$. Since such a limit cannot depend on $p$, we deduce some arithmetic restriction on $\mathscr{M}$ and $\mathscr{N}$. 
	\item[(Step 4)] Finally, under the restrictions found in the previous step, we write our aptolic quasi-isometry $\Psi$ as $(c,p) \mapsto (\alpha(c),\beta(p))$ for some bijection $\alpha : M(A) \to N(B)$ and some quasi-isometry $\beta : A \to B$, and we show that $\beta$ must be quasi-$\kappa$-to-one for some specific $\kappa>0$ thanks to Proposition~\ref{prop:Bowtie}.
\end{description}
The fourth step requires some explanation. First, recall from \cite{MR4419103} the following definition:

\begin{definition}
Let $f : X \to Y$ be a quasi-isometry between two graphs of bounded degrees. Given a real number $\kappa>0$, $f$ is \emph{quasi-$\kappa$-to-one} if there exists a constant $C \geq 0$ such that
$$\left|  |f^{-1}(S)| - \kappa \cdot |S| \right| \leq C \cdot |\partial S|$$
for every finite subgraph $S \subset X$.
\end{definition}

\noindent
If $X$ and $Y$ are non-amenable, then a quasi-isometry $X \to Y$ is quasi-$\kappa$-to-one for every $\kappa>0$. However, if $X$ and $Y$ are amenable, then, if a quasi-isometry $X \to Y$ is both quasi-$\kappa_1$-to-one and quasi-$\kappa_2$-to-one, then we must have $\kappa_1=\kappa_2$. We emphasize, however, that a quasi-isometry may not be quasi-$\kappa$-to-one for any $\kappa>0$. This number $\kappa$ quantifies the obstruction for a quasi-isometry to lie at finite distance from a bijection. In particular, a quasi-isometry is quasi-one-to-one if and only if it lies at finite distance from a bijection. We refer to \cite{MR4419103} for more information on the subject. 

\medskip \noindent
In order to understand when a quasi-isometry between halo groups is quasi-$\kappa$-to-one, we need the following definition:

\begin{definition}
Let $\zeta,\xi : \mathbb{N} \to \mathbb{N}$ be two non-decreasing sequences and $r>0$ a real number. We write $\zeta \lhd_r \xi$ if there exists $t\in \mathbb{N}$ such that 
$$\zeta (k)\leq\xi (\lfloor r k\rfloor+t),$$
for all large enough $k$. 
We write $\zeta \bowtie_r \xi$ if $\zeta \lhd_r \xi$ and $\xi \lhd_{1/r} \zeta$. 
\end{definition}

\noindent
Now, we are ready to state and prove the statement mentioned in the fourth step described above. 

\begin{prop}\label{prop:Bowtie}
Let $\mathscr{M}A,\mathscr{N}B$ be two finitely generated halo groups. Assume that $M(A),N(B)$ are locally finite, that $\mathscr{M},\mathscr{N}$ are charming, and  satisfy $\Lambda_\mathscr{N} \bowtie_\kappa \Lambda_\mathscr{M}$ for some $\kappa>0$. Assume that there exist a bijection $\alpha : M(A) \to N(B)$ and a quasi-isometry $\beta : A \to B$ such that $(c,p) \mapsto (\alpha(c), \beta(p))$ defines an aptolic quasi-isometry $\mathscr{M}A \to \mathscr{N}B$. Then $\beta$ is quasi-$\kappa$-to-one.
\end{prop}

\noindent
We begin by proving two preliminary observations:

\begin{lemma}\label{lem:QuasiToOneConnected}
Let $f : X \to Y$ be a quasi-isometry between two graphs of bounded degrees and let $\kappa>0$ be a real number. Then $f$ is quasi-$\kappa$-to-one if and only if there exists a constant $C \geq 0$ such that
$$\left| |f^{-1}(S)|- \kappa \cdot |S| \right| \leq C \cdot |\partial S|$$
for every finite and $2$-coarsely connected subgraph $S \subset X$.
\end{lemma}

\begin{proof}
Assume that there exists a $C \geq 0$ such that the condition above holds and let $S \subset X$ be an arbitrary finite subgraph. Let $S_1, \ldots, S_r$ denotes the $2$-coarsely connected components of $S$. Then
$$\left| |f^{-1}(S)| -  \kappa \cdot |S| \right| \leq \sum\limits_{i=1}^r \left| |f^{-1}(S_i)| - \kappa \cdot |S_i| \right| \leq C \sum\limits_{i=1}^r |\partial S_i| = C \cdot |\partial S|,$$
where the last equality is justified by the decomposition $\partial S = \partial S_1 \sqcup \cdots \sqcup \partial S_r$. Thus, we conclude that $f$ is quasi-$\kappa$-to-one. The converse is clear.
\end{proof}

\begin{lemma}\label{lem:Boundary}
Let $X,Y$ be two graphs of bounded degrees and $\varphi : X \to Y$ a quasi-isometry. Fix a quasi-inverse $\bar{\varphi} : Y \to X$ and two constants $C,K$ such that $\varphi, \bar{\varphi}$ are $(C,K)$-quasi-isometries with $\varphi \circ \bar{\varphi}, \bar{\varphi} \circ \varphi$ at distance $\leq K$ from identities. There exist constants $P,Q$ such that, if $|U| > \kappa |\varphi^{-1}(U)|+ D |\partial U|$ for some $U \subset Y$ and $D>0$, then 
$$\left| \bar{\varphi}^{-1}(V) \right| \geq \kappa \cdot |V| + \frac{D-P}{Q} \cdot |\partial V|$$
with $V:= \bar{\varphi}(U)^{+K}$. 
\end{lemma}

\begin{proof}
First of all, let us observe that there exists some $L >0$ independent of $U$ such that
\begin{equation}\label{eq:first}
\left| |\varphi^{-1}(U)|- | \bar{\varphi}(U)^{+K}| \right| \leq L \cdot |\partial U |.
\end{equation}
Indeed, observe that
$$\varphi^{-1}(U) \subset \bar{\varphi}(U)^{+K} \subset \varphi^{-1} \left( U^{+C(3K+1)} \right);$$
see the proof of \cite[Lemma~3.6(iii)]{MR4419103} for a short justification. Consequently, 
$$\left| |\varphi^{-1}(U)|- | \bar{\varphi}(U)^{+K}| \right| \leq \left| \varphi^{-1} \left( \partial_{C(3K+1)} U \right) \right|.$$
Recall that $\partial_R U$ denotes $U^{+R} \backslash U$. The desired inequality now follows from \cite[Fact~2.12(ii)]{MR4419103}.

\medskip \noindent
Next, let us observe that there exists some $M>0$ independent of $U$ such that
\begin{equation}\label{eq:second}
\left| |\bar{\varphi}^{-1}( \bar{\varphi}(U)^{+K})| - |U| \right| \leq M \cdot |\partial U|.
\end{equation}
If a point $x$ satisfies $\bar{\varphi}(x) \in \bar{\varphi}(U)^{+K}$, there exists some $u \in U$ such that $d(\bar{\varphi}(u), \bar{\varphi}(x)) \leq K$, hence $d(x,u) \leq C(d(\bar{\varphi}(x), \bar{\varphi}(u))+K) \leq 2CK$. Therefore, we have $U \subset \bar{\varphi}^{-1} \left( \bar{\varphi}(U)^{+K} \right) \subset U^{+2CK}$. This implies that
$$\left| |\bar{\varphi}^{-1}( \bar{\varphi}(U)^{+K})| - |U| \right| \leq | \partial_{2CK}U|,$$
and the desired inequality follows from \cite[Fact~2.8(ii)]{MR4419103}.

\medskip \noindent
Finally, let us observe that there exists some $N>0$ independent of $U$ such that
\begin{equation}\label{eq:third}
\left| \partial \bar{\varphi}(U)^{+K} \right| \leq N \cdot |\partial U|.
\end{equation}
Given a point $x$ in the boundary of the $K$-neighbourhood of $\bar{\varphi}(U)$, fix a point at distance $\leq K$ from $x$ that lies in the image of $\bar{\varphi}$, say $\bar{\varphi}(y)$. If $y \in U$, then $d(x,\bar{\varphi}(U)) \leq d(x,\bar{\varphi}(y)) \leq K$, a contradiction. Therefore, $y \notin U$. On the other hand,
$$d(y,U) \leq C \left( d(\bar{\varphi}(y), \bar{\varphi}(U) ) + K \right) \leq C \left( d(\bar{\varphi}(y),x) + d(x, \bar{\varphi}(U)) +K \right) \leq C(3K+1).$$
Thus, we have proved that $\partial \bar{\varphi}(U)^{+K} \subset \bar{\varphi} \left( \partial_{C(3K+1)} U \right)$, which implies that $|\partial \bar{\varphi}(U)^{+K}| \leq |\partial_{C(3K+1)} U|$. The desired inequality follows from \cite[Fact~2.8(ii)]{MR4419103}.

\medskip \noindent
We conclude from the inequalities~\ref{eq:first}, \ref{eq:second}, and~\ref{eq:third} that
$$\begin{array}{lcl} \displaystyle \left| \bar{\varphi}^{-1}(V) \right| & \geq & \displaystyle |U|-M |\partial U| > \kappa |\varphi^{-1}(U)+ (D-M) |\partial U| \\ & > & \displaystyle \kappa |V| + (k-M-\kappa L) |\partial U| \geq \kappa |V| + \frac{D-M-\kappa L}{N} |\partial V| \end{array}$$
which is the inequality given by our lemma. 
\end{proof}

\begin{proof}[Proof of Proposition~\ref{prop:Bowtie}.]
We start deducing from Lemma \ref{lem:InclusionAlpha} that
there exists a constant $K \geq 0$ such that 
$$\Lambda_\mathscr{M} \left( \beta^{-1}(S) \right) \leq \Lambda_\mathscr{N} \left( S^{+K} \right)$$ 
for every $S \subset B$ $2$-coarsely connected. Indeed, fixing some $K_0 \geq 0$ such that $\beta^{-1}(S^{+K_0})$ is connected, we deduce from Lemma~\ref{lem:InclusionAlpha} that, for some constant $K_1 \geq 0$ that does not depend on $S$, we have
$$\Lambda_\mathscr{M} \left( \beta^{-1} (S) \right) \leq \Lambda_{\mathscr{M}} \left( \beta^{-1} \left( S^{+K_0} \right) \right) \leq \Lambda_\mathscr{N} \left( \beta \left( \beta^{-1} \left( S^{+K_0} \right) \right)^{+K_1} \right) \leq \Lambda_\mathscr{N} \left( S^{+K} \right)$$
where $K:=K_0+K_1$, as desired.

\medskip \noindent
Now, assume that the quasi-isometry $\beta : A \to B$ is not quasi-$\kappa$-to-one. So, for every $k \geq 1$, there exists $A_k \subset B$ such that $\left|  |\beta^{-1}(A_k)|- \kappa |A_k| \right| >k |\partial A_k|$. According to Lemma~\ref{lem:QuasiToOneConnected}, we can assume that the $A_k$ are all $2$-coarsely connected. Up to extracting a subsequence, we assume that either $|\beta^{-1}(A_k)| > \kappa |A_k| +k |\partial A_k|$ for every $k \geq 1$ or $\kappa |A_k|>  |\beta^{-1}(A_k)| + k |\partial A_k|$. In the former case, we write
$$\Lambda_{\mathscr{M}} \left( \left\lfloor \kappa |A_k| + k |\partial A_k| \right\rfloor \right) < \Lambda_\mathscr{M} \left( |\beta^{-1}(A_k) | \right) \leq \Lambda_\mathscr{N}\left( |A_k^{+K}| \right) \leq \Lambda_\mathscr{N} \left( |A_k| + L|\partial A_k| \right),$$
for some constant $L$. 
Fix some $t\in \mathbb{N}$. 
Let $n_k=|A_k| + L |\partial A_k|$. We have
\[  \kappa |A_k| + k |\partial A_k| =   \kappa n_k +  (k-\kappa L) |\partial A_k|.\]
Hence for $k\geq \kappa L+ (t+1)/|\partial A_k|$, we have  
\[  \left\lfloor \kappa |A_k| + k |\partial A_k| \right\rfloor \geq \left\lfloor \kappa n_k \right\rfloor + t. \]
We deduce that for all such $k$, \[\Lambda_\mathscr{N}(n_k)>\Lambda_{\mathscr{M}} \left( \left\lfloor \kappa n_k \right\rfloor + t \right).\]
Since $t$ is arbitrary, this contradicts the fact that $\Lambda_\mathscr{N} \lhd_{\kappa} \Lambda_\mathscr{M}$.
In the latter case, we work instead with a quasi-inverse $\bar{\beta}$; and, applying Lemma~\ref{lem:Boundary}, we deduce similarly a contradiction with $\Lambda_\mathscr{M} \lhd_{1/\kappa} \Lambda_\mathscr{N}$. 
\end{proof}

\subsection{Lamplighter groups}\label{section:AmenableLighter}

\noindent
As a warm up, let us illustrate the strategy described above with lamplighter groups by proving the following improvement of \cite[Corollary~1.6]{LampGT}:

\begin{thm}\label{thm:LampAmenable}
Let $E,F$ be two finite groups and $H,K$ two finitely generated amenable groups satisfying the thick bigon property. If $E \wr H$ and $F \wr K$ are quasi-isometric, then $|E|,|F|$ are powers of a common number, say $|E|=q^r$ and $|F|=q^s$, and there exists quasi-$(s/r)$-to-one quasi-isometry $H \to K$.
\end{thm}

\noindent
One easily shows that the converse of the theorem also holds. See \cite{LampGT} for more details.

\begin{proof}[Proof of Theorem~\ref{thm:LampAmenable}.]
For convenience, set $e:=|E|$ and $f:=|F|$. As a consequence of Theorem~\ref{thm:InALeaf} and Corollary~\ref{cor:AptoAltitude}, which applies according to Lemma~\ref{lem:LighterAltitude}, we know that, if $E \wr H$ and $F \wr K$ are quasi-isometric, then there exists an aptolic quasi-isometry $\Psi : E \wr H \to F \wr K$. It follows from Proposition~\ref{prop:LampGrowthArithm} that there exist two sequences $(x_n), (y_n)$ satisfying $x_n= \Theta(y_n)$ such that
$$\Lambda_e(y_n)=e^{y_n} \text{ divides } \Lambda_f(x_n)=f^{x_n}, \text{ which divides } \Lambda_e(y_n+ o(n))= f^{y_n+o(n)}.$$
Thus, given a prime number $p \geq 2$, 
$$y_n \cdot \mathrm{val}_p(e) \leq x_n \cdot \mathrm{val}_p(f) \leq (y_n+o(n)) \cdot \mathrm{val}_p(e),$$
which implies that $\mathrm{val}_p(e)=0$ if and only if $\mathrm{val}_p(f)=0$, i.e.\ $e$ and $f$ have the same prime divisors; and that $(x_n/y_n)$ converges to $\mathrm{val}_p(e)/\mathrm{val}_p(f)$ for every prime $p$ divising $e$ and $f$. Therefore, the latter quantity cannot depend on the prime $p$, i.e.\ there exist $r,s \geq 1$ such that $\mathrm{val}_p(e)/\mathrm{val}_p(f)=r/s$ for every prime $p$ dividing $e$ and $f$. This amounts to saying that there exists some integer $q$ such that $e=q^r$ and $f=q^s$, as desired.

\medskip \noindent
Now, fix a bijection $\alpha : \bigoplus_HE \to \bigoplus_KF$ and a quasi-isometry $\beta : H \to K$ such that our aptolic quasi-isometric $\Psi$ can be written as $(c,p) \mapsto (\alpha(c),\beta(p))$. We want to prove that $\beta$ is quasi-$(s/r)$-to-one. By Proposition~\ref{prop:Bowtie}, it suffices to prove that $\Lambda_f  \bowtie_{s/r} \Lambda_e$.
Note that  \[\Lambda_f(n)=f^n=q^{sn}\leq q^{r(\lfloor sn/r \rfloor +1)}\leq \Lambda_e(\lfloor sn/r \rfloor +1),\]
 which proves  $\Lambda_f \lhd_{s/r} \Lambda_e$. The converse inequality is proved identically.
\end{proof}

\subsection{Lampshuffler and lampjuggler groups}

\noindent
Now, let us turn to lampjuggler groups. We emphasize that, in the classification given below, we are not assuming that the groups $H$ and $K$ are amenable.

\begin{thm}\label{thm:LampjugglerClassification}
Let $r,s \geq 1$ be two integers and $H,K$ two finitely generated groups satisfying the thick bigon property. The lampjuggler groups $\circledS_r H$ and $\circledS_s K$ are quasi-isometric if and only if there exists a quasi-$(s/r)$-to-one quasi-isometry $H \to K$. 
\end{thm}

\noindent
We begin by proving the converse of the theorem, without assuming that our groups $H$ and $K$ satisfying the thick bigon property.

\begin{lemma}\label{lem:QIjuggler}
Let $r,s \geq 1$ be two integers and $H,K$ two finitely generated groups. If there exists a quasi-$(s/r)$-to-one quasi-isometry $H \to K$, then the lampjuggler groups $\circledS_r H$ and $\circledS_s K$ are quasi-isometric.
\end{lemma}

\begin{proof}
Let $\beta : H \to K$ be a quasi-$(s/r)$-to-one quasi-isometry. According to \cite[Proposition~4.2]{MR4419103}, there exists a bijection $\alpha_0 : H \times \{1, \ldots, r\} \to K \times \{1, \ldots, s\}$ that lies at finite distance (say $D$) from $\iota_K \circ \beta \circ \pi_H$, where $\pi_H$ denotes the canonical projection $H \times \{1 ,\ldots, r\} \to H$ and $\iota_K$ the canonical embedding $K \to K \times \{1, \ldots, s\}$. (Here, we think of $H \times \{1, \ldots, r\}$ (resp. $K \times \{1, \ldots, s\}$) as the product of (the Cayley graph of) $H$ (resp. $K$) with a complete graph of size $r$ (resp. $s$).) Let 
$$\alpha : \mathrm{FSym}(H \times \{1, \ldots, r\}) \overset{\sim}{\longrightarrow} \mathrm{FSym}(K \times \{1, \ldots, s\})$$
denote the bijection induced by $\alpha_0$. We claim that $\Psi : (c,p) \mapsto (\alpha(c),\beta(p))$ defines a quasi-isometry $\circledS_rH \to \circledS_sK$. 

\medskip \noindent
First, we show that $\Psi$ is Lipschitz. So let $(c_1,p_1),(c_2,p_2) \in \circledS_rH$ be two adjacent vertices. Two cases may happen. First, we can have $c_1=c_2$ (say equal to $c$) and $p_1,p_2$ adjacent in $H$. In this case, 
$$d(\Psi(c_1,p_1),\Psi(c_2,p_2))= d((\alpha(c),\beta(p_1)), (\alpha(c), \beta(p_2))) = d(\beta(p_1),\beta(p_2))$$
is bounded above by a constant $C$ depending only on the parameters of $\beta$. Next, we can have $p_1=p_2$ (say equal to $p$) and $c_2=c_1 \mu$ where $\mu$ switches a point $a \in \{p\} \times \{1, \ldots, r\}$ with a point $b$ in $\{q \} \times \{1, \ldots, r\}$ for some $q \in H$ satisfying $d(p,q) \leq 1$. Notice that $\Psi(c_2,p)=( \alpha(c_1) \nu, \beta(p))$ where $\nu$ switches $\alpha_0(a)$ and $\alpha_0(b)$. Therefore,
$$\begin{array}{lcl} d( \Psi(c_1,p_1), \Psi(c_2,p_2)) & = & d( ( \alpha(c_1), \beta(p)), (\alpha(c_1) \nu, \beta(p))) \\ \\ & \leq & 2 d(\beta(p), \pi_K(\alpha_0(a))) +2 d(\pi_K(\alpha_0(a)), \pi_K(\alpha_0(b))) \\ \\ & \leq & 4 d(\beta(p), \pi_K(\alpha_0(a))) +2 d(\beta(p),\beta(q)) +2 d(\beta(q), \pi_K(\alpha_0(b))). \end{array}$$
But we know that
$$\begin{array}{lcl} d(\beta(p),\pi_K(\alpha_0(a))) & = & d(\pi_K(\iota_K(\beta(p))), \pi_K(\alpha_0(a))) \leq d(\iota_K(\beta(p)), \alpha_0(a)) +1 \\ \\ & \leq & d(\iota_K(\beta(\pi_H(a))), \alpha_0(a)) +1\leq D+1,\end{array}$$
and similarly that $d(\beta(q), \pi_K(\alpha_0(b))) \leq D+1$. Consequently,
$$d( \Psi(c_1,p_1), \Psi(c_2,p_2)) \leq 6(D+1) + 2C.$$
We conclude that $\Psi$ is indeed Lipschitz. Given a quasi-inverse $\bar{\beta} : K \to H$, which is necessarily quasi-$(r/s)$-to-one according to \cite[Proposition~3.6]{MR4419103}, the same argument shows that $\bar{\Psi} : (c,p) \mapsto ( \alpha^{-1}(c), \bar{\beta}(p))$ is Lipschitz. We conclude that $\Psi$ is a quasi-isometry with $\bar{\Psi}$ as a quasi-inverse.
\end{proof}

\begin{proof}[Proof of Theorem~\ref{thm:LampjugglerClassification}.]
As a consequence of Theorem~\ref{thm:InALeaf} and Corollary~\ref{cor:AptoAltitude}, which applies according to Lemma~\ref{lem:AltitudeJ}, we know that, if $\circledS_r H$ and $\circledS_s K$ are quasi-isometric, then there exists an aptolic quasi-isometry $\Psi : \circledS_r H \to \circledS_s K$. Fix a bijection $\alpha : \mathrm{FSym}(H \times \{1, \ldots, r\}) \to \mathrm{Sym}(K \times \{1 ,\ldots, s\})$ and a quasi-isometry $\beta : H \to K$ such that our aptolic quasi-isometry $\Psi$ can be written as $(c,p) \mapsto (\alpha(c),\beta(p))$. Since we are not looking for an arithmetic relation between $r$ and $s$, we can skip the third step in our strategy and directly prove that $\beta$ must be quasi-$(s/r)$-to-one. 

\medskip \noindent
By Proposition~\ref{prop:Bowtie}, it suffices to prove that $\Lambda_r \bowtie_{r/s} \Lambda_s$, where $\Lambda_p$ denotes the map $n \mapsto (pn)!$.
Indeed, we have
\[\Lambda_r(n)=(rn)!\leq (s(\lfloor rn/s\rfloor+1))!=\Lambda_s(\lfloor rn/s\rfloor+1),\]
which proves that $\Lambda_r \lhd_{r/s} \Lambda_s$. The other direction is proved identically.

\medskip \noindent
The converse of the theorem is proved by Lemma~\ref{lem:QIjuggler}.
\end{proof}

\noindent
As a particular case of Theorem~\ref{thm:LampjugglerClassification}, notice that:

\begin{cor}\label{cor:QILampshuffler}
Let $H,K$ be two finitely generated groups satisfying the thick bigon property. The lampshuffler groups $\circledS H$ and $\circledS K$ are quasi-isometric if and only if $H$ and $K$ are bijectively quasi-isometric. 
\end{cor}

\noindent
With Corollary~\ref{cor:ManyNotQI}, we distinguished various halo groups of different nature by comparing lamp growths. However, lampjuggler and lampdesigner groups cannot be distinguished that way, because their lamp growths are too similar. They diverge to infinity with comparable speeds, but they are also nicely interlaced so that Proposition~\ref{prop:ArithmEquivalent} cannot be applied directly. However, in the amenable case, these two types of halo groups can be distinguished.

\begin{prop}\label{prop:DesignerVsJuggler}
If $F$ is a finite group, $r \geq 1$ an integer, and $H,K$ two finitely generated amenable groups satisfying the thick bigon property, then $F \boxplus K$ and $\circledS_r H$ cannot be quasi-isometric. 
\end{prop}

\noindent
We begin by proving two estimates which will be useful during the proof of our  proposition.

\begin{fact}\label{fact:pAdicFactorial}
Let $x,n$ be two integers and $p$ a prime number. Then
$$\frac{x}{p-1} - \frac{p}{p-1} - \log_p(x) \leq \mathrm{val}_p(x!) \leq \frac{x}{p-1}- \frac{1}{p-1}.$$
\end{fact}

\begin{proof}
We know from Legendre's formula that
$$\mathrm{val}_p(x!) = \sum\limits_{i=1}^{\lfloor \log_p(x) \rfloor} \left\lfloor \frac{x}{p^i} \right\rfloor.$$
But
$$\begin{array}{lcl} \displaystyle \sum\limits_{i=1}^{\lfloor \log_p(x) \rfloor} \left\lfloor \frac{x}{p^i} \right\rfloor & \leq & \displaystyle x \sum\limits_{i=1}^{\lfloor \log_p(x) \rfloor} \frac{1}{p^i} = x \frac{p^{- \lfloor \log_p(x) \rfloor -1} - p^{-1} }{p^{-1}-1} = \frac{px}{p-1} \left( \frac{1}{p} - \frac{1}{p^{\lfloor \log_p(x) \rfloor +1}} \right) \\ \\ & \leq & \displaystyle \frac{px}{p-1} \left( \frac{1}{p} - \frac{1}{px} \right) = \frac{x}{p-1}- \frac{1}{p-1} \end{array}$$
and similarly
$$\begin{array}{lcl} \displaystyle \sum\limits_{i=1}^{\lfloor \log_p(x) \rfloor} \left\lfloor \frac{x}{p^i} \right\rfloor & \geq & \displaystyle x \sum\limits_{i=1}^{\lfloor \log_p(x) \rfloor} \frac{1}{p^i} - \log_p(x) = \frac{px}{p-1} \left( \frac{1}{p} - \frac{1}{p^{\lfloor \log_p(x) \rfloor +1}} \right) - \log_p(x) \\ \\ & \geq & \displaystyle \frac{px}{p-1} \left( \frac{1}{p}- \frac{1}{x} \right) - \log_p(x) = \frac{x}{p-1} - \frac{p}{p-1} - \log_p(x) \end{array}$$
This concludes the proof of our fact.
\end{proof}

\begin{fact}\label{fact:pAdicDesigner}
Given two integers $a,x \geq 1$ and a prime number $p$,
$$\mathrm{val}_p(a) x + \frac{x}{p-1}- \frac{p}{p-1} - \log_p(x) \leq \mathrm{val}_p(a^{x}x!) \leq \mathrm{val}_p(a) x+ \frac{x}{p-1} - \frac{1}{p-1}.$$
\end{fact}

\begin{proof}
Since $\mathrm{val}_p(a^{x}x!) = \mathrm{val}_p(a)x + \mathrm{val}_p(x!)$, Fact~\ref{fact:pAdicFactorial} yields the desired conclusion.
\end{proof}

\begin{proof}[Proof of Proposition~\ref{prop:DesignerVsJuggler}.]
Assume for contradiction that there exists a quasi-isometry $\circledS_r H \to F \boxplus K$. As a consequence of Theorem~\ref{thm:InALeaf} and Corollary~\ref{cor:AptoAltitude}, which applies according to Lemma~\ref{lem:AltitudeJ}, we know that such an isometry can be chosen aptolic. 

\medskip \noindent
The lamp growth of $F \boxplus K$ is $\Lambda : n \mapsto f^nn!$, where $f:=|F|$, and the lamp growth of $\circledS_rH$ is $\Lambda_r : n \mapsto (rn)!$. It follows from Proposition~\ref{prop:LampGrowthArithm} that there exist two sequences $(x_n),(y_n)$ satisfying $y_n \geq n$ and $x_n= \Theta(y_n)$ such that
$$\Lambda_r(y_n) \text{ divides } \Lambda_f(x_n), \text{ which divides } \Lambda_r(y_n+o(n)).$$
We deduce from the estimates given by Facts~\ref{fact:pAdicFactorial} and~\ref{fact:pAdicDesigner} that, for every prime $p$,
$$\frac{r}{p-1}y_n + o(y_n) \leq \left( \mathrm{val}_p(f)+ \frac{1}{p-1} \right) x_n + o(x_n) \leq \frac{r}{p-1}y_n +o(y_n).$$
Therefore, the sequence $(ry_n/x_n)$ converges to $(p-1)\mathrm{val}_p(f)+1$, which implies that the latter quantity does not depend on the prime $p$. So it must equal to $1$ all the time since $\mathrm{val}_p(f)=0$ for $p$ large enough. Hence $\mathrm{val}_p(f)=0$ for every prime $p$, i.e.\ $f=1$. 

\medskip \noindent
Thus, we find that $\circledS_r H$ is quasi-isometric to $K$, which is impossible since we know from Theorem~\ref{thm:InALeaf} that $\circledS_r H$ cannot satisfy the thick bigon property. 
\end{proof}

\subsection{Lampdesigner groups}

\noindent
In this section, we focus on lampdesigner groups. Our next statement shows that they are rather rigid. This can be thought of as a consequence of the fact that, in the decomposition of the lamp growth of a lampdesigner as a product of the lamp growth of a lamplighter with the lamp growth of a lampjuggler, the factorial term is dominant, i.e.\ the rigidity of the lampjugglers given by Theorem~\ref{thm:LampjugglerClassification} wins over the flexibility of the lamplighters. 

\begin{thm}\label{thm:ClassificationDesigner}
Let $E,F$ be two finite groups and $H,K$ two finitely generated amenable groups satisfying the thick bigon property. The lampdesigner groups $E \boxplus H$ and $F \boxplus K$ are quasi-isometric if and only if $|E|=|F|$ and $H,K$ are bijectively quasi-isometric.
\end{thm}

\noindent
We begin by proving that the converse of the theorem holds in greater generality.

\begin{lemma}\label{lem:QIdesigner}
Let $E,F$ be two finite groups and $H,K$ two finitely generated groups. If $|E|=|F|$ and if $H,K$ are bijectively quasi-isometric, then $E\boxplus H$ and $F \boxplus K$ are quasi-isometric.
\end{lemma}

\begin{proof}
Assume that there exist a bijection $\alpha_0 : E \to F$ and a bijective quasi-isometry $\beta : H \to K$. The bijection $\alpha_0$ naturally induces a bijection
$$\alpha : \left\{ \begin{array}{ccc} E \wr_H \mathrm{FSym}(H) & \to & F \wr_K \mathrm{FSym}(K) \\ (c,\sigma) & \mapsto & \left( \alpha_0 \circ c \circ \beta^{-1}, \beta \circ \sigma \circ \beta^{-1} \right) \end{array} \right.$$
whose inverse is $(c,\sigma) \mapsto \left( \alpha_0^{-1} \circ c \circ \beta, \beta^{-1} \circ \sigma \circ \beta \right)$. Let us prove that
$$\Psi : \left\{ \begin{array}{ccc} E \boxplus H & \to & F \boxplus K \\ ((c,\sigma),h) & \mapsto & (\alpha(c,\sigma), \beta(h)) \end{array} \right.$$
is a quasi-isometry. We first prove that $\Psi$ is Lipschitz. So let $x:=((c,\sigma),h),y \in E \boxplus H$ be two adjacent vertices. Three cases may happen.
\begin{itemize}
	\item Assume that there exists $h'\in H$ adjacent to $h$ such that $y=((c,\sigma),h')$. Because $\Psi(y)=(\alpha(c,\sigma),\beta(h'))$, we have $$d(\Psi(x),\Psi(y)) \leq d(\beta(h),\beta(h')) \leq C_1+C_2$$ where $C_1,C_2$ are such that $\beta$ is a $(C_1,C_2)$-quasi-isometry.
	\item Assume that there exists a permutation $\iota= (h \ h')$ for some neighbour $h'$ of $h$ such that $y=((c, \sigma \circ \iota),h)$. Because $$\Psi(y)=\left( \left( \alpha_0 \circ c \circ \beta^{-1}, \left( \beta \circ \sigma \circ \beta^{-1} \right) \circ \left( \beta \circ \iota \circ \beta^{-1} \right) \right), \beta(h) \right)$$ where $\beta^{-1} \circ \iota \circ \beta^{-1}= (\beta(h) \ \beta(h'))$, it follows that $$d(\Psi(x),\Psi(y)) \leq 4 d(\beta(h),\beta(h')) \leq 4(C_1+C_2).$$
	\item Assume that there exists some $e \in E$ such that $y=((c\delta_e^h,\sigma),h)$, where $\delta_e^h$ denotes the map $H \to E$ sending $h$ to $e$ and all the other points to $1$. Then $$\begin{array}{lcl} \Psi(y) & = & \displaystyle \left( \left( \alpha_0 \circ (c\delta_e^h) \circ \beta^{-1}, \beta \circ \sigma \circ \beta^{-1} \right), \beta(h) \right) \\ \\ & = & \displaystyle \left( \left( ( \alpha_0 \circ c \circ \beta^{-1}) (\alpha_0 \circ \delta_e^h \circ \beta^{-1}), \beta \circ \sigma \circ \beta^{-1} \right), \beta(h) \right) \\ \\ & = & \left( \left( (\alpha_0 \circ c \circ \beta^{-1}) \delta_{\alpha_0(e)}^{\beta(h)}, \beta \circ \sigma \circ \beta^{-1} \right), \beta(h) \right) \end{array}$$ hence $d(\Psi(x),\Psi(y)) \leq 1$.
\end{itemize}
As desired, we conclude that $\Psi$ is Lipschitz. The same argument shows that
$$\bar{\Psi} : \left\{ \begin{array}{ccc} F \boxplus K & \to & E \boxplus H \\ ((c, \sigma),k) & \mapsto & \left( \alpha^{-1}(c,\sigma), \overline{\beta}(k) \right) \end{array} \right.,$$
where $\bar{\beta}$ is a quasi-inverse of $\beta$, is also Lipschitz. We conclude that $\Psi$ is a quasi-isometry having $\bar{\Psi}$ as a quasi-inverse. 
\end{proof}

\begin{proof}[Proof of Theorem~\ref{thm:ClassificationDesigner}.]
Assume that there exists a quasi-isometry $\Psi : E \boxplus H \to F \boxplus K$. It follows from Theorem~\ref{thm:InALeaf} and Corollary~\ref{cor:AptoAltitude}, which applies according to Lemma~\ref{lem:AltitudeD}, that $\Psi$ can be chosen aptolic. So there exist a bijection $\alpha : E \wr_H \mathrm{FSym}(H) \to F \wr_K \mathrm{FSym}(K)$ and a quasi-isometry $\beta : H \to K$ such that $\Psi : (c,p) \mapsto (\alpha(c), \beta(p))$. 

\medskip \noindent
For convenience, set $e:=|E|$ and $f:= |F|$. The lamp growth of $E \boxplus H$ (resp. $F \boxplus K$) is $\Lambda_e : n \mapsto e^n n!$ (resp. $\Lambda_f : n \mapsto f^nn!$). We deduce from Proposition~\ref{prop:LampGrowthArithm} that there exist two sequences $(x_n),(y_n)$ satisfying $y_n \geq n$ and $x_n= \Theta(y_n)$ such that
$$\Lambda_e(y_n) \text{ divides } \Lambda_f(x_n), \text{ which divides } \Lambda_e(y_n+o(n)).$$
Given a prime number $p \geq 2$ and using Fact~\ref{fact:pAdicDesigner}, it follows that
$$\left( \mathrm{val}_p(e)+ \frac{1}{p-1} \right) y_n+ o(y_n) \leq \left( \mathrm{val}_p(f)+ \frac{1}{p-1} \right) x_n + o(x_n) \leq \left( \mathrm{val}_p(e) + \frac{1}{p-1} \right) y_n + o(y_n),$$
which implies that the sequence $(x_n/y_n)$ converges and 
$$\lim\limits_{n \to + \infty} \frac{x_n}{y_n} = \frac{(p-1) \mathrm{val}_p(e)+1}{(p-1) \mathrm{val}_p(f)+1}.$$ 
Necessarily, the latter quantity does not depend on $p$, so it must always equal $1$ since $\mathrm{val}_p(e)= \mathrm{val}_p(f)=0$ for $p$ large enough. We conclude that $\mathrm{val}_p(e)= \mathrm{val}_p(f)$ for every prime $p$, i.e.\ $e=f$. 

\medskip \noindent
Next, we claim that $\beta$ is quasi-one-to-one, which amounts to saying that $\beta$ lies at finite distance from a bijection according to \cite[Proposition~4.1]{MR4419103}. This readily follows from Proposition~\ref{prop:Bowtie}.

\medskip \noindent
The converse is given by Lemma~\ref{lem:QIdesigner}. 
\end{proof}

\subsection{Lampcloner groups}

\noindent
About lampcloner groups, most of the work is already done by Theorem~\ref{thm:Lampcloner}. Below, we complete the proof of Theorem~\ref{thm:IntroCloner}. 

\begin{lemma}\label{lem:ClonerKappa}
Let $\mathfrak{h},\mathfrak{k}$ be two finite fields and $A,B$ two finitely generated groups satisfying the thick bigon property. If $\oslash_\mathfrak{h}A$ and $\oslash_\mathfrak{k} B$ are quasi-isometric, then there exists a quasi-$\kappa$-to-one quasi-isometry $A \to B$ where $\kappa:= \sqrt{ \log |\mathfrak{k}| / \log |\mathfrak{h}|}$.
\end{lemma}

\begin{proof}
Assume that there exists a quasi-isometry $\Psi : \oslash_\mathfrak{h} A \to \oslash_\mathfrak{k} B$. It follows from Theorem~\ref{thm:EmbeddingThmGeneral} and Corollary~\ref{cor:AptoAltitude}, which applies according to Lemma~\ref{lem:AltitudeCloner}, that $\Psi$ can be chosen aptolic. So there exist a bijection $\alpha : E(A,\mathfrak{h}) \to E(B,\mathfrak{k})$ and a quasi-isometry $\beta : A \to B$ such that $\Psi : (c,p) \mapsto (\alpha(c), \beta(p))$. We want to prove that $\beta$ is quasi-$\kappa$-to-one, which by Proposition~\ref{prop:Bowtie} amounts to proving that $\Lambda_k \bowtie_{\kappa} \Lambda_h$.

\medskip \noindent
For convenience, set $h:= |\mathfrak{h}|$ and $k:= |\mathfrak{k}|$. Let $\Lambda_h$ (resp. $\Lambda_h$) denote the lamp growth of $\oslash_\mathfrak{h} A$ (resp. $\oslash_\mathfrak{k} B$). 
Recall that by Fact~\ref{fact:GrowthCloner}, we have
$\Lambda_k(n)=n^2 \log(k) + O(n)$. In particular, we have $\Lambda_k(n)\leq n^2 \log(k) +Cn$ and $\Lambda_h(n)\geq n^2 \log(h) -Cn$ for some $C\geq 1$. Hence for all $n\geq 1$, and letting $t=1+2C/\sqrt{\log(h)}$, we have
\begin{eqnarray*}
\Lambda_k(n) \leq  n^2 \log(k) +Cn & \leq &(\kappa n)^2 \log(h)+Cn \\& \leq &  (\lfloor\kappa n\rfloor+1)^2 \log(h)+Cn\\ & \leq & (\lfloor\kappa n\rfloor+t)^2 \log(h)-Cn\\ & \leq & \Lambda_h(\lfloor\kappa n\rfloor+t),
\end{eqnarray*}
which proves $\Lambda_k \lhd_{\kappa} \Lambda_h$, the other direction being similar.
\end{proof}

\section{Concluding remarks}\label{section:Conclusion}

\noindent
In this final section, we record some remarks and open questions regarding what has been done in the article.

\paragraph{Other groups.} So far, our main results dealt with quasi-isometries between halo groups. But it is worth noticing that our techniques may also apply to quasi-isometries between halo groups and other groups that dot not (obviously) decompose as halo products. Proposition~\ref{prop:Heisenberg} below is an illustration of this assertion. 

\medskip \noindent
Set $P:= k[S^{\pm 1},T^{\pm 1}]$ for some finite field $k$. Define
$$\mathrm{LHeis}_2:= \left\{ \left( \begin{array}{ccc} 1 & x & z \\ 0 & 1 & y \\ 0 & 0 & 1 \end{array} \right) \mid x,y,z \in P \right\} \rtimes \mathbb{Z}^2,$$
where $\mathbb{Z}^2$ acts on the left factor as 
$$(n,m) \cdot \left( \begin{array}{ccc} 1 & x & z \\ 0 & 1 & y \\ 0 & 0 & 1 \end{array} \right) = \left( \begin{array}{ccc} 1 & S^nT^m x  & S^{2n}T^{2m}z \\ 0 & 1 & S^nT^m y \\ 0 & 0 & 1 \end{array}  \right).$$
The definition is inspired by \cite[Example~1.3]{CrossWired}, which provides an example of a group quasi-isometric to a lamplighter over $\mathbb{Z}$. 

\begin{prop}\label{prop:Heisenberg}
The group $\mathrm{LHeis}_2$ is not quasi-isometric to a halo product $\mathscr{L} \mathbb{Z}^2$ where $\mathscr{L}$ is large-scale commutative and $L(\mathbb{Z}^2)$ locally finite. 
\end{prop}

\noindent
In particular, contrary to the cyclic case, $\mathrm{LHeis}_2$ is not quasi-isometric to a lamplighter over $\mathbb{Z}^2$. But it cannot be quasi-isometric to a lampjuggler, a lampdesigner, or a lampcloner group over $\mathbb{Z}^2$ either. 

\begin{proof}[Sketch of proof of Proposition~\ref{prop:Heisenberg}.]
We take as generators of $\mathrm{LHeis}_2$ the four elements
$$\left( \left( \begin{array}{ccc} 1 & 1 & 0 \\ 0 & 1 & 0 \\ 0 & 0 & 1 \end{array} \right), (0,0) \right), \ \left( \left( \begin{array}{ccc} 1 & 0 & 0 \\ 0 & 1 & 1 \\ 0 & 0 & 1 \end{array} \right), (0,0) \right), \ (\mathrm{Id}, (1,0)), \ (\mathrm{Id},(0,1)).$$
Observe that
$$\left( \left( \begin{array}{ccc} 1 & x & z \\ 0 & 1 & y \\ 0 & 0 & 1 \end{array} \right), (m,n) \right) \cdot \left( \left( \begin{array}{ccc} 1 & 1 & 0 \\ 0 & 1 & 0 \\ 0 & 0 & 1 \end{array} \right), (0,0) \right) = \left( \left( \begin{array}{ccc} 1 & x + S^nT^m & z \\ 0 & 1 & y \\ 0 & 0 & 1 \end{array} \right), (m,n) \right)$$
and
$$\left( \left( \begin{array}{ccc} 1 & x & z \\ 0 & 1 & y \\ 0 & 0 & 1 \end{array} \right), (m,n) \right) \cdot \left( \left( \begin{array}{ccc} 1 & 0 & 0 \\ 0 & 1 & 1 \\ 0 & 0 & 1 \end{array} \right), (0,0) \right) = \left( \left( \begin{array}{ccc} 1 & x & z + xS^nT^m \\ 0 & 1 & y + S^nT^m \\ 0 & 0 & 1 \end{array} \right), (m,n) \right).$$
Now, consider the following configuration of leaves (i.e.\ of $\mathbb{Z}^2$-cosets) in $\mathrm{LHeis}_2$:
\begin{center}
\includegraphics[width=0.5\linewidth]{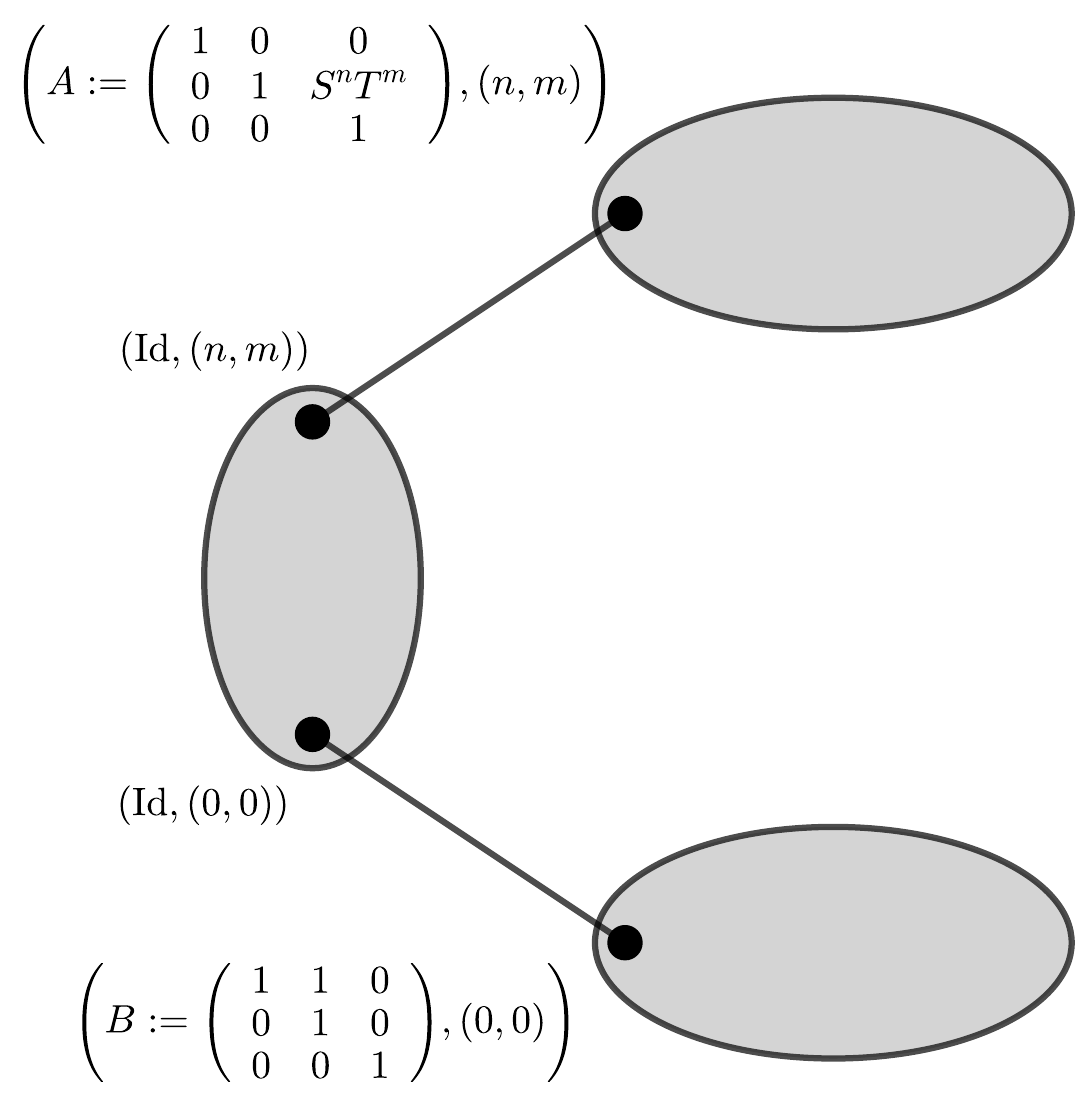}
\end{center}
where $n,m$ are very large. If there exists a quasi-isometry $\Phi : \mathscr{L}\mathbb{Z}^2 \to \mathrm{LHeis}_2$, then it follows from the proof of Theorem~\ref{thm:Peripheral} that $\Phi$ sends each $\mathbb{Z}^2$-coset of $\mathscr{L} \mathbb{Z}^2$ at finite Hausdorff distance from a coset of a subgroup quasi-isometric to $\mathbb{Z}^2$. Such a subgroup can be chosen to be a $\mathbb{Z}^2$, and it follows from the next claim that such a subgroup must be a coset of (a finite index subgroup of) the factor $\mathbb{Z}^2$. 

\begin{claim}
Every free abelian subgroup of rank two in $\mathrm{LHeis}_2$ is contained in a conjugate of the factor $\mathbb{Z}^2$. 
\end{claim}

\noindent
As a consequence, following the proof of Proposition~\ref{prop:GraphLeaves}, one can define the graph of leaves $\mathscr{G}_\eta$ of $\mathrm{LHeis}_2$ for some large $\eta \geq 0$ and show that $\Phi$ induces a bijective map $\mathscr{G}_\epsilon(\mathscr{L}\mathbb{Z}^2) \to \mathscr{G}_\eta$ sending adjacent vertices to adjacent vertices. We deduce from Proposition~\ref{lem:SpanCube} that there must exist a point $(p,q)$ very far from $(n,m)$ and a point $(r,s)$ very far from $(0,0)$, and a third leaf such that:
\begin{center}
\includegraphics[width=0.6\linewidth]{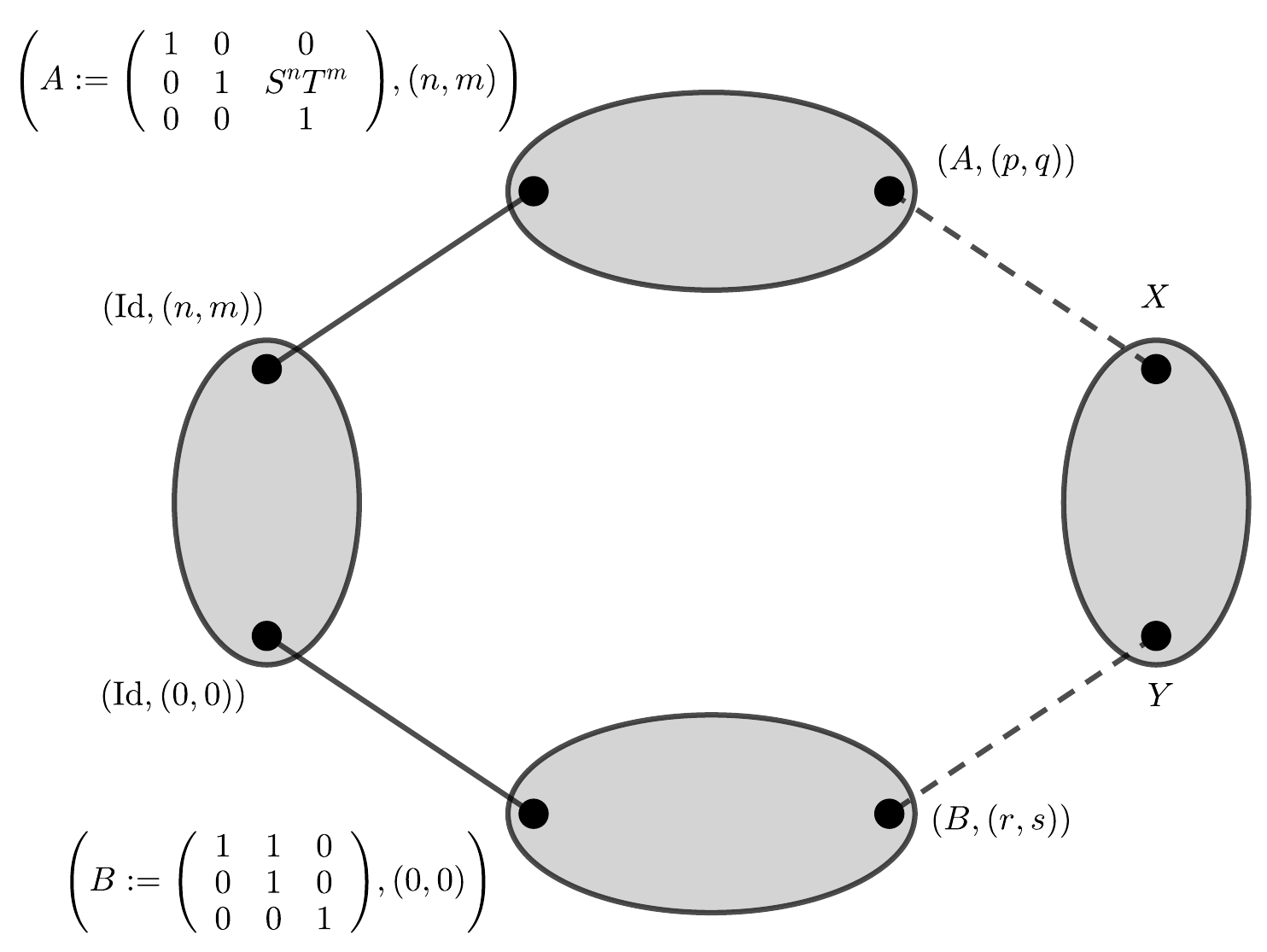}
\end{center}
where $X=(C,(p',q'))$ is close to $(A,(p,q))$ and $Y=(C,(r',s'))$ close to $(B,(r,s))$. Modding out by the centre of the Heisenberg group, we are left with a similar configuration of leaves in a lamplighter group over $\mathbb{Z}^2$, so we know that $(p,q)$ must be close to $(0,0)$, and $(r,s)$ is close to $(n,m)$. 
Since $(A,(p,q))$ is close to $X$ and $(p,q)$ is close to $(0,0)$, there exists \[J=\left( \begin{array}{ccc} 1&a& c\\ 0&1&b \\ 0&0&1 \end{array} \right),\]
where $a,b,c$ have small degrees, such that $C=AJ$. Hence: 
$$C=\left( \begin{array}{ccc} 1&1+a& c \\ 0&1& S^nT^m+b\\ 0&0&1 \end{array} \right).$$
Similarly, since $(B,(r,s))$ is close to $Y=(C,(r',s'))$, and $(r,s)$ is close to $(n,m)$, we have 
$C=BR$, where $R$ is the conjugate by $(n,m)$ of a matrix with coefficients of small degrees. In other words,
\[R=\left( \begin{array}{ccc} 1&S^{n}T^{m}x& S^{2n}T^{2m}z\\ 0&1&S^{n}T^{m}y \\ 0&0&1 \end{array} \right)\]
where $x,y,z$ have small degrees. Hence 
$$C=\left( \begin{array}{ccc} 1&1+S^{n}T^{m}x& S^{2n}T^{2m}z+S^{n}T^{m}y  \\ 0&1&S^{n}T^{m}y \\ 0&0&1 \end{array} \right).$$
This forces $y=1$. Looking at the $z$-coordinate, we get $c=S^{n}T^{m}+zS^{2n}T^{2m}$, contradicting the fact that $c$ has small degrees.
\end{proof}

\paragraph{Open questions.} The results proved in the article naturally lead to several questions, which we record below. 

\medskip \noindent
First, our study of the large-scale geometry of halo products, whose cornerstone is the Embedding Theorem, only deals with halo groups over finitely generated groups satisfying the thick bigon property. Such groups are necessarily one-ended. The case of infinitely-ended groups is not even understood for lamplighters. But it is reasonable to ask for a solution in the two-ended case:

\begin{question}
When are two halo groups over $\mathbb{Z}$ quasi-isometric?
\end{question}

\noindent
For instance, does the coarse differentiation from \cite{EFWI, EFWII} apply to lampjuggler, lampdesigner, or lampcloner groups over $\mathbb{Z}$? Notice that we already know that a group $\mathrm{(finite)}\wr \mathbb{Z}$ cannot be quasi-isometric to a $\circledS_n \mathbb{Z}$, $\mathrm{(finite)} \boxplus \mathbb{Z}$, or $\oslash_\mathrm{finite} \mathbb{Z}$. Indeed, one can apply the arguments from the proof of Proposition~\ref{prop:cones} in order to show that the latter groups have infinite-dimensional asymptotic cones while $\mathrm{(finite)} \wr \mathbb{Z}$ have asymptotic cones of finite topological dimension. But one can ask: when are two lampjugglers $\circledS_n \mathbb{Z}$ and $\circledS_m \mathbb{Z}$ quasi-isometric?

\medskip \noindent
Next, even though we proved our Embedding Theorem (Theorem~\ref{thm:EmbeddingThmGeneral}) for halo groups $\mathscr{L}H$ with $H$ satisfying the thick bigon property and with $L(H)$ locally finite, we proved aptolicity of quasi-isometries between halo groups with more restrictive assumptions. For instance, in Section~\ref{section:Morse}, we assumed that $\mathscr{L}$ is large-scale commutative and that $H$ is stiff; and, in Section~\ref{section:Sufficient}, we assumed finite altitude. 

\begin{question}
Let $\mathscr{L}H$ be a finitely generated halo group with $H$ one-ended finitely presented and $L(H)$ locally finite. Does every quasi-isometry $\mathscr{L}H \to \mathscr{L}H$ lie at finite distance from an aptolic quasi-isometry? 
\end{question}

\noindent
Notice that, removing the assumption on the local finiteness of $L(H)$, a leaf-preserving quasi-isometry $\mathscr{L}H \to \mathscr{L}H$ may not lie at finite distance from a aptolic quasi-isometry. For instance, given some non-trivial finite group $F$, define $\mathscr{L}$ by $L(S):= \ast_{S} F$ for every $S \subset H$. Then the quasi-isometries of $\mathscr{L}H= F \ast H$ are leaf-preserving but not necessarily aptolic (up to finite distance). 

\medskip \noindent
Next, some of the classifications up to quasi-isometry of the groups considered in the article are not complete. It is natural to ask what happens in the remaining cases.

\begin{question}
Let $F$ be a finite group, $n \geq 1$ an integer, and $H,K$ two one-ended finitely presented groups. Can the lampjuggler $\circledS_nH$ be quasi-isometric to the lampdesigner $F \boxplus K$?
\end{question}

\noindent
Proposition~\ref{prop:DesignerVsJuggler} gives a negative answer when $H$ and $K$ are amenable. What does happen when $H$ and $K$ are non-amenable?

\begin{question}
Let $E,F$ be two finite groups and $H,K$ two one-ended finitely presented groups. When are the lampdesigners $E \boxplus H$ and $F \boxplus K$ quasi-isometric?
\end{question}

\noindent
Theorem~\ref{thm:ClassificationDesigner} shows that, if $H$ and $K$ are amenable, then $E \boxplus H$ and $F \boxplus K$ are quasi-isometric if and only $|E|=|F|$ and $H,K$ are bijectively quasi-isometric. The converse being true without the amenability condition. What does happen when $H$ and $K$ are not amenable?

\begin{question}
Let $\mathfrak{h},\mathfrak{k}$ be two finite fields and $A,B$ two one-ended finitely presented groups. When are the lampcloners $\oslash_\mathfrak{h}A$ and $\oslash_\mathfrak{k}B$ quasi-isometric?
\end{question}

\noindent
Theorem~\ref{thm:Lampcloner} and Lemma~\ref{lem:ClonerKappa} show that, if $\oslash_\mathfrak{h}A$ and $\oslash_\mathfrak{k}B$ are quasi-isometric, then $\mathfrak{h},\mathfrak{k}$ must have the same characteristic and there must exist quasi-$\kappa$-to-one quasi-isometry $A \to B$ for a specific $\kappa$. Does the converse hold? Or is it possible to show that $|\mathfrak{h}|=|\mathfrak{k}|$ if $\oslash_\mathfrak{h}A$ and $\oslash_\mathfrak{k}B$ are quasi-isometric?

\medskip \noindent
The examples of halo groups considered in Sections~\ref{section:Aptolic}--\ref{section:amenable} are essentially all given by large-scale commutative halos of groups. It would be interesting to explore examples given by halos that are not large-scale commutative. For instance:

\begin{question}
Given a one-ended finitely presented group $H$, when are the two nilpotent wreath products $\mathbb{Z}/n\mathbb{Z} \wr^{\mathfrak{n}_p} H$ and $\mathbb{Z}/m \mathbb{Z} \wr^{\mathfrak{n}_q} H$ quasi-isometric?
\end{question}

\noindent
It is worth noticing that, even though we did not mention it, the altitude of a halo product can be used as a quasi-isometric invariant. More precisely, given a leaf-preserving quasi-isometry $\mathscr{M}A \to \mathscr{N}B$ between two finitely generated halo products, if $\mathscr{M}$ has $(\epsilon,\eta,R)$-altitude $N$ for some $\epsilon,\eta,R$, then there exist $\epsilon',\eta',R'$ such that $\mathscr{N}$ has $(\epsilon',\eta',R')$-altitude $N$. This observation can be used, for instance, in order to distinguish lamplighters, lampshufflers, and lampcloners.

\begin{question}
What are some natural examples of halo products of arbitrarily large altitude?
\end{question}

\noindent
It is worth noticing that the altitude estimated in Section~\ref{section:Sufficient} are rather small. 

\medskip \noindent
Finally, recall from \cite{MR4419103} that the \emph{scaling group} $\mathrm{sc}(G)$ of a finitely generated group $G$ is the set of $\kappa>0$ such that there exists a quasi-$\kappa$-to-one quasi-isometry $G \to G$. (It is worth mentioning that $\mathrm{sc}(G)=\mathbb{R}_+^\ast$ for every non-amenable group, so the notion is only interesting for amenable groups.) Given a finitely generated halo group $\mathscr{L}H$, it is not difficult to show that, given an aptolic quasi-isometry $\Psi : (c,p) \mapsto (\alpha(c),\beta(p))$ for some bijection $\alpha : L(H) \to L(H)$ and some quasi-isometry $\beta : H \to H$, if $\beta$ is quasi-$\kappa$-to-one, then so is $\Psi$. Thus, it follows easily from the results proved in Section~\ref{section:amenable} that the following groups:
\begin{itemize}
	\item $F \wr H$ with $F$ finite non-trivial and $H$ amenable satisfying the thick bigon property;
	\item $\circledS_nH$ with $n \geq 1$ and $H$ amenable satisfying the thick bigon property;
	\item $F \boxplus H$ with $F$ finite non-trivial $H$ amenable satisfying the thick bigon property;
	\item $\oslash_\mathfrak{k}A$ with $\mathfrak{k}$ finite and $A$ amenable satisfying the thick bigon property,
\end{itemize}
all have their scaling groups reduced to $\{1\}$. This naturally leads to the following question:

\begin{question}
Let $\mathscr{L}H$ be a finitely generated halo group with $H$ one-ended finitely presented and $L(H)$ locally finite. Is the scaling group $\mathrm{sc}(\mathscr{L}H)$ reduced to $\{1\}$?
\end{question}

\paragraph{Permutational halo products.} Definition~\ref{def:Halo} introduces permutational halo products, but in the rest of the article we only dealt with halo products associated to actions of groups on themselves by left-multiplication. The reason is that, in a permutational halo product $\mathscr{L}_XH$, if some point of $X$ has an infinite $H$-stabiliser, then $H$ is no longer almost malnormal and it follows from Theorem~\ref{thm:Peripheral} that one gets a different large-scale geometry. However, if $H$ acts cofinitely on $X$ with finite stabilisers, most of our techniques can be applied. We decided to avoid such a generality in order to keep the exposition as clear as possible. 

\medskip \noindent
Nevertheless, we emphasize that the general notion of permutational halo products is of independent interest. For instance, we expect them to be a good source of embeddings of ``big groups'' into groups with good finiteness properties, similar (but hopefully more elementary) to the role played by some Thompson-like groups. Mimicking \cite{MR3356831}, we ask:

\begin{question}\label{question:FP}
Let $\mathscr{L}_X H$ be a permutation halo product. Assume that, for every $1 \leq k \leq n$, the induced action of $H$ on $X^k$ has finitely many orbits and stabilisers of type $F_{m-k}$. Under which conditions on $\mathscr{L}_X$ is $\mathscr{L}_XH$ of type $F_n$?
\end{question}

\noindent
Following the examples given in Section~\ref{section:Halo}, sufficient criteria could allow one to construct naturally groups of type $F_\infty$ containing all the mapping class groups, all the special linear groups, or all the automorphism groups of free groups.

\addcontentsline{toc}{section}{References}

\bibliographystyle{alpha}
{\footnotesize\bibliography{HaloGroups}}

\Address

%

\end{document}